\documentclass[12pt]{amsart}
\usepackage{comment}
\usepackage[
backend=biber,
style=alphabetic,
giveninits=true,
maxcitenames=6,
mincitenames=6,
maxbibnames=10,
sorting=nyt
]{biblatex}
\usepackage{amssymb}
\usepackage{amsfonts}
\usepackage{amscd,color}
\usepackage[final]{showkeys}
\usepackage{tikz}
\usepackage{tikz-cd}
\usetikzlibrary{cd}
\usepackage{hyperref}
\usepackage[capitalize,nameinlink]{cleveref}
\usepackage[utf8]{inputenc}
\usepackage{braket}
\usepackage{ytableau}
\usepackage[bb=stix]{mathalpha}

\usepackage[all]{xy}
\usepackage{xypic,epsfig}
\usepackage{newunicodechar}
\newunicodechar{̂}{\^{} }
\definecolor{cardinal}{rgb}{0.77, 0.12, 0.23}
\definecolor{maroon}{rgb}{128, 0, 0}
\definecolor{glasgowblue}{rgb}{0.16, 0.32, 0.75}
\definecolor{applegreen}{rgb}{0.55, 0.71, 0.0}
\newcounter{Todo}

\newcommand{\Z}{\mathbb{Z}}

\newcommand{\C}{\mathbb{C}}

\newcommand{\cS}{\mathcal{S}}
\newcommand{\cQ}{\mathcal{Q}}
\newcommand{\SL}{\,\mathrm{SL}}
\newcommand{\GL}{\,\mathrm{GL}}

\newcommand{\Fl}{\,\mathrm{Fl}}
\newcommand{\Gr}{\,\mathrm{Gr}}

\newcommand{\fS}{\mathfrak{S}}
\newcommand{\End}{\operatorname{End}}

\newcommand{\mb}[1]{\mathbb{#1}}
\newcommand{\mc}[1]{\mathcal{#1}}
\newcommand{\mf}[1]{\mathfrak{#1}}

\newcommand{\bs}[1]{\boldsymbol{#1}}
\newcommand{\op}[1]{\operatorname{#1}}

\newcommand{\veps}{\varepsilon}
\newcommand{\Hpt}{H^*_T(\mathrm{pt})}
\newcommand{\ring}{R}

\newcommand{\ad}{\operatorname{ad}}
\newcommand{\bt}{\mathbb{t}}

\newtheorem{defn}{Definition}[section]
\newtheorem{claim}{Claim}[section]
\newtheorem{remark}{Remark}[section]
\newtheorem{thm}{Theorem}[section]
\newtheorem{prop}{Proposition}[section]
\newtheorem{cor}[prop]{Corollary}
\newtheorem{lemma}[prop]{Lemma}
\newtheorem{example}[prop]{Example}
\newtheorem{conjecture}[prop]{Conjecture}
\crefname{conjecture}{Conjecture}{Conjectures}
\Crefname{conjecture}{Conjecture}{Conjectures}
\numberwithin{equation}{section}
\numberwithin{figure}{section}

\newcommand{\YB}{\operatorname{YB}}
\newcommand{\T}{\mathsf{T}}

\newcommand{\gl}{\mathfrak{gl}}

\title[Equivariant Quantum Cohomology via the Clifford algebra]{Equivariant Quantum Cohomology of Grassmannians via the Clifford algebra}
\author{Christian Korff}
\address{
Christian Korff, School of Mathematics and Statistics,
University of Glasgow,
Glasgow G12 8QQ, United Kingdom
}
\email{Christian.Korff@glasgow.ac.uk}
\urladdr{https://sites.google.com/view/christiankorff}

\author{Mikhail Vasilev}
\address{
Mikhail Vasilev, School of Mathematics and Statistics,
University of Glasgow,
Glasgow G12 8QQ, United Kingdom
}
\email{Mikhail.Vasilev@glasgow.ac.uk}


\date{May 2026}

\begin{document}
\begin{abstract}
    We construct an explicit equivariant quantum Satake map for Grassmannians, which enables us to express their torus-equivariant quantum cohomology in terms of that of projective space. We then consider the exterior algebra of the latter, which admits a canonical identification with a Clifford algebra. We describe the resulting action in several complementary ways: first, from a geometric perspective via push–pull maps, and second, in terms of the shuffle product, which also arises in the simplest cohomological Hall algebra associated with the $A_1$-quiver. Exploiting the Clifford algebra structure, we derive new recurrence relations among equivariant Gromov–Witten invariants, yielding a new method for their computation in terms of Wick's Theorem. As an application, we provide combinatorial proofs of Graham positivity for both equivariant quantum Pieri rules, and in one case extend these results to quantum triple Schubert calculus.
\end{abstract}

\maketitle
\tableofcontents

\section{Introduction}

Quantum cohomology emerged from the foundational works of Gepner \cite{gepner1991fusion}, Intriligator \cite{intriligator1991fusion}, Vafa \cite{vafa1991topological} and Witten \cite{witten1993verlinde} on fusion rings and can be understood as a $q$-deformation of ordinary (intersection) cohomology. Grassmannians were among the first varieties whose quantum cohomology ring $qH^*(\op{Gr}(k,n))$ was computed explicitly \cite{agnihotri1995quantum, bertram1997quantum}. Since then, an extensive body of work has developed around the combinatorics of Schubert calculus and its quantum analogue, with significant impact on adjacent fields such as geometric representation theory. Equivariant quantum Schubert calculus was subsequently introduced in the works by Givental and Kim; see e.g. \cite{givental1995quantum,givental1996equivariant,kim1996equivariant}. Here one has in addition a group action on the underlying topological space leading to an even richer algebraic and combinatorial structure.

More recently, new impetus to the area has come from the works \cite{nekrasov2009quantum,braverman2011quantum,maulik2019quantum} which allows one to interpret the direct sum of quantum cohomologies of Nakajima varieties as quantum group modules using stable envelopes; see  \cite{maulik2019quantum}. The basic example are the rings for cotangent bundles $\T^*\Gr(k;n)$ of Grassmannians $\Gr(k;n)=\Gr(k;\C^n)$, where $\bigoplus_{k=0}^nqH^*(\T^*\Gr(k;n))$ is endowed with the structure of a module for the simplest Yangian $Y_\hbar(\mf{sl}_2)$. Here the deformation parameter $\hbar$ controls the transition from the cotangent bundle down to the base manifold. While the limit from $qH^*(\T^*\Gr(k;n))$ to $qH^*(\Gr(k;n))$ can in principle be taken, one ends up with a very degenerate limit of the quantum group structure: the latter can be described in terms of Yang-Baxter algebras which has been related to the simplest cohomological Hall algebra (CoHA) of the $A_1$-quiver in \cite{GKS20} (see also \cite{gorbounov2017quantum,gorbounov2025quantum} for a discussion of the equivariant quantum K-theory of Grassmannians using Yang-Baxter algebras). In general, it can be challenging to explicitly describe the quantum group structure in this limit down to the base manifold in concrete geometric terms and compare it explicitly to the known combinatorics of quantum Schubert calculus in the literature.

\begin{figure}\label{fig:CliffYoung}
\centering
\includegraphics[width=.95\textwidth]{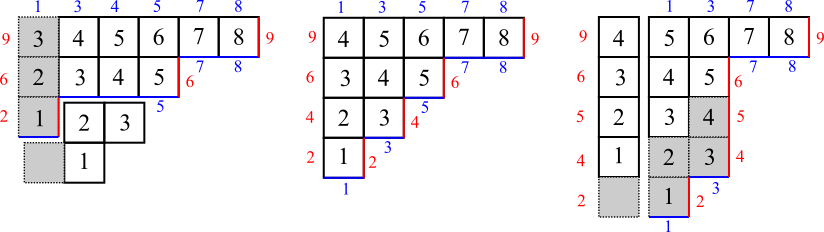} 
\caption{A graphical depiction of the action \eqref{recferm} of the Clifford algebra in terms of Young diagrams labelling Schubert classes. In the middle a Young diagram labelling a Schubert class $\sigma_w$ in $qH^*(\Gr(4;9))$ corresponding to the Grassmannian permutation $w=[2469\,13578]$. Acting with the maps \eqref{recferm} on the corresponding Schubert class $\sigma_w$ gives the classes $\sigma_{w'}\in qH^*(\Gr(3;9))$ (left) and $\sigma_{w''}\in qH^*(\Gr(5;9))$ (right) up to a sign. For example, $\sigma_{w'}=(-1)^{\#(w,w')}\psi_4(\sigma_w)$ is obtained by first adding a column of maximal height to the Young diagram of $w$, then removing a rim-hook of length 4 and, finally, multiplying with $(-1)^{\#(w,w')}$ where $\#(w,w')$ is the height of the rim hook (number of rows) minus one. On the right the computation of $\sigma_{w''}=(-1)^{\#(w'',w)}\psi^*_5(\sigma_w)$ which follows the reverse procedure: add a ribbon of length 5, then remove a column of maximal height and multiply with $(-1)^{\#(w'',w)}$ where $\#(w'',w)$ is the height of the added rim hook minus one. If the outcome of either procedure is not a Young diagram then the image is zero. 
}
\end{figure}

\subsection{A geometric construction of the Clifford algebra}
Therefore, we shall take a different approach in this article and directly identify the sum $\bigoplus_{k=0}^nqH_T^*(\Gr(k;n))$ of {\em torus equivariant quantum cohomology rings of Grassmannians} for fixed $n\ge 2$, as a cyclic module $\mc{F}$ of a (finite) {\em Clifford algebra}. Namely, we consider the associative unital algebra  $\mc{C}_n=\langle\psi_i,\psi_i^*~:~i=1,\ldots,n\rangle/\sim$ over $\ring[q^{\pm 1}]\simeq qH^*_T[\mathrm{pt}]=\Z[q^{\pm 1}]\otimes_\Z\Z[y_1,\ldots,y_n]$ subject to the relations
\[
\psi_i\psi_j+\psi_j\psi_i=\psi_i^*\psi_j^*+\psi_j^*\psi_i^*=0,\qquad
\psi_i\psi^*_j+\psi_j^*\psi_i=\delta_{ij}\;.
\]
These identities are known as the {\em canonical anti-commutation relations} (CAR) and describe in quantum physics {\em free fermions}. For this reason we will sometimes refer to the $\psi_i^*$ as {\em fermionic creation} and to the $\psi_i$ as {\em fermionic annihilation operators}. The irreducible module $\mc{F}\simeq \bigoplus_{k=0}^nqH_T^*(\Gr(k;n))$ is called the (finite-dimensional) {\em fermionic Fock space}.

Let us describe the Clifford algebra action on $\mc{F}$ geometrically building on the earlier works \cite{GKS20,korffstroppel2010}: consider the following natural projections from the 2-step flag variety $\Fl(k,k+1;n)$ to Grassmannians,
\begin{equation}\label{pushpull}
    \begin{tikzcd}
        & \Fl(k, k +1; n) \arrow[ld, "p_1"] \arrow[rd, "p_2"] & \\
        \Gr(k; n) & & \Gr(k + 1; n)\;.
    \end{tikzcd}
\end{equation}
Then for each $0\le k<n$ and $i=1,\ldots, n$ we define maps 
$$\psi^*_i:qH^*_T(\Gr(k;n))\to qH^*_T(\Gr(k+1;n)),\quad\psi_i:qH^*_T(\Gr(k+1;n))\to qH^*_T(\Gr(k;n))$$ as follows: for $i=1$ and $i=n$ we respectively set
\begin{equation}\label{initial}
 \psi^*_1 =(p_2)_{*} p_1^*, \qquad 
 \psi_n =(-1)^{k} (p_1)_{*} p_2^*\;.
\end{equation}
The action of the remaining Clifford generators is defined recursively via
\begin{gather}\label{recferm}
     \psi^*_{i + 1} = c_1(\cS_{k + 1}) \circ \psi^*_i - \psi^*_i \circ c_1(\cS_k) - y_{i} \circ \psi^{*}_i,\\ \psi_{i - 1} = \psi_i \circ c_1(\cS_{k+1}) - c_1(\cS_{k}) \circ \psi_i - y_i \circ \psi_i,\notag
\end{gather}
where $c_1(\mc{S}_k)$ denotes the (quantum) multiplication with the first Chern class of the tautological bundle over the Grassmannian $\Gr(k;n)$ and $y_i$ is the multiplication with an equivariant parameter related to the torus action.\footnote{The maps in \eqref{initial} coincide with the ones considered in \cite{GKS20} when describing the geometric action of Yang-Baxter algebras and we shall comment below on the precise connection. The maps defined in \eqref{recferm} on the other hand were not considered in \cite{GKS20}. Their algebraic analogues in the non-equivariant setting were discussed in \cite{korffstroppel2010}.} 

Since the definition of the action of the Clifford algebra via \eqref{recferm} is somewhat indirect, let us briefly describe it explicitly in terms of Schubert classes using the combinatorics of Young diagrams. Recall that the Schubert classes in $qH^*_T(\Gr(k;n))$ are labelled in terms of Grassmannian permutations $w\in S^{k,n}\subset S_n$ which are in bijection with partitions $\lambda(w)$ whose Young diagram fits into the $k\times (n-k)$ rectangle and for which $|\lambda(w)|=\ell(w)$, the length of the permutation. The action of the maps \eqref{recferm} on Schubert classes then corresponds to adding (removing) a column of maximal height and removing (adding) a rim-hook (up to a sign factor); see Figure \ref{fig:CliffYoung} for an example. 

\subsection{Shuffle product and the CoHA of $A_1$}
In Section \ref{sec:shuffle} we shall explicitly relate the action of the  fermionic creation operators $\psi_i^*:qH^*_T(\Gr(k;n))\to qH^*_T(\Gr(k+1;n))$ to the shuffle product of \cite{feigin1995vector} which describes the graded multiplication in the simplest CoHA of the $A_1$-quiver; see \cite{kontsevich2011cohomological} for details. We note that previously the connection of equivariant cohomology of Grassmannians to the CoHA was discussed in \cite[Section 8]{GKS20} in terms of Yang-Baxter algebras. Our description in terms of fermions and the shuffle product gives a new combinatorial description of this connection which extends to the quantum case. We also give a combinatorial description of the annihilation operators $\psi_i:qH^*_T(\Gr(k;n))\to qH^*_T(\Gr(k-1;n))$ using the Newton interpolation formula; see \eqref{annihshuffle} and \eqref{annihshuffle2}. In particular, we show that both, creation and annihilation operators, are related via the commutative diagram 
\begin{equation}\label{LRcd}
    \begin{tikzcd}
         qH^*_T(\Gr(k-1;n)\arrow[d,"\Theta_{k-1}"]  &\arrow[l,"(-1)^{n-k+1-i}\psi_i"] qH^*_T(\Gr(k;n)) \arrow[d,"\Theta_k"] \arrow[r, "(-1)^{n-k-i}\psi^*_i"] &  qH^*_T(\Gr(k +1; n)) \arrow[d, "\Theta_{k+1}"]  \\
        qH^*_T(\Gr(n-k+1;n)&\arrow[l,"\psi^*_{n+1-i}"]  qH^*_T(\Gr(n-k;n)) \arrow[r, "\psi_{n+1-i}"] &  qH^*_T(\Gr(n-k-1; n))
    \end{tikzcd}
    \end{equation}
by employing a ring isomorphism $\Theta_k:qH^*_T(\Gr(k;n))\to qH^*_T(\Gr(n-k;n))$ discussed in \cite[Cor. 6.8]{gorbounovkorff2014} (and for quantum K-theory in \cite[Cor. 4.14]{gorbounov2017quantum}) for which we give a new combinatorial proof; see Theorem \ref{thm:LRduality} in the main text.

\subsection{New recurrence formulae for Gromov-Witten invariants}
The maps \eqref{recferm} allow for a new combinatorics in quantum Schubert calculus which relates the quantum multiplication in $qH^*_T(\Gr(k;n))$ to either $qH^*_T(\Gr(k-1;n))$ or $qH^*_T(\Gr(k+1;n))$ and one arrives at recurrence relations for equivariant Gromov-Witten invariants of Grassmannians of different dimensions.\footnote{Previously established recurrence formulae and algorithms such as \cite[Prop.5.1 and Thm 2]{MihalceaAIM} are based on relations between invariants in the {\em same} ring.} 

Applying this recurrence repeatedly one then deduces a novel expression for the equivariant Gromov-Witten invariants purely in terms of the Clifford algebra maps \eqref{recferm}. Before we can state this result we need to briefly recall two basic facts related to the Clifford algebra (further details are given in the main body of the article):
\begin{itemize}
\item[(i)] Given a group morphism $\alpha:G\to\GL_n(\ring)$ there is a natural $G$-action on the Clifford algebra via automorphisms, $g.\psi_i=\sum_{j}\psi_j (g^{-1})_{ji}$ and $g.\psi^*_i=\sum_{j}g_{ij}\psi^*_j$. In our case we will exploit this for $G=\Z^n\rtimes S_n$ the extended affine Weyl group and the above group morphism $\alpha$ is prescribed by a $G$-action on the quantum cohomology of projective space $qH^*_T(\mb{P}^{n-1})$; see Lemma \ref{lem:Waffaction}.
\item[(ii)] The Clifford algebra comes equipped with a natural functional $\langle\;\bullet\;\rangle:\mc{C}_n\to\ring[q^{\pm 1}]$ called the {\em vacuum expectation value} which can be computed via a determinant formula called `Wick's Theorem',
$\langle \psi_{i_k}\ldots\psi_{i_1}\psi^*_{j_1}\ldots\psi^*_{j_k}\rangle=\det(\langle\psi_{i_a}\psi^*_{j_b}\rangle)_{1\le a,b\le k}$ with $\langle\psi_{i_a}\psi^*_{j_b}\rangle=\delta_{i_aj_b}$. 
\end{itemize}
Recall that a skew diagram $\rho=\lambda/\mu$ of two partitions $\lambda,\mu$ is called a {\em horizontal strip} if the difference of their Young diagram contains at most one box in each column. Moreover, the content $c(\square)$ of a box $\square\in\rho$ is the difference of its column and row number in the Young diagram of $\lambda$. 
\begin{thm}
    Let $u=[u_1\ldots u_n],v=[v_1\ldots v_n],w=[w_1\ldots w_n]\in S^{k,n}$ be Grassmannian permutations and set $\lambda_i=u_{k+1-i}+i-k-1$ for $i=1,\ldots,k$. Then the equivariant Gromov-Witten invariants, the structure constants of $qH^*_T(\Gr(k;n))$, are given by 
\begin{equation}\label{intro:EGW=VEV}
C_{uv}^{w}(y,q)=\sum_{T(\lambda)}
\langle\psi_{w_k}\cdots\psi_{w_1}(t^{\rho_1}.\psi^{*}_{v_1}) \cdots(t^{\rho_k}.\psi^*_{v_k})\rangle
\end{equation}
where the sum runs over all semi-standard tableaux $T$ of shape $\lambda$, a partitioning of the Young diagram of $\lambda$ into horizontal strips $\rho_i$, each of which defines a translation in the normal subgroup $\Z^n\simeq\langle t_1,\ldots t_n\rangle\le G$ via  $t^{\rho_i}=\prod_{\square\in\rho_i}t_{k+c(\square)}$.   
\end{thm}
To the best of the authors' knowledge analogous expressions for the quantum invariants of higher flag varieties or their cotangent bundles are currently not known.

We note that the idea of obtaining such formulae via the Clifford algebra goes back to the work \cite[Section 11]{korffstroppel2010}, where recurrence formulae for the non-equivariant cohomology ring $qH^*(\Gr(k;n))$ were obtained; see also \cite{Korff09}. The new result here is their geometric description in terms of push-pull maps and the extension to equivariant quantum Schubert calculus which allows one to relate these formulae to the action of the extended affine Weyl group $W=\Z^n\rtimes S_n$ on $qH^*_T(\mb{P}^{n-1})$ from Lemma \ref{lem:Waffaction}. Key in making this connection is an explicit quantum version of the equivariant Satake correspondence for Grassmannians.

\subsection{An equivariant quantum Satake map for Grassmannians}
The main idea we are going to employ in our derivation of \eqref{intro:EGW=VEV} and the underlying Clifford algebra structure is to describe the ring $qH_T^*(\op{Gr}(k,n))$ in terms of the (much simpler) quantum cohomology ring $qH_T^*(\mb{P}^{n-1})$ of projective space $\mb{P}^{n-1}=\op{Gr}(1,n)$; see e.g. \cite{hori2000mirror,bertram2005two,kim2008quantum,golyshev2011quantum} in the non-equivariant setting. 

Let us briefly recall the construction. It follows from the geometric Satake construction  \cite{ginzburg1995perverse,mirkovic2007geometric} that when identifying the cohomology of $\Gr(k;n)$ as a minuscule Schubert cell in the affine Grassmannian of $\GL_n(\C)$ the latter corresponds to the $k$-fold exterior power of the cohomology of $\mb{P}^{n-1}$ under the same identification. In particular, the Satake correspondence identifies $H^*(\Gr(k;n))$ with the $\gl_n(\C)$-module $\bigwedge^kV$ and the multiplication by the first Chern class corresponds to the action by the principal nilpotent element of $\gl_n(\C)$. This picture has been extended to quantum cohomology in \cite{ciocan2008abelian,golyshev2011quantum} by replacing the principal nilpotent element with the cyclic element in $\gl_n(\mb{C})$, which then describes the quantum Pieri rule in $qH^*(\op{Gr}(k,n))$ provided one sets $q=1$, i.e. one considers the so-called Verlinde algebra.

The construction of the Satake map extends to equivariant (non-quantum) Schubert calculus by considering the exterior $\ring$-algebra $\bigwedge V$ with $V\simeq H^*_T(\mb{P}^{n-1})$ as $\gl_n(\ring)$-modules; see e.g. \cite{martin2000symplectic,Laksov09,anderson2023equivariant,labelle2024equivariant} as well as references therein. In this article we state a quantum version of the equivariant Satake map for Grassmannians; see also \cite{cotti2025satake,cheng2025quantum} for related constructions albeit with a rather different focus on quantum difference equations. 

Our construction will use instead `quantised Chern roots' $X^q_1,\ldots,X^q_k\in\End_{\ring} V^{\otimes k}[q^{\pm 1}]$, a set of mutually commuting $n\times n$ matrices with values in $\ring[q^{\pm 1}]$ each satisfying the matrix identity
\begin{equation}\label{intro:BAE}
    \prod_{j=1}^n(X_i^q-y_j\cdot1)=(-1)^{k-1}q\cdot 1\;,\qquad i=1,\ldots,k\;.
\end{equation}
Apart from a sign factor in the deformation parameter $q$, the latter identity is the defining relation of the quantum cohomology of projective space and we refer to the matrices $X^q_i$ as `quantised Chern roots'. This terminology is further justified by the following result. If we consider the double Schubert polynomial $\fS_w(X_1^q,\ldots,X^q_k|y)\in\End_{\ring} V^{\otimes k}[q^{\pm 1}]$ with $w\in S^{k,n}$ a Grassmannian permutation then the latter commutes with the canonical projection $V^{\otimes k}[q^{\pm 1}]\twoheadrightarrow\bigwedge^kV[q^{\pm 1}]$ and we prove the following:
\begin{prop}\label{intro:quantumSatake}
    The $\ring[q^{\pm1}]$-module $\bigwedge^kV[q^{\pm1}]$ together with the bilinear operation
    \[
    \ket{w}\star\ket{u}:=\fS_w(X_1^q,\ldots,X^q_k|y)\ket{u},\quad 
    \ket{u}=v_{u_1}\wedge\ldots\wedge v_{u_k},
    \]
    where $u,w\in S^{k,n}$ is a pair of Grassmannian permutations and $\{v_1,\ldots,v_n\}\subset V$ the standard basis of the free module $V\simeq \ring^n$, is a well-defined commutative associative unital graded ring, which is isomorphic to $qH^*_T(\Gr(k;n))$. Moreover, the isomorphism is canonical, i.e. the basis vectors $\ket{w}=v_{w_1}\wedge\ldots\wedge v_{w_k}$ of $\bigwedge^kV[q^{\pm 1}]$ map to Schubert classes.
\end{prop}

The key ingredient which we will use in proving the above is Mihalcea's result that the ring structure is fixed, once we can prove the quantum Chevalley rule holds \cite[Cor.7.1]{MihalceaAIM}. As an immediate consequence we obtain an alternative expression for the equivariant Gromov-Witten invariants.
\begin{cor}
    Let $u,v,w\in S^{k,n}$ be Grassmannian permutations then
    \[
    C^w_{uv}(y,q)=\bra{w}\fS_u(X_1^q,\ldots,X_n^q|y)\ket{v}\;,
    \]
    where the right hand side is computed via a natural pairing $\bigwedge^kV^*[q^{\pm 1}]\otimes\bigwedge^kV[q^{\pm 1}]\to\ring[q^{\pm 1}]$ using Poincar\'e duality and $\bra{w}=v^{w_1}\wedge\ldots\wedge v^{w_k}$ corresponds to the dual Schubert basis.
\end{cor}
The last identity leads to another combinatorial algorithm involving semi-standard tableaux: recall that for Grassmannian permutations the double Schubert polynomial specialises to a factorial Schur polynomial and the latter have a presentation as a weighted sum over semi-standard tableaux; see Appendix \ref{app:A}. Each term in this sum acts by simple matrix multiplication on $V^{\otimes k}[q^{\pm 1}]$ and the final outcome of the computation can then simply be projected onto $\bigwedge^kV[q^{\pm 1}]$. This is a generalisation of the construction from \cite[Section 4.2 and 5.7]{korff2020cylindric} for non-equivariant quantum cohomology.

From the realisation of equivariant quantum cohomology via the quantum Satake map one is then naturally led to consider the Clifford algebra $\mc{C}_n(\ring)$: the latter is simply the exterior algebra $\bigwedge V[q^{\pm 1}]$ and the geometric maps defined in \eqref{recferm} correspond to multiplication and contraction operators. In other words, the quantum Satake map extends to an isomorphism of Clifford algebra modules.

\subsection{Applications: Graham positivity and quantum triple Schubert calculus}
As an application of our formalism we give new combinatorial proofs of Graham positivity for the quantum Pieri rules in $qH^*_T(\Gr(k;n))$; see Section \ref{sec:Pieri}. There is a known result originally due to Mihalcea \cite{MihalceaAIM} that the structure constants $C^{\nu,d}_{\lambda\mu}(y)$ of the equivariant quantum cohomology ring $qH^*_T(\Gr(k;n))$ satisfy {\em Graham positivity}, i.e. the $C^{\nu,d}_{\lambda\mu}(y)$ are non-negative integer polynomials in the torus weights $y_{i+1}-y_i$. A combinatorial proof of this statement or a manifestly positive combinatorial algorithm to compute the equivariant Gromov-Witten invariants is currently not known. More recently, Bertiger et al. introduced in \cite{BEMT22} a statistics on Postnikov's \cite{Post05} cylindric tableaux which allows for a manifestly positive computation of quantum Pieri rules and we directly compare our formula against theirs in Section \ref{sec:cylloops}. 

We conclude by extending our construction to triple Schubert calculus and suggest a quantum version of it; see Section \ref{sec:tripleSchubert}. Recall that in \cite{MS99} Molev and Sagan considered the problem of multiplying two factorial Schur polynomials with different sets of equivariant parameters
\begin{equation}\label{MolevSagan}
    s_{\lambda}(x|z) s_{\mu}(x|y) = \sum\limits_{\nu} c_{\lambda, \mu}^{\nu}(y,z) s_{\nu}(x|y)
\end{equation}
and derived a combinatorial formula for the coefficients $c_{\lambda,\mu}^{\nu}(y,z)$, which is manifestly positive in the following sense: $c_{\lambda,\mu}^{\nu}(y,z) \in \Z_{\ge 0}[y_i - z_j]$. Knutson and Tao then gave a geometric meaning for computing the expansion \eqref{MolevSagan} in terms of so-called {\em triple Schubert calculus}; see Section 6 of \cite{KnutTao2003} for details. In their setting of triple Schubert calculus the parameters $y$ and $z$ in the product \eqref{MolevSagan} are the generators of the ring $H^*(\Fl_n \times \Fl_n) \simeq H^*(\Fl_n) \otimes H^*(\Fl_n)$, namely $y_i$ is the first Chern class of the $i$-th tautological line bundle over the first copy of $\Fl_n$ and $z_j$ is the first Chern class of the $j$-th tautological line bundle over the second copy of $\Fl_n$. Recently, the positivity of triple Schubert calculus for arbitrary partial flags was proved by geometric methods in \cite{gao2025graham}.

Prompted by these results in the literature, we consider an extension of our quantum Satake map from Proposition \ref{intro:quantumSatake} to the base ring $\Z[y_1,\ldots,y_n,z_1,\ldots,z_n]$ and define a triple quantum multiplication by replacing the $x$-variables in \eqref{MolevSagan} with our quantised Chern roots satisfying \eqref{intro:BAE}. We conjecture that the resulting ring has structure constants that also satisfy a refined version of Graham positivity and prove a manifestly positive quantum Pieri rule in support of our conjecture.

\subsection{The connection with lattice models and Yang-Baxter algebras}
For completeness, let us briefly outline the connection between our construction and the previous work on equivariant quantum cohomology of Grassmannians using Yang-Baxter algebras in \cite{gorbounovkorff2014} and \cite{GKS20}; a detailed discussion can be found in Appendix \ref{app:C}.

Yang-Baxter algebras are unital associative bi-algebras whose defining relations are encoded in a matrix equation known as the $RTT$-equation,
\[
R_{12}(x_1 - x_2)T_1(x_1) T_2(x_2) = T_2(x_2) T_1(x_1) R_{12}(x_1 - x_2)\;.
\]
Historically, this is how quantum groups, such as Yangians, were first introduced by the Faddeev School and this method is now known as the Faddeev-Reshetikhin-Takhtajan (FRT)-construction of quantum groups \cite{faddeev1988quantization,faddeev1990lectures}. Here $R$ denotes a given solution of the quantum Yang-Baxter equation (the data defining the algebra) and $T$ is called the {\em monodromy matrix} whose elements give the generators of the algebra. 

In the case of the quantum cohomology of Grassmannians (as opposed to that of their cotangent bundles) the Yang-Baxter algebra $\YB$ is not the Yangian $Y_\hbar(\mf{sl}_2)$, but instead the corresponding $R$-matrix describes a so-called exactly solvable five-vertex model in statistical mechanics \cite{korff2014quantum,gorbounovkorff2014}, it has been geometrically constructed in \cite{GKS20} including the monodromy matrix
\begin{equation} \label{monodromy}
    T(x) = \begin{pmatrix}
        T_{11}(x) & T_{12}(x) \\
        T_{21}(x) & T_{22}(x)
    \end{pmatrix}\;,
\end{equation}
where the entries are currents whose coefficients are convolution operators: namely, the `quantum trace' of the monodromy matrix, $t(x)= T_{11}(x) + q T_{22}(x)$, where $q$ is the deformation parameter of the quantum cohomology ring, acts on $\mc{F}\simeq \bigoplus_{k=0}^nqH_T^*(\Gr(k;n))$ by quantum multiplication with the Chern polynomial of the tautological bundle on each individual summand, while the action of the off-diagonal elements $T_{12}(x),T_{21}(x)$ involves push and pull maps of the type \eqref{initial} and, thus, the latter map between different quantum cohomology rings. 

The relation between the Yang-Baxter algebra construction from \cite{GKS20} and our construction using the quantum Satake map from Prop \ref{intro:quantumSatake} is best summarised in terms of the following quantum version of a Cauchy-type identity for factorial Schur functions, which is a new result; compare with their non-equivariant quantum versions in \cite[Lemma 5.1 and Cor.5.40]{korff2020cylindric}:
\begin{prop}
    Let $t(x)\in\End_{\ring}(qH^*_T(\Gr(k;n))$ be the quantum multiplication by the Chern polynomial of the tautological bundle $\mc{S}_k$. Then
    \[
    t(x_1)\cdots t(x_{n-k})=\sum_{w\in S^{k,n}}\fS_w(X_1^q,\ldots,X^q_k|y)\fS_{w^*}(x_1,\ldots,x_{n-k}|-y),
    \]
    where $w^* = w \cdot (12\ldots n)^{k}\in S^{n-k,n}$ is a Grassmannian permutation and the $X^q_i$ are mutually commuting matrices satisfying \eqref{intro:BAE}.
\end{prop}
Besides the above relation we give explicit formulae for all the Yang-Baxter algebra generators $T_{ij}(x)$ in terms of the maps \eqref{initial} and \eqref{recferm} in Appendix \ref{app:C} of this work. In particular, we show that the image $\mc{Y}$ of $\YB$ in $\End \bigoplus_{k=0}^nqH_T^*(\Gr(k;n))$, as well as its `dual' $\mc{Y}'$ under the ring isomorphism $qH^*_T(\Gr(k;n)\simeq qH^*_T(\Gr(n-k;n)$ from Theorem \ref{thm:LRduality} are both generated by the Clifford algebra and coincide; see Theorem \ref{fermtheor2}.

This shows that the Clifford algebra is a more fundamental algebraic structure underpinning the equivariant quantum cohomology of Grassmannians albeit the formulation in terms of Yang-Baxter algebras leads to a nice graphical calculus of lattice models which greatly simplifies some combinatorial descriptions. 

Finally, we point out that the matrix identity \eqref{intro:BAE} when rewritten in terms of eigenvalues of the matrices $X_i^q$ are nothing else but the Bethe ansatz equations of the underlying lattice model. The latter are needed in the Bethe ansatz or quantum inverse scattering method construction of the idempotents for each of the quantum cohomology rings; see \cite{korff2014quantum,gorbounovkorff2014,GKS20} for details.

\bigskip
\noindent{\bf Acknowledgments.} C.K. gratefully acknowledges the sharing of knowledge and expertise with Vassily Gorbounov, Leonardo Mihalcea and Catharina Stroppel. M.V.  is grateful to the University of Glasgow for financial support in form of a Postgraduate Scholarship of the College of Science and Engineering.

\section{Preliminaries}
Most of the results in this section can be found in the literature but we collect them here to introduce our notation and keep this article self-contained. A new result is presented in Section \ref{sec:pushpull} which states the geometric action of the Clifford algebra for equivariant (non-quantum) cohomology which we then extend to the quantum case later in the article.

\subsection{Equivariant quantum cohomology of Grassmannians}
Denote by $\Gr(k;n)$ the Grassmannian of $k$ dimensional hyperplanes in $\C^n$. There is a natural left $\GL_n(\C)$ action on $\Gr(k;n)$. Fix a maximal torus $T\cong(\C^\times)^n$ in $\GL_n(\C)$, the subset of $n\times n$ invertible diagonal matrices with entries $y_1,\ldots,y_n$ which are called the {\em equivariant parameters}. We denote by
\[
\ring=\Z[y_1,\ldots,y_n]\cong\Hpt
\]
the equivariant cohomology of a point. Rather than recalling the geometric definition of the equivariant quantum cohomology ring $qH^*_T(\Gr(k;n))$ we shall work instead with its presentation as quotient ring using factorial Schur polynomials; see Appendix \ref{app:A}. However, we give a dictionary between the algebraic and geometric definitions.

Consider the following generating function for factorial versions $e_r(x|y)$ of the elementary symmetric polynomials,
\[
\prod_{i=1}^k(z-x_i)=\sum_{r=0}^k(-1)^re_r(x|y)\;(z|y)^{k-r},\qquad(z|y)^r:=\prod_{i=1}^r(z-y_i),
\]
where $(z|y)^r$ are called the {\em falling factorial powers}. Geometrically, the above generating function corresponds to the Chern polynomial $z^kc_{-z^{-1}}(\cS_k)$ of the {\em tautological bundle} $\cS_k$ fixed by the short exact sequence $0\to \cS_k\to\C^n\to \cQ_{n-k}\to 0$, where $\cQ_{n-k}$ denotes the {\em quotient bundle}. The variables $x_1,\cdots,x_k$ are the first $k$ {\em Chern roots}. The following result is due to Mihalcea \cite[Thm 1.1]{Mihalcea}.
\begin{thm}[Mihalcea]
    The equivariant quantum cohomology ring $qH^*_T(\Gr(k;n))$ has the presentation,
    \begin{equation}\label{QHdef}
    qH^*_T(\Gr(k;n)\cong\ring[q][e_1,\ldots,e_k]/\langle h_{k+1},\ldots,h_{n-1},h_n+(-1)^kq\rangle
    \end{equation}
    where $q$ is the quantum deformation parameter and $h_r=\det(\tau^{j-1}e_{1-i+j})_{1\le i,j\le r}$ are the Chern classes of the quotient bundle with $\tau$ denoting the shift operator; see Appendix \ref{app:A}.
\end{thm}
Furthermore, we recall that $qH^*_T(\Gr(k;n))$ has a distinguished additive $\ring[q]$-basis, the {\em Schubert classes} $\sigma_\lambda$, where $\lambda$ runs over all partitions $\lambda=(\lambda_1,\ldots,\lambda_k)$ whose Young diagram fits into a $k\times (n-k)$ rectangle, i.e. $\lambda_1\le n-k$. We shall denote this set by $\mc{P}_{k,n}$. Under the above isomorphism the Schubert classes are sent to the factorial Schur functions $s_\lambda=\det(\tau^{j-1}e_{\lambda'_i-i+j})_{1\le i,j\le k}$ with $\lambda'$ denoting the conjugate partition. The latter are specialisations of Lascoux and Sch\"utzenberger's double Schubert polynomials to Grassmannian permutations; see Appendix \ref{app:A} for details.

The main aim of this article is to give new combinatorial formulae for the computation of the structure constants $C_{\lambda\mu}^{\nu}(y,q)$ of the ring \eqref{QHdef} in the Schubert basis defined via the expansion (we shall denote quantum multiplication by $\star$ throughout for clarity)
\begin{equation}\label{EGWdef}
\sigma_\lambda\star\sigma_\mu=\sum_{\nu}C^{\nu}_{\lambda\mu}(y,q)\sigma_\nu=
\sum_{\nu}q^dC^{\nu,d}_{\lambda\mu}(y)\sigma_\nu,
\end{equation}
where the coefficients are called the {\em equivariant Gromov-Witten invariants}. 
\subsection{The Satake correspondence for projective space}
Let $V=\ring^n$ be the free $\ring$-module with $n\ge 2$ arbitrary but fixed and denote by $\{v_1,\ldots,v_n\}\subset V$ the standard basis. We use the notation $V^*$ for its dual with basis $\{v^1,\ldots,v^n\}$.

As $\ring$-modules we have the isomorphism $\Sigma_1:V\to H^*_T(\mb{P}^{n-1})$ where 
\[
H^*_T(\mb{P}^{n-1})\cong\ring[x]/\langle (x-y_1)\cdots(x-y_n)\rangle
\]
and we identify the standard basis in $V$ with the Schubert basis in $H^*_T(\mb{P}^{n-1})$ via
\begin{equation}\label{Satake1}
\Sigma_1:\quad v_{i+1}\mapsto\sigma_{i}=(x-y_1)\cdots(x-y_i),\qquad \forall i=0,1,\ldots,n-1\;.
\end{equation}
We now turn the latter module isomorphism into an isomorphism of $\ring$-algebras: denote by $\{E_{ij}\in \gl_n(\ring)~:~1\le i,j\le n\}$ the standard basis, the unit matrices, and define the following sum of a nilpotent and a (trivially) semi-simple element,
\begin{equation}\label{X0}
X=X_{\rm n}+X_{\rm s}=\sum_{i=1}^{n-1}E_{i+1,i}+\sum_{i=1}^ny_iE_{ii}\;.
\end{equation}
Define an $\ring$-algebra structure on $V$ by setting $v_i\odot v_j=(X-y_1\,1)\cdots(X-y_i\,1).v_j$.
\begin{prop}
    The $\ring$-algebra $(V,\odot)$ is associative and commutative. Moreover, the map \eqref{Satake1} is an isomorphism of $\ring$-algebras $V\cong H^*_T(\mb{P}^{n-1})$ and 
    \[
    \Sigma_1\circ(1+X)=c(\mc{S})\circ\Sigma_1\;,
    \]
    where $c(\mc{S})$ is the total Chern class of the tautological bundle $\cS=\cS_1$.
\end{prop}
The proof uses the following lemma.
\begin{lemma}\label{lemX0}
    For $1\le m<n$ we have that
    \[
    (X-y_1\cdot 1)\cdots(X-y_m\cdot1)=\sum_{i-j=k}E_{ij}+\sum_{0\le i-j< k}p_{ij}^{(m)}(y)E_{ij}
    \]
    where the $p_{ij}^{(m)}(y)$ are either zero if $j>i$ or homogeneous polynomials in the equivariant parameters of degree $m-i+j$ with
    \[
        p_{ij}^{(m)}(y) = h_{m + j - i}(y_j, \ldots, y_i|y).
    \]
    In particular, we have for $k=n$ the matrix identity:
    \begin{gather}\label{BAE0}
        (X-y_1\cdot 1)\ldots(X-y_n\cdot 1)=0\;.
    \end{gather}
\end{lemma}

\begin{proof}
We prove the first statement by induction in the factorial power $m$. The base of induction is trivial. Thus, we are left to complete the induction step. Using the induction assumption we have
\begin{align*}
    \;& (X-y_1) \cdots (X-y_m)(X-y_{m+1})v_j = (X-y_{m + 1})\sum\limits_{i = j}^{j + m}h_{m + j - i}(y_j, \ldots, y_i|y) v_i \\
    \;&= \sum\limits_{i = j}^{j + m}h_{m + j - i}(y_j, \ldots, y_i|y) v_{i + 1} + \sum\limits_{i = j}^{j + m} (y_i - y_{m + 1}) h_{m + j - i}(y_j, \ldots, y_i|y) v_i \\
    \;&= \sum\limits_{i = j + 1}^{j + m + 1}h_{m + j + 1 - i}(y_j, \ldots, y_{i - 1}|y) v_i + \sum\limits_{i = j}^{j + m} (y_i - y_{m + 1}) h_{m + j - i}(y_j, \ldots, y_i|y) v_i \\
    \;&= \sum\limits_{i = j}^{j+m + 1} (h_{m + j + 1 - i}(y_j, \ldots, y_{i - 1}|y) +(y_i - y_{m + 1}) h_{m + j - i}(y_j, \ldots, y_i|y))v_i \\
    \;&= \sum\limits_{i = j}^{j + m + 1}h_{m+1+j-i}(y_j, \ldots, y_i|y)v_i,
\end{align*}
where in the last equality we have applied the branching rule for complete factorial symmetric polynomials; see Appendix \ref{app:A}.
\end{proof}
\subsection{The left Weyl group action}
Recall that there is a natural $\GL_n(\C)$-action on $\Gr(k;n)$ which induces a left Weyl group action on $H^*_T(\Gr(k;n))$,
\begin{equation}\label{LeftW}
(s_i,\sigma_w)\mapsto\left\{\begin{array}{ll}
\sigma_w+(y_i-y_{i+1})\sigma_{s_iw},& \ell(s_iw)=\ell(w)-1\;\&\;s_i w\in S^{k,n}\\
\sigma_w,&\text{ else}
\end{array}
\right.\;,
\end{equation}
where $i=1,2,\ldots,n-1$ and $S^{k,n}$ is the set of minimal length coset representatives of $S_n/S_k\times S_{n-k}$. Note that there is also a natural $S_n$-action on $\ring$ (permuting the equivariant parameters) and since $H^*_T(\Gr(k,n))$ is an $\ring$-module it is tacitly understood that the action \eqref{LeftW} is compatible with the ring action, i.e. $s_i.(r\sigma)=(s_i.r)(s_i.\sigma)$ for all $r\in\ring$ and $\sigma\in H^*_T(\Gr(k;n))$. In other words, $H^*_T(\Gr(k;n))$ is a module for the skew group ring $\ring\#S_n$ and this will be understood throughout the article including for our discussion of the quantum ring.

In light of the isomorphism \eqref{Satake1}, we wish to describe the above left Weyl group action for $k=1$ in terms of the Lie group $\GL_n(\ring)$ in order to then extend it to the tensor products of $\ring$-modules, $V^{\otimes k}$ and $\bigwedge^kV$, below.
\begin{lemma}\label{lem:Waction}
The map $(s_i,v_j)\mapsto g_i.v_j$ with 
\begin{equation}\label{g_i}
g_i:=\exp(\ad_{X_{\rm s}}(E_{i,i+1}))\in\GL_n(\ring),\qquad i=1,\ldots,n-1
\end{equation}
and $j=1,\ldots,n$ defines a group action of $S_n$ on $V$ and the map \eqref{Satake1} is an $\ring\#S_n$-module isomorphism with respect to the action defined via \eqref{g_i} on $H^*_T(\mb{P}^{n-1})$.
\end{lemma}

\begin{proof}
If we specialise the left Weyl group action \eqref{LeftW} to $H^*_T(\mb{P}^{n-1})$ by setting $k=1$ we arrive at
\begin{equation}\label{LeftW1}
(s_i,\sigma_j)\mapsto\left\{\begin{array}{ll}
\sigma_j+(y_j-y_{j+1})\sigma_{j+1},& j=i\\
\sigma_j,&\text{ else}
\end{array}
\right.\;.
\end{equation}
On the other hand, we have that $\ad_{X_{\rm s}}(E_{i,i+1})=(y_i-y_{i+1})E_{i,i+1}$ and hence that
\[
\exp(\ad_{X_{\rm s}}(E_{i,i+1}))=1+(y_i-y_{i+1})E_{i,i+1}\;.
\]
From this explicit form of the action on $V$ and \eqref{Satake1} the claim is now obvious.
\end{proof}

Having identified the Lie algebra element $X\in\gl_n(\ring)$ in \eqref{X0} with the first Chern class of the tautological bundle via \eqref{Satake1} it must be invariant under the left Weyl group action.
\begin{lemma}\label{lem:Winvariance}
    Let $g_i\in\GL_n(\ring)$ for $i=1,\ldots,n-1$ be defined as in \eqref{g_i}. Then we have the commutation relations
    \begin{equation}\label{Winvariance}
    g_i^{-1}Xg_i=X^{s_i},\qquad i=1,\ldots,n-1,
    \end{equation}
    where $X^{s_i}\in\gl_n(\ring)$ denotes the element \eqref{X0} with the equivariant parameters $y_i$ and $y_{i+1}$ exchanged.
\end{lemma}

\begin{proof}
    A simple case of matrix multiplication.
\end{proof}
\subsection{The Satake correspondence for Grassmannians}
We remind the reader that under the Satake correspondence one identifies $H^*(\op{Gr}(k,n))$ with the $k$-fold exterior power of the cohomology of $\mb{P}^{n-1}$. In particular, the correspondence is an isomorphism of $\gl_n(\C)$-modules and the multiplication by the first Chern class corresponds to the action by the principal nilpotent element of $\gl_n(\C)$. Here we recall the extension of the Satake correspondence to the torus-equivariant case; see for example \cite{anderson2023equivariant,labelle2024equivariant} and references therein.

We start by identifying the direct sum of all cohomology rings with the exterior $\ring$-algebra of $V\simeq H^*_T(\mb{P}^{n-1})$ as $\ring$-modules: namely, we consider the $\ring$-module isomorphism
\begin{equation}\label{Satake_dom}
\Sigma:\bigwedge V=\bigoplus_{k=0}^n
\bigwedge\nolimits^kV\to \bigoplus_{k=0}^nH^*_T(\Gr(k;n))
\end{equation}
where for any subset $I=\{i_1,\ldots,i_k\}\subset\{1,\ldots,n\}$ that is fixed by a Grassmannian permutation $w=[i_1\ldots  i_n]\in S^{k,n}$ we set
    \begin{equation}\label{Satake}
    \Sigma:\quad \ket{w}=v_{i_1}\wedge\ldots\wedge v_{i_k}\mapsto \sigma_w
    \end{equation}
with $k=0,1,\ldots,n$. Here we have used that for each fixed $0\le k\le n$ one can alternatively label Schubert classes by Grassmannian permutations as the latter are in bijection with partitions whose Young diagrams fit into the $k\times(n-k)$ rectangle. The bijection is given by $w\mapsto\lambda(w)=(\lambda_1,\ldots,\lambda_k)$, where $\lambda(w)$ is the partition with parts $i_j=\lambda_{k+1-j} + j$. We will use both types of labels interchangeably throughout.

\begin{example}\rm
Consider $k=6$, $n=10$ and $\lambda=(4,4,2,2,1,0)$. The corresponding Grassmannian permutation $w(\lambda)$ is recovered by drawing a lattice path prescribing the outline of the Young diagram of $\lambda$, where we record the positions $\{i_1,\ldots,i_k\}$ of each vertical and the positions $\{i_{k+1},\ldots,i_n\}$ of each horizontal move in terms of black and a white-go stone, respectively. The resulting Grassmannian permutation in one-line notation is then given by the positions of the go-stones as $w(\lambda)=[i_1\ldots i_ki_{k+1} \ldots i_n]$; see Figure \ref{fig:Young}. Alternatively, one can obtain $w(\lambda)$ as a reduced expression in terms of simple reflections by reading out the content of the boxes in the Young diagram bottom to top and right to left. In our example,
\[
w(\lambda)=s_8s_9\,s_7s_8\,s_4s_5s_6s_7\,s_2s_3s_4s_5s_6\;.
\]
We will occasionally also need the image of Schubert classes under level-rank duality which in terms of Young diagrams is the conjugate partition $\lambda'$ obtained by swapping black and white go-stones and reading their positions from right to left; see Figure \ref{fig:Young}.
\end{example}

Recall that there is natural action of the universal enveloping algebra $U$ of $\gl_n(\ring)$ on $V^{\otimes k}$ via the coproduct,
\[
\Delta^{k-1}(X)=X_1+\cdots +X_k,\quad X\in\gl_n(\ring)
\]
where $X_i=1\otimes\cdots1\otimes X\otimes1\cdots \otimes1$ acts non-trivially in the $i$th factor. This action descends to $\bigwedge^kV$ because of Schur-Weyl duality. The following result for Grassmannians can be found in the literature; see e.g. \cite{anderson2023equivariant,labelle2024equivariant} and references therein.
\begin{thm}
    The restriction $\Sigma_k=\Sigma|_{\bigwedge^kV}$ of the map \eqref{Satake} to $\bigwedge^kV$ is an isomorphism of $\ring$-algebras $\bigwedge^kV\cong H^*_T(\Gr(k;n))$ via
    \[
    \Sigma_k\circ \Delta^{k-1}(X)=c_1(\cS_k)\circ\Sigma_k\;,
    \]
    where $X\in\gl_n(\ring)$ is defined in \eqref{X0} and $c_1(\cS_k)$ is the first Chern class of the tautological bundle $\cS_k$. As the latter generates the ring $H^*_T(\Gr(k;n))$, this fixes the isomorphism completely.
\end{thm}

Let $\mb{X}_k=(\mb{P}^{n-1})^{\times k}$. In what follows we will make sometimes use of the fact that we can canonically identify the tensor product $V^{\otimes k}$ with the equivariant cohomology $H^*_T(\mb{X}_k)\cong H_T^*(\mb{P}^{n-1})^{\otimes k}$ (as $\ring$-algebras) and that multiplication by symmetric polynomials in the Chern roots `commutes' with the natural projection $V^{\otimes k}\twoheadrightarrow\bigwedge^kV$ in the following sense: let $x_1,\ldots,x_k$ denote the $k$ Chern roots of the $k$ copies of $\mb{P}^{n-1}$, then multiplication with any elementary symmetric polynomial $e_r(x_1,\ldots,x_k)$ in $H^*_T(\mb{X}_k)$ corresponds to multiplication by the $r$th Chern class $c_r(\cS_k)$ of the tautological bundle in $H^*_T(\Gr(k;n))$.
\begin{cor}
    Under the isomorphism \eqref{Satake} we have the identities
    \begin{equation}
    \Sigma_k\circ e_r(X_1,\ldots,X_k)=c_r(\cS_k)\circ\Sigma_k,
    \qquad r=0,1,\ldots,k
    \end{equation}
    and 
    \begin{equation}
    \Sigma_k\circ \fS_w(X_1,\ldots,X_k;y)=\sigma_w\circ\Sigma_k,\qquad
    \forall w\in S^{k,n},
    \end{equation}
    where $\fS_w(x;y)$ is the double Schubert polynomial and the classes on the right hand side act by multiplication in $H^*_T(\Gr(k;n))$.
\end{cor}
In light of the first identity it is natural to define the analogue of the Chern polynomial.
\begin{lemma}\label{lem:transfermatrix}
Fix $1\le k\le n-1$ and let $u$ be an indeterminate. Then
\[
t(u)=(u\cdot 1+X_1)\cdots (u\cdot 1+X_k)\in\gl_n(\ring)\otimes_\ring\ring[u]
\]
commutes with the left Weyl group action on $\bigwedge^kV$ and, thus, setting $t_i=t(-y_i)\in\End_{\ring} \bigwedge^kV$ we have
\begin{equation}
s_it_i=t_{i+1}s_i\qquad\text{and}\qquad s_it_j=t_js_i\;\text{ if }|i-j|>1\;.
\end{equation}
Here by abuse of notation we have denoted the image of the $s_i$ in $\End \bigwedge^kV$ by the same symbol. We note that the $t_i$ are singular and that $t_1t_2\ldots t_n=0$.
\end{lemma}
\begin{proof}
    The claims follow directly from our previous Lemmata \ref{lemX0} and  \ref{lem:Winvariance} as well as Schur-Weyl duality.
\end{proof}
\subsection{The exterior and the Clifford algebra} Consider the exterior $\ring$-algebra $\bigwedge V$ from \eqref{Satake}. The latter can be identified with the Clifford algebra $\mc{C}_n(\ring)$ over $V$ with trivial quadratic form. Namely, 
denote by $\psi_i^* = v_i\wedge$ the multiplication operator $\bigwedge^kV\to\bigwedge^{k+1}V$ with $v_i\in V$ and by $\psi_i = \imath_{v^i}$ the contraction operator $\bigwedge^kV\to\bigwedge^{k-1}V$ with $v^i\in V^*$,
\begin{equation}\label{contraction}
\imath_{v^i}(v_{j_1}\wedge\cdots\wedge v_{j_k})=\sum_{r=1}^k(-1)^{r-1}\delta_{ij_r}
v_{j_1}\wedge\cdots\wedge v_{j_{r-1}}\wedge v_{j_{r+1}}\cdots\wedge v_{j_k}\;.
\end{equation}
It is a standard fact that these operators satisfy the canonical anti-commutation relations (CAR) of free fermions,
\begin{equation}\label{CAR}
\psi_i\psi_j+\psi_j\psi_i=\psi_i^*\psi_j^*+\psi_j^*\psi_i^*=0,\qquad
\psi_i\psi^*_j+\psi_j^*\psi_i=\delta_{ij}\;.
\end{equation}
As already pointed out in the introduction, the natural action of $\GL_n(\ring)$ on $V$ induces automorphisms of the Clifford algebra via
\begin{equation}\label{GLaction}
g.\psi^*_i:=\sum_j g_{ij}\psi_j^*\qquad\text{and}\qquad g.\psi_i:=\sum_j\psi_j(g^{-1})_{ji},\qquad\forall g\in\GL_n(\ring)\;.
\end{equation}
One verifies via a straightforward computation that $g.\psi_i,g.\psi_i^*$ still satisfy the CAR. In particular, any action of a group $G$ on $V$ induces via its representation map $G\to \GL_n(\ring)$ a group action via automorphisms of the Clifford algebra, $g.\psi^*_i=\psi^*_{g.v_i}$ and $g.\psi_i=\imath_{(g^{-1})^T.v^i}$. We now apply this to the case of the left Weyl group action with $G=S_n$ using Lemma \ref{lem:Waction}. Note that the left Weyl group action naturally extends to an action on $\bigwedge V$ under Schur-Weyl duality. The next lemma describes how this action on $\bigwedge V$ interacts with that of the Clifford algebra.
\begin{lemma}\label{lem:WClifford}
    We have the commutation relations
    \[
    s_i\circ\psi^*_j=(g_i.\psi_j^*)\circ s_i
    \quad\text{and}\quad
    s_i\circ\psi_j=(g_i.\psi_j)\circ s_i
    \]
    where the matrices $g_i\in\GL_n(\ring)$ with $i=1,\ldots,n-1$ are defined in \eqref{g_i}. Explicitly,
    \[
    g_i.\psi^*_j=\psi^*_j+\delta_{i+1,j}(y_i-y_{i+1})\psi^*_{i}
    \quad\text{and}\quad
    g_i.\psi_j=\psi_j-\delta_{i,j}(y_i-y_{i+1})\psi_{i+1}\;.
    \]
\end{lemma}
\begin{proof}
    A straightforward computation using \eqref{GLaction} and the definition \eqref{g_i}.
\end{proof}

\subsection{Cyclic modules of the Clifford algebra: the Fock space and its dual}
We recall the construction of some fundamental Clifford algebra modules as these might not be familiar to a wider mathematical audience. For this we now consider the (finite) Clifford algebra $\mc{C}_n(\ring)$ as abstract $\ring$-algebra generated by $\{\psi_i,\psi_i^*~:~i=1,\ldots,n\}$ subject to the CAR \eqref{CAR}.

Define the following free $\ring$-modules,
\begin{gather}
W_n^-=\bigoplus_{k=1}^n\ring \psi_k,\qquad W_n^+=\bigoplus^n_{k=1}\ring\psi^*_k,\qquad
 W_n=W_n^-\oplus W_n^+
\end{gather}
and set
\begin{equation}
    \mc{F}_n=\mc{C}_n/\mc{C}_n W^-\qquad\text{and}\qquad
    \mc{F}^*_n=\mc{C}_nW^+\backslash\mc{C}_n\;.
\end{equation}
The latter spaces are natural left and right $\mc{C}_n$-modules generated by the cyclic vectors
\begin{equation}
\ket{0}=1\!\!\mod \mc{C}_n W^-
\qquad\text{and}\qquad
\bra{0}=1\!\!\mod W^+\mc{C}_n\;,
\end{equation}
respectively. The cyclic vectors satisfy the relations
\[
\psi_i\ket{0}=0,\qquad\bra{0}\psi_i^*=0,\qquad i=1,\ldots,n\;.
\]
We call $\mc{F}_n$ the {\em fermionic Fock space} and $\mc{F}^*_n$ its dual, since there exists a symmetric bilinear form $\mc{F}^*_n\otimes_{\mc{C}_n}\mc{F}_n\to\ring$ called the {\em vacuum expectation value} (VEV) denoted by
\[
\tilde F\otimes F\mapsto \bra{0}\tilde FF\ket{0}=\langle \tilde F F\rangle
\]
The latter functional is defined via the following relations:
\begin{gather}
\langle 1\rangle=1,\quad\langle\psi_i\psi_j\rangle=\langle\psi_i^*\psi^*_j\rangle=\langle\psi^*_i\psi_j\rangle=0,\quad\langle\psi_i\psi^*_j\rangle=\delta_{ij}
\end{gather}
and for $F_i\in\{\psi_i,\psi^*_i\}$ we set
\[
\langle F_1\ldots F_r\rangle=\left\{\begin{array}{ll}
    \sum\limits_{w}(-1)^w\langle F_{w_1}F_{w_2}\rangle\cdots\langle F_{w_{r-1}}F_{w_r}\rangle, &\;r\;\text{ even}  \\
    0, &\;r\;\text{ odd} 
\end{array}
\right.\;,
\]
where the sum runs over all permutations $w=[w_1\ldots w_r]\in S_r$ such that $w_i<w_{i+1}$ for $i=1,\ldots,r-1$ and $w_1<w_3<\ldots <w_{r-2}<w_r$. This relation is known as {\em Wick's Theorem} in the physics literature, where it is used to compute density-density correlation functions for free fermion gases.

\begin{prop}
    Sending $\mc{F}_n\ni\ket{0}\mapsto 1\in\ring$ provides an isomorphism of (left) $\mc{C}_n(\ring)$-modules.
\end{prop}
\begin{proof}
Since $\ket{0}$ is cyclic, it suffices to show that the map $\psi^*_{i_1}\ldots\psi^*_{i_k}\ket{0}\mapsto v_{i_1}\wedge\ldots\wedge v_{i_k}$ for any Grassmannian permutation $w=[i_1\ldots i_n]\in S^{k,n}$ with $k=0,1,\ldots,n$ intertwines the action of the Clifford algebra. But this is obvious from the definitions.
\end{proof}

\begin{remark}\rm
    By similar means one shows that $(\bigwedge V)^*\simeq\bigwedge V^*$ has the structure of a {\em right} $\mc{C}_n(\ring)$-module isomorphic to the dual Fock module $\mc{F}^*$. The Clifford generators naturally induce maps $\psi_i:\bigwedge^kV^*\to\bigwedge^{k+1}V^*$ and $\psi^*_i:\bigwedge^kV^*\to\bigwedge^{k-1}V^*$. Moreover, for any $w=[i_1\ldots i_n]\in S^{k,n}$ one can identify $\bra{0}\psi_{i_k}\cdots\psi_{i_1}\in\mc{F^*}$ with the dual Schubert class $\sigma^w$ satisfying $\langle\sigma^u,\sigma_w\rangle=\delta_{uw}$ with respect to the Poincar\'e bilinear form.
\end{remark}
\subsection{Push forward and pull back maps}\label{sec:pushpull}
The next result describes the action of the Clifford algebra on the equivariant cohomology of Grassmannians under the Satake map; to the best knowledge of the authors this geometric description is new. Recall the diagram \eqref{pushpull} from the introduction, 
where $p_1$ and $p_2$ are the natural maps that project a flag $F=(F_1\subset F_2\subset\C^n)\in\Fl(k,k+1;n)$ onto its first and second element, respectively. Denote by $\cS_k$ and $\cS_{k + 1}$ the tautological bundles over the Grassmannians $\Gr(k;n)$ and $\Gr(k + 1;n)$ respectively and by $\hat\cS_k=p_1^* (\cS_k)$ and $\hat\cS_{k+1}=p_2^*(\cS_{k + 1})$ their pull backs to the two-step flag variety $\Fl(k, k+1; n)$. 
\begin{prop}\label{fermionform}
     We have the following identification of maps under the $\ring$-module isomorphism $\Sigma=\bigoplus_{0\le k\le n}\Sigma_k$ from \eqref{Satake_dom}: 
    \begin{gather}\label{fermgeom}
       \Sigma_{k+1}\circ\psi^*_{l+1}= (p_2)_{*}\circ\prod_{i=1}^l \,(c_1( \hat\cS_{k + 1}/\hat\cS_k )-y_i)\circ p_1^*\circ\Sigma_k\\
       \Sigma_{k-1}\circ \psi_{n - l}=(-1)^{k - 1}(p_1)_* \circ\prod_{i=1}^l (c_1(\hat\cS_{k }/\hat\cS_{k - 1} )-y_{n+1-i})\circ p_2^*\circ\Sigma_k ,
    \end{gather}
    where $0\le k,l\le n$ and we recall (for any partial flag variety) the identity $c_1(\hat\cS_{k+1}/\hat\cS_k)=x_{k+1}$, the $(k+1)$th Chern root.
\end{prop}
N.B. the formulae \eqref{fermgeom} are nothing else but the solution of the geometric recurrence relations \eqref{recferm} with initial conditions \ref{initial} from the introduction. From this recursive definition it is immediate that $\psi_{i}^* = \psi_{n +1 - i} = 0$ for $i > n$. By abuse of notation we shall henceforth use the same symbols to denote the algebraically and geometrically defined operators $\psi_i,\psi_i^*$.

\begin{proof}
    Let us prove the first formula in \eqref{fermgeom}, the second one is proved in the same manner. We proceed by induction, the base of induction follows from \cite[Prop.4.6]{GKS20}, where equivariant localisation was used to compute the matrix elements of the push-pull operators $(p_1)_* p_2^*$ and $(p_2)_* p_1^*$.  
    Namely, we have according to \eqref{initial} that
    \[
    \psi^*_1(\sigma_\lambda)=\left\{\begin{array}{ll}
    \sigma_{(\lambda_1-1,\ldots,\lambda_k-1)},&\ell(\lambda)=k\\
    0,&\text{ else}
    \end{array}
    \right.\;,\quad
    \psi_n(\sigma_\lambda)=\left\{\begin{array}{ll}
    (-1)^{k-1}\sigma_{\mu},&\lambda_1=n-k\\
    0,&\text{ else}
    \end{array}
    \right.\;,
    \]
    where $\lambda/\mu$ is a maximal rim hook of length $n$, which matches with the action on $\bigwedge^kV$ under the map \eqref{Satake}. 
    
    Let us assume that \eqref{fermgeom} holds true for $l \leq n-2$, then using formula \eqref{recferm} for any $v \in \bigwedge^{k}V$ we have
    \begin{multline*}
    (\Sigma^{-1}_{k + 1} \circ c_1(\cS_{k + 1}) \circ \Sigma_{k + 1}  \circ \psi^*_{l + 1} - \psi^*_{l + 1} \circ \Sigma_k^{-1} \circ c_1(\cS_k) \circ \Sigma_k - y_{l + 1} \circ \psi^{*}_{l + 1})v
    \\
    = (X_1 + \ldots X_{k + 1})(v_{l + 1} \wedge v) - v_{l + 1} \wedge (X_1 + \ldots X_k)v - y_{l + 1} v_{l + 1} \wedge v.
    \end{multline*}
    Let us lift the last equality to the tensor product $V^{\otimes (k + 1)}$ resulting in
    \begin{multline*}
    (X_1 + \ldots X_{k + 1})(v_{l + 1} \otimes \widetilde{v})
    - v_{l + 1} \otimes (X_1 + \ldots X_k) \widetilde{v} - y_{l + 1} v_{l + 1} \otimes \widetilde{v} 
    \\
    = (X v_{l + 1}) \otimes \widetilde{v} - y_{l + 1} v_{l + 1} \otimes \widetilde{v} = (v_{l + 2} + y_{l + 1} v_{l + 1}) \otimes \widetilde{v} - y_{l + 1} v_{l + 1} \otimes \widetilde{v} = v_{l + 2} \otimes \widetilde{v},
    \end{multline*}
    where $\widetilde{v} \in V^{\otimes k}$ is a lift of $v \in \bigwedge^k V$. Projecting back onto $\bigwedge^kV$ we obtain
    \[
    (\Sigma^{-1}_{k + 1} \circ c_1(\cS_{k + 1}) \circ \Sigma_{k + 1}  \circ \psi^*_{l + 1} - \psi^*_{l + 1} \circ \Sigma_k^{-1} \circ c_1(\cS_k) \circ \Sigma_k - y_{l + 1} \circ \psi^{*}_{l + 1})v = v_{l + 2} \wedge v,
    \]
    which finishes the proof of the statement.
\end{proof}

\section{Equivariant quantum Schubert calculus via the Clifford algebra}
We now extend our construction to the quantum case. There are two ways of motivating the definition of the quantum ring $qH^*(\Gr(k;n))$ from an algebraic/combinatorial perspective -- both are of course interlinked:
\begin{enumerate}
    \item Consider the loop algebra $\gl_n[z,z^{-1}]=\gl_n(\ring)\otimes_\ring\ring[z,z^{-1}]$ and its natural action on $V[z,z^{-1}]=V\otimes_\ring\ring[z,z^{-1}]$. Define a `quantised' Chern root by considering the element
    \begin{equation}\label{X}
    X^q=X^q_{\rm n}+X_{\rm s}=z\,E_{1,n}+\sum_{i=1}^{n-1}E_{i+1,i}+\sum_{i=1}^ny_iE_{ii}\;,
    \end{equation}
    where the loop parameter $z$ is up to a sign the quantum parameter $q$. As we will show below the matrix \eqref{X} now satisfies the defining relations of $qH^*_T(\mb{P}^{n-1})$. Thus, when extending the Satake map \eqref{Satake}, the elementary symmetric polynomials $e_r(X^q_1,\ldots,X^q_k)$ now describe the {\em quantum} multiplication by the tautological Chern classes $c_r(\cS_k)$ in $qH^*(\Gr(k;n))$ and the double Schubert polynomial $\fS_w(X_1^q,\ldots,X_k^q;y)$ the {\em quantum} multiplication by a Schubert class.
    \item Extend the left Weyl group action from Lemma \ref{lem:Waction} to an action of the extended affine Weyl group $W=\Z^n\rtimes S_n$, where the translations $t_i=t_{\epsilon_i}$ generating the normal subgroup $\Z^n$ act on $\bigwedge^kV\otimes_\ring\ring[z,z^{-1}]$ by
    \[
    t_i=t(-y_i)=(X^q_1-y_i\cdot 1)\cdots (X_k^q-y_i\cdot 1),\quad i=1,\ldots,n\;;
    \]
    compare with Lemma \ref{lem:transfermatrix}. In particular, the affine Weyl reflection $s_0$ acts according to
    \begin{equation}\label{s0}
    s_0\mapsto g_0=\exp\left(z^{-1}\op{ad}_{X_{\rm s}}E_{n,1}\right)
    =1+z^{-1}(y_n-y_1)E_{n,1}
    \end{equation}
    and $g^{-1}_0.X^q.g_0=(X^q)^{s_0}$ where $(X^q)^{s_0}$ denotes the matrix $X^q$ with the equivariant parameters $y_1$ and $y_n$ swapped; compare with Lemma \ref{lem:Winvariance}. In particular, we will show via Formula \eqref{fermionformula} that the action of the normal subgroup $\Z^n\le W$ for projective space $\mb{P}^{n-1}=\Gr(1;n)$ `generates' the quantum multiplication in $qH^*_T(\Gr(k;n)$ together with the action of the Clifford algebra.
\end{enumerate}

Because both of these algebraic descriptions of the quantum multiplication can be expressed in terms of the Clifford algebra, this leads to a recurrence formula of the quantum product in $qH^*(\Gr(k;n))$ in terms of push-pull maps.

\subsection{Quantised Chern roots}
Set $V[z^{\pm 1}]=V\otimes_\ring\ring[z^{\pm 1}]$. The matrix $X^q\in\gl_n[z^{\pm 1}]$ defined in \eqref{X} reads explicitly,
\begin{equation}\label{Xmatrix}
    X^q=\begin{pmatrix}
    y_1 & 0 &\cdots &0 & z\\
    1& y_{2} &0 && 0\\
    0 &\ddots& \ddots &\ddots &\vdots\\
    \vdots &&1& y_{n-1}&0\\
    0& \cdots&0& 1 & y_n
    \end{pmatrix}\;.
\end{equation}
We shall identify the latter with the quantised Chern root of projective space because of the following Lemma.
\begin{lemma} \label{lem:BAE}
For $1\le m \leq n$ we have that
    \begin{multline*}
        (X^q-y_1\cdot 1)\cdots(X^q-y_m\cdot1)=\\
        z \delta_{m,n}\cdot 1 + \sum_{i-j=m}E_{ij}+z\sum_{i-j=m-n}E_{ij}+
    \sum_{0\le i-j<m}p^{(m)}_{ij}(y)E_{ij}+z\sum_{i-j<m-n}p^{(m)}_{ij}(y)E_{ij}
    \end{multline*}
    where the $p^{(m)}_{ij}(y)$ are either zero or homogeneous polynomials of degree $m-i+j$ if $i>j$ and of degree $m-n-i+j$ if $i<j$, concretely
    \[
        p_{ij}^{(m)} = \left\{
        \begin{array}{ll}
            h_{m + j - i}(y_j,\ldots,y_i|y), \quad i \geq j \\
            h_{m - n - i + j}(y_j, \ldots, y_n, y_1, \ldots, y_i|y), \quad i < j. 
        \end{array}
        \right.
    \]
    For $m=n$ we have
    \begin{equation}\label{GrassBAE}
    \prod_{j=1}^n(X^q-y_j\cdot1)=z\cdot 1\;.
\end{equation}
In particular, $\det(X^q-y_j\cdot1)=z$ and, thus, each $(X^q-y_j\cdot 1)\in\GL_n(\ring[z^{\pm1}])$ for all $j=1,\ldots,n$.
\end{lemma}
\begin{proof}
    The proof follows the same lines as for Lemma \ref{lemX0}, and thus we omit it.  
\end{proof}

\subsection{The equivariant quantum Satake map for Grassmannians}
In close analogy with the non-quantum case we now define the quantum Satake map starting with projective space. Recall that
\[
qH^*_T(\mb{P}^{n-1})\cong\ring[q^{\pm 1}][x]/\langle(x-y_1)\cdots(x-y_n)-q\rangle
\]
where we have assumed the quantum parameter to be invertible in order to define the action of the extended affine Weyl group below.

Define an isomorphism $\Sigma^q_1:V[z^{\pm 1}]/\langle z-q\rangle\to qH^*_T(\mb{P}^{n-1})$ of $\ring[q^{\pm1}]$-modules by sending the standard basis $\{v_i~:~1\le i\le n\}$ in $V[z^{\pm 1}]/\langle z-q\rangle$ to the Schubert basis in $qH^*_T(\mb{P}^{n-1})$ as before,
\begin{equation}\label{qSatake1}
\Sigma^q_1:\quad v_{i + 1}\mapsto \sigma_{i}=(x-y_1)\cdots(x-y_i),\qquad \forall i=0,1,\ldots,n-1\;.
\end{equation}
However, we now define a different multiplication map on $V[z^{\pm 1}]/\langle z-q\rangle$ setting instead $v_i\star v_j=(X^q-y_1\,1)\cdots(X^q-y_{i - 1}\,1).v_j$, where $X^q$ is the matrix defined above.
\begin{prop}
    The $\ring[q^{\pm 1}]$-algebra $(V[z^{\pm 1}]/\langle z-q\rangle,\star)$ is associative and commutative. Moreover, the map \eqref{qSatake1} is an isomorphism of $\ring[q^{\pm1}]$-algebras and 
    \[
    \Sigma^q_1\circ(1+X^q)=c(\mc{S})\circ\Sigma^q_1\;,
    \]
    where $c(\mc{S})$ now denotes the {\em quantum} multiplication operator by the total Chern class of the tautological bundle $\cS=\cS_1$.
\end{prop}
\begin{proof}
    The proof is an immediate consequence of the matrix identity \eqref{GrassBAE}.
\end{proof}
Since $X^q\in\gl_n(\ring[z^{\pm1}])$ there is a natural action of the symmetric polynomials in the quantised Chern roots
\[
X_i^q=1\otimes\cdots\otimes1\otimes\underset{i}{X^q}\otimes1\otimes\cdots\otimes1,\quad i=1,\ldots,k
\]
on the $\ring[z^{\pm1}]$-module $V^{\otimes k}\otimes_\ring\ring[z^{\pm1}]\cong (V)^{\otimes k}[z^{\pm 1}]$ using the standard coproduct formula,
\[
\Delta^{k-1}(X^q)=X_1^q+\cdots +X^q_k,
\]
where $\Delta$ is an $R[z^{\pm1}]$-algebra morphism.
\begin{prop}\label{prop:quantumSatake}
    Define an isomorphism $\Sigma^q_k:\bigwedge^kV[z^{\pm1}]/\langle z+(-1)^kq\rangle\to qH^*_T(\Gr(k;n))$ of $\ring[q^{\pm1}]$-modules by naturally extending the map \eqref{Satake} over the base $\ring[q^{\pm 1}]$.  
    Then this map becomes an isomorphism of $\ring[q^{\pm1}]$-algebras if we in addition set
    \begin{equation}
    \Sigma^q_k\circ \fS_w(X^q_1,\ldots,X^q_k|y)=\sigma_w\circ\Sigma^q_k,\qquad
    \forall w\in S^{k,n},
    \end{equation}
    where $\fS_w(x;y)$ is the double Schubert polynomial and the Schubert classes $\sigma_w$ on the right hand side act by {\em quantum} multiplication in $qH^*_T(\Gr(k;n))$.
\end{prop}
We split the proof of this proposition into several steps. First we prove the following, see \cite[Thrm 1.2]{Laksov09}:
\begin{claim}
    Let $\fS_{w}(X^q_1, \ldots, X^q_k|y)$ be the double Schubert polynomial in the variables $X^q_i$, then for any Grassmannian permutation $w\in S^{k,n}$ we have
    \begin{equation}\label{action}
        \fS_{w}(X^q_1, \dots, X^q_n | y) \ket{e} = \ket{w}\;,
    \end{equation}
    where $e\in S^{k,n}$ is the identity permutation.
\end{claim}

\begin{proof}
We start by observing that $\ket{e}$ can be written in the following convenient form provided we consider $\bigwedge^kV[z^{\pm 1}]$ as a submodule of $V^{\otimes k}[z^{\pm1}]$,
\[
\ket{e} = v_1 \wedge \ldots\wedge v_k = (-1)^{\frac{k (k - 1)}{2}}\Delta(X^q) v_1 \otimes \ldots \otimes v_1,
\]
where $\Delta(X^q) = \prod\limits_{i<j}^k(X^q_i - X^q_j)$ is the Vandermonde determinant. Recall that for Grassmannian permutations $w$ we have that $\fS_w(x|y)=s_{\lambda(w)}(x|y)$, the factorial Schur polynomial; see Appendix \ref{app:A}. Using the alternant formula for the factorial Schur polynomial we find
\begin{align*}
s_{\lambda}(X^q_1, \ldots , X^q_k |y)\ket{e} =\;& (-1)^{\frac{k(k-1)}{2}}\underset{k \times k}{\det}((X^q_j |y)^{\lambda_i + k -i})v_1 \otimes \ldots\otimes v_1 \\
=\;& (-1)^{\frac{k(k-1)}{2}} \sum\limits_{\sigma \in S_k}(-1)^{\sigma} \prod\limits_{l = 1}^k (X^q_j|y)^{\lambda_{\sigma(j)} + k - \sigma(j)} v_1\otimes \ldots\otimes v_1\\ 
=\;& v_{\lambda_k + 1} \wedge \ldots\wedge v_{\lambda_1 + k} = \ket{\lambda}\;.
\end{align*}
Here we have used in the last step that all terms involving the equivariant parameters cancel.
\end{proof}
Introduce the following $\Z$-grading on $\bigwedge^kV[z^{\pm1}]$: we set $\deg y_i=1$, $\deg\ket\lambda=|\lambda|$ and $\deg z=n$. Note that this the natural grading inherited from $qH^*_T(\Gr(k;n))$ under the linear isomorphism \eqref{Satake}. %
Next we show that the ring structure on $\bigwedge^kV[z^{\pm1}]$ is well-defined:
\begin{claim}
    The pair $(\bigwedge^kV[z^{\pm1}],\star)$ with product
    \begin{equation}\label{star_product}
    \ket{\lambda}\star \ket{\nu}=s_{\lambda}(X^q_1, \ldots, X^q_k|y)\ket{\nu},  
    \end{equation}
    is a well-defined graded commutative associative ring with multiplicative identity $\ket{\emptyset}=\ket{e}=v_1\wedge v_2\wedge\ldots\wedge v_k$.
\end{claim}
\begin{proof}
 We start by noting that the factorial Schur polynomials $s_\lambda(X^q|y)$ and $s_\mu(X^q|y)$ commute in $\End_{\ring} V[z^{\pm1}]^{\otimes k}$, simply because the quantised Chern roots $X_i^q$ mutually commute. But as the factorial Schur polynomials are symmetric in the $X_i^q$ their restriction to the antisymmetric tensor product must also form an abelian family of operators in $\End_{\ring} \bigwedge^kV[z^{\pm 1}]$.

To show that this implies commutativity of the ring structure $\star$, we now use formula \eqref{action}, which immediately gives
    \begin{align*}
    \ket{\lambda} \star \ket{\mu} = s_{\lambda}(X^q|y) \ket{\mu} =\;& s_{\lambda}(X^q|y) s_{\mu}(X^q|y)\ket{\emptyset}\\=\;&
    s_{\mu}(X^q|y) s_{\lambda}(X^q|y)\ket{\emptyset} = s_{\mu}(X^q|y) \ket{\lambda} =\ket{\mu} \star \ket{\lambda}.
    \end{align*}
    Associativity follows from associativity of matrix multiplication and $\ket{e}$ is a multiplicative identity because $s_{\emptyset}(X^q|y) = 1$.
    
    Finally, to show that this is a graded ring, we observe that each monomial $\prod_{\Box\in\lambda}(X^q_{T(\Box)}-y_{T(\Box)+c(\Box)})$ acting on $\ket{\mu}$ increases the degree by $|\lambda|$ since multiplication by $X^q$ either (1) shifts $v_{i_{T(\Box)}}$ to $v_{i_{T(\Box)}+1}$ (which amounts to adding a box), or (2) multiplies with $y_{i_{T(\Box)}}$, or (3) removes a maximal rim hook (by shifting $i_k=n$ to $1$) of length $n-1$ and multiplies with $z=(-1)^{k-1}q$.
\end{proof}
As a last step in the proof of Prop. \ref{prop:quantumSatake} it remains to show that both graded ring structures coincide. Using a result of Mihalcea \cite[Cor. 7.1]{MihalceaAIM} it suffices to show that the quantum Chevalley formula from $qH^*_T(\Gr(k;n))$ holds; this identifies the ring up to isomorphism (together with the natural grading inherited from $qH^*_T(\Gr(k;n))$).
\begin{claim}\label{claim:quantumChevalley}
The quantum Chevalley formula holds in $(\bigwedge^kV[z^{\pm 1}]/\langle z+(-1)^kq\rangle,\star)$.
\end{claim}
\begin{proof}
    Noting that
    \[
    s_\Box(X^q|y)\ket{\lambda}=e_1(X^q|y)\ket{\lambda}=(X_1^q+\cdots+X_k^q-y_1-\cdots-y_k)\ket{\lambda}
    \]
    it suffices to consider the action of $X_r^q$ on a state $\ket{\lambda}$. Using the bijection between Grassmannian permutations $w=[i_1\ldots i_n]\in S^{k,n}$ and partitions $\lambda$ with $i_j=\lambda_{k+1-j}+j$ one easily verifies for $1\le r<k$ that either $$X_r^q\ket{\lambda}=\ket{\lambda^{(r)}}+y_{i_r}\ket{\lambda}$$
    where $\lambda^{(r)}$ is the partition obtained by adding to the Young diagram of $\lambda$ a single box with content $i_r-k$, or if that is not possible ($i_r+1=i_{r+1}$), then $X_r^q\ket{\lambda}=y_{i_r}\ket{\lambda}$. The case $r=k$ is special: in this case $X_r^q\ket{\lambda}=q\ket{\lambda^{-}}+y_{i_k}\ket{\lambda}$ where $\lambda^{-}$ is the partition obtained by removing a maximal rim hook from the Young diagram of $\lambda$. This is only possible if $i_1>1$ and $i_k=n$, otherwise the quantum term is absent. Thus, we arrive at
    \[
    \ket{1}\star\ket{\lambda}=\sum_{\mu-\lambda=(1)}\ket{\mu}+(y_{i_1}+\cdots+y_{i_k}-y_1-\cdots-y_k)\ket{\lambda}
    +q\ket{\lambda^-}\;.
    \] 
    This is the quantum Chevalley formula for $qH^*_T(\Gr(k;n))$; see \cite[Thm 1]{MihalceaAIM}.
\end{proof}

\subsection{The extended affine Weyl group action}
Recall the left Weyl group action from Lemma \ref{lem:Waction}. We now extend this action to an action of the extended affine Weyl group $W=\Z^n\rtimes S_n$ on $V[z^{\pm 1}]$.
\begin{lemma}\label{lem:Waffaction}
    Define a left $W$-action on $V[z^{\pm 1}]$ via
    \begin{equation}
    (s_i,v_k)\mapsto g_i.v_k\quad\text{ and }\quad (t_j,v_k)\mapsto (X^q-y_j\cdot1).v_k\;,
    \end{equation}
    where $s_i$ with $i=1,\ldots,n-1$ are the simple reflections in $S_n$ and $t_j$ with $j=1,\ldots,n$ the translations by the standard basis of $\Z^n$. Here it is understood that we interpret $V[z^{\pm 1}]$ as module with respect to the skew group ring $\ring\#W$ where the translations are acting trivially on $\ring$.  We note that
    \begin{equation}\label{rho&z}
    t_1t_2\cdots t_n=z\cdot1\quad\text{and}\quad\varrho=s_1\cdots s_{n-1}t_n=zE_{1,n}+\sum_{i=1}^{n-1}E_{i+1,i}\;.
    \end{equation}
    Moreover, the affine Weyl reflection $s_0=\varrho s_{n-1}\varrho^{-1}$ acts via the matrix \eqref{s0}.
\end{lemma}
Using the standard coproduct for the group algebra of $W$ the action extends canonically to $V[z^{\pm 1}]^{\otimes k}$ and $\bigwedge^kV[z^{\pm 1}]$ by acting diagonally. If we consider in particular the action of the cyclic element $\varrho$ in \eqref{rho&z} on $\bigwedge^kV[z^{\pm 1}]$ we find that
\[
\varrho.(v_{i_1}\wedge\cdots\wedge v_{i_{k-1}}\wedge v_n )=
(-1)^{k-1}z\;v_1\wedge v_{i_1+1}\wedge\cdots\wedge v_{i_{k-1}+1}
\]
and, thus, for positivity it is natural to define the quantum parameter for $qH^*_T(\Gr(k;n))$ as $q=(-1)^{k-1}z$.

The following lemma summarises the commutation relations between the extended affine Weyl group and the Clifford algebra. 
\begin{lemma}\label{lem:WaffClifford}
In addition to the relations from Lemma \ref{lem:WClifford}, we have
\begin{gather}
     s_0\circ\psi^*_i=(g_0.\psi_i^*)\circ s_0,\quad
     g_0.\psi^*_i=\psi^*_i+\delta_{i1}z^{-1}(y_n-y_{1})\psi^*_{n}\\
     s_0\circ\psi_i=(g_0.\psi_i)\circ s_0,\quad
     g_0.\psi_i=\psi_i-\delta_{in}z^{-1}(y_n-y_1)\psi_{1}\;.
\end{gather}
as well as
\begin{equation}\label{rhopsi}
\begin{array}{c}
\varrho\circ\psi^*_i= z^{\delta_{in}} \psi^*_{i+1} \circ \varrho= ((-1)^{\op{deg}-1}q)^{\delta_{in}}\psi^*_{i+1}\circ \varrho, \\
\varrho\circ\psi_i= z^{-\delta_{i1}} \psi_{i-1} \circ \varrho =((-1)^{\op{deg}-1}q)^{-\delta_{i1}}\psi_{i-1}\circ \varrho,
\end{array}
\end{equation}
where $i=1,\ldots,n$ and indices are understood modulo $n$ and deg denotes the `degree operator' which assigns to a tensor its degree. Finally, for the translations $t_i$ we find the relations
\begin{gather} \label{commuterelat}
    t_j\circ\psi^*_i=(t_j.\psi_i^*)\circ t_j,\qquad t_j.\psi_i^*=\psi^*_{i+1}+(y_i-y_j)\psi_i^*=\psi^*_{t_j.v_i}\\
    t^{-1}_j\circ\psi_i=(t^{-1}_j.\psi_i)\circ t^{-1}_j,\qquad t^{-1}_j.\psi_i=\psi_{i-1}+(y_i-y_j)\psi_i=\psi_{(t_j)^T.v_i}\;.
\end{gather}
\end{lemma}

\subsection{The fermionic product formula}
We now state a recurrence formula for the quantum product in $qH^*_T(\Gr(k;n))$ using the Clifford algebra. In the non-equivariant case ($y_i=0$) this formula was obtained in \cite[Eq (1.5) and Thm 1.3]{Korff09} using the results from \cite[Section 11]{korffstroppel2010}. The new result in the present article is its non-trivial extension to the equivariant case and the geometric interpretation of the Clifford algebra in terms of push-pull maps. To state the formula it is convenient to label the Clifford generators in light of \eqref{rhopsi} by arbitrary integers setting for any $1 \leq i \leq n$ 
\begin{equation}\label{quasiperiod}
    \psi_{i + n} = z^{-1}\psi_i=(-1)^{\op{deg}-1}q^{-1} \psi_i\quad\text{ and }
    \quad \psi^*_{i + n} =z\psi^*_i= (-1)^{\op{deg}-1}q\, \psi^*_i\;.
\end{equation}
These `quasi-periodic' boundary conditions where imposed {\em ad-hoc} in \cite{korffstroppel2010,Korff09} but now can be interpreted naturally in light of the extended affine Weyl group action from Lemma \ref{lem:WaffClifford} and \eqref{rhopsi}. Henceforth, we shall use the quantum parameter $q=(-1)^{\op{deg}-1}z$ in formulae and it is always understood that we specialise the loop parameter $z$ in the quantised Chern roots accordingly.

\begin{remark}\rm
Recall that in the recursion relation \eqref{recferm} we used the classical (as opposed to quantum) product in $H^*_T(\Gr(k;n))$. Let us  now use the quantum product in $qH^*(\Gr(k;n))$ instead to define, 
\begin{gather*}
\widetilde{\psi}^*_{i + 1} = c_1(\cS_{k + 1}) \star_{k} \widetilde{\psi}^*_i - \widetilde{\psi}^*_i c_1(\cS_k) \star_{k-1} - y_i \circ \widetilde{\psi}_i\\
\widetilde{\psi}_{i - 1} = \widetilde{\psi}_i c_1(\cS_k) \star_{k} - c_1(\cS_{k - 1}) \star_{k-1} \widetilde{\psi}_i - y_i \circ \widetilde{\psi}_i 
\end{gather*}
with the same initial conditions \eqref{initial}. N.B. the deformation parameters $q_k,q_{k-1}$ in the rings $qH^*_T(\Gr(k;n))$ and $qH^*_T(\Gr(k-1;n))$ are related via a minus sign, $q_k=-q_{k-1}=(-1)^{k-1}z$, and we have therefore labelled the products to distinguish in which ring the multiplication is taken. We have the following result:
    for $ 1 \leq  i \leq n$ the action of the operators  $\tilde\psi_i^*,\tilde\psi_i:\bigwedge^k V[z^{\pm1}]\to \bigwedge^{k \pm 1} V[z^{\pm1}]$ coincides with the action of the classical ones $\psi^*_i, \psi_i$,
    \[
    \psi^*_i = \widetilde{\psi}_i^*, \quad \psi_{n + 1 -i} = \widetilde{\psi}_{n + 1 -i}\;.
    \]
    But for $i>n$ we now have the following ``quasi-periodicity'' relations
    \[
    \widetilde{\psi}^*_{i + n} = z \widetilde{\psi}^*_{i} = (-1)^{k} q \widetilde{\psi}_i^*, \quad \widetilde{\psi}_{i - n} = z \widetilde{\psi}_i = (-1)^{k} q \widetilde{\psi}_i\;;
    \]
    compare with \eqref{quasiperiod}. To unburden formulae we shall keep the notation $\psi^*_i,\psi_i$ for the `quantum operators' but it is now always tacitly understood that the relations \eqref{quasiperiod} hold.     
\end{remark}
Recall from Lemma \ref{lem:BAE} that the matrices $X^q-y_j\cdot1$ are invertible if $q^{-1}$ exists, they are the images of the translations $t_j$ in $\End_{\ring} V[z^{\pm 1}]$; see Lemma \ref{lem:Waffaction}. Hence, they define automorphisms of the Clifford algebra over $\ring[q^{\pm1}]$ via 
    \[
    \psi_i^*\mapsto t_j.\psi^*_i = \psi^*_{i + 1} + (y_i - y_j)\psi^*_i=\psi^*_{t_j.v_i},
    \]
where $i,j=1,2,\ldots,n$ and $\psi^*_{t_j.v_i}=(t_j.v_i)\wedge$ is the wedge operator by the translated basis vector $t_j.v_i$. Using these automorphisms we now define the quantum product recursively via the push-pull maps \eqref{fermgeom}. 

For $\lambda,\mu$ partitions with $\mu\subset\lambda$ recall that one calls the resulting skew diagram $\lambda/\mu$ a {\em horizontal strip} if each column contains at most a single box; see e.g. \cite[Chapter I]{macdonald1998symmetric}. Suppose that $\lambda,\mu$ are such that their Young diagrams fit into the $k\times(n-k)$ rectangle, then to each horizontal strip $\lambda/\mu$ we consider the following composition of translations,
\begin{equation}\label{t2h.s.}
t^{\lambda/\mu} = \prod\limits_{\Box \in \lambda/\mu} t_{k + c(\Box)},
\end{equation}
where the product runs over all boxes in the strip $\lambda/\mu$ and $c(\Box)=\op{col}(\Box)-\op{row}(\Box)$ is its content, the column minus the row number of $\Box$ in the diagram of $\lambda$.
\begin{remark}\label{rmk:Richardson}
\rm
    There is an alternative definition of the above product \eqref{t2h.s.} of translations which is closer to a geometric interpretation using Grassmannian permutations instead. Namely, let $u,w\in S^{k,n}$ and $u\le w$ in the Bruhat order with reduced words $u=s_{i_1}\ldots s_{i_a}$ and $w=s_{j_1}\ldots s_{j_b}$ obtained by reading out the content of each column in the corresponding Young diagrams bottom to top from right to left. We note that such pairs $(u\le w)$ of Grassmannian permutation label Richardson varieties whose projection on to $\Gr(k;n)$ yield a stratification. Replacing in the word $wu^{-1}=s_{l_1}\ldots s_{l_d}$ the simple reflections $s_{l_i}$ with translations $t_{l_i}$ gives precisely the translation $t^{\lambda/\mu}$ from \eqref{t2h.s.} where $\mu$ is the partition corresponding to $u$ and $\lambda$ is the partition corresponding to $w$. 
\end{remark}

\begin{example}\rm
    Let $k=4$ and $n=9$. We fix $\lambda=(5,3,2,1)$ and $\mu=(3,3,1,0)$ with Young diagrams
    \[
    \begin{ytableau}[]
     4 & 5 & 6 & 7& 8  \\ 3& 4 & 5 \\ 2 & 3 \\ 1\\   
    \end{ytableau}  \qquad \qquad
    \begin{ytableau}
        4 & 5 & 6\\ 3 & 4 & 5 \\ 2\\
    \end{ytableau}\;,
    \]
    where the entries in each box give the content plus $k$. The resulting reduced expressions for the corresponding Grassmannian permutations are
    \[
    w=s_8\,s_7\,s_5s_6\,s_3s_4s_5\, s_1s_2s_3s_4\qquad\qquad
    u=s_5s_6\,s_4s_5\,s_2s_3s_4\;.
    \]
    Thus, we find that $wu^{-1}=s_8s_7s_3s_1$ which matches the product of translations $t_8t_7t_3t_1$ obtained from the skew diagram $\lambda/\mu$,
    \[
    \begin{ytableau}[]
      \times & \times & \times & 7& 8  \\ \times & \times & \times \\  \times & 3 \\ 1\\   
    \end{ytableau}\;.
    \]
\end{example}

Fix two Grassmannian permutations $u,w\in S^{k,n}$ and let $\lambda(u)$ be the partition corresponding to $u$ and $w=[i_1\ldots i_ki_{k+1}\ldots i_n]$. Then we define the following binary operation on $\bigwedge^kV[z^{\pm1}]/\langle z+(-1)^kq\rangle$: 
    \begin{equation}
    \label{fermionformula}
        \ket{u} \circledast \ket{w} := \sum\limits_{T(u)} (t^{\rho_1}. \psi^{*}_{i_1})\circ (t^{\rho_{2}}.\psi^*_{i_{2}})\circ \cdots\circ(t^{\rho_k}.\psi^*_{i_k}).1
    \end{equation}
    where the sum runs over all semi-standard tableaux $T(u)$ of shape $\lambda(u)$, i.e. sequences of partitions\footnote{N.B. we are labelling the horizontal strips in reverse order so that they match the labelling of the indices of the fermionic creation operators in \eqref{fermionformula} above.}
\[
\emptyset=\lambda^{(0)}\subset\lambda^{(1)}\subset\cdots\subset\lambda^{(k)}=\lambda
\]
such that $\rho_{k+1-i}=\lambda^{(i)}/\lambda^{(i-1)}$ is a horizontal strip  and $\lambda^{(i)}$ fits into the rectangle $i\times(n-i)$. Moreover, we have identified $\bigwedge^0V[q^{\pm 1}]\simeq\ring[q^{\pm1}]$ in \eqref{fermionformula}.
\begin{thm}\label{thm:fermion}
    The operation \eqref{fermionformula} defines an associative and commutative product on $\bigwedge^kV[z^{\pm1}]/\langle z+(-1)^kq\rangle$ which coincides with the quantum product $\star$. In particular, the map \eqref{Satake} becomes an isomorphism of $\ring[q^{\pm 1}]$-algebras $(\bigwedge^kV[z^{\pm1}]/\langle z+(-1)^kq\rangle,\circledast)\to qH^*_T(\Gr(k;n))$. 
\end{thm}
As both products, $\star$ and $\circledast$, coincide we will henceforth drop the latter notation.
\begin{remark}\rm
    Setting $y_1=\ldots=y_n=0$ we recover the non-equivariant fermionic product formula from \cite[Eqn (1.5)]{Korff09}. Observe that in this case $t^{\rho_i}.\psi^*_j=\psi^*_{j+|\rho_i|}$ where $|\rho_i|=\#i$ is the number of boxes in the horizontal strip $\rho_i$, or equivalently, the number of entries $i$ in the semistandard tableau $T$. Thus, the product formula \eqref{fermionformula} simplifies to
    \begin{equation}
        \ket{u} \star \ket{w} := \sum\limits_{T(u)} \psi^{*}_{i_1+|\rho_1|}\circ\psi^*_{i_{2}+|\rho_{2}|}\circ \cdots\circ\psi^*_{i_k+|\rho_k|}.1\;.
    \end{equation}
    In terms of the affine Weyl group action this can be understood as follows: note that in the non-equivariant limit setting $y_1=\cdots=y_n=0$ the reflections $s_i$ act trivially. Since $t_i=s_{i-1}\cdots s_1\varrho s_{n-1}\cdots s_i$ the action of the translations simply reduces to acting with the cyclic element $\varrho$ and, hence, $t^{\rho_i}.\psi^*_j=\varrho^{|\rho_i|}.\psi^*_j=\psi^*_{j+|\rho_i|}$. N.B. that in \cite{Korff09} the fermionic recurrence formula was written in terms of the quantum deformation parameter $q=(-1)^kz$ instead, which leads to alternating signs as $k$ varies with each application of $\psi^*_i$. 
\end{remark}

\begin{example}\label{exampleproduct} \rm
    Let $k = 2$ and $n = 5$, consider the partitions $\lambda = (2 , 1)$ and $\mu = (3,1)$. The corresponding Grassmannian permutation $w(\mu)\in S^{2,5}$ for $\mu$ is $w(\mu)=[25134]$. For the Young diagram of $\lambda$ there are only two semi-standard Young tableaux,
    \[
    \begin{ytableau}[]
    1 & 1  \\ 2 \\   
    \end{ytableau}  \qquad \qquad
    \begin{ytableau}
        1 & 2\\
        2 \\
    \end{ytableau}\;,
    \]
    which result in the following product expansion
    \begin{align*}
        \ket{\lambda} \star \ket{\mu } =\;& (t_1. \psi^*_2) (t_1 t_2. \psi^*_5).1 + (t_1 t_3. \psi_2^*) (t_1. \psi^*_5).1 \\
        =\;& -z (y_2 - y_1)(y_5 - y_3) \psi^*_1 \psi^*_2.1 - z (y_5 - y_1)\psi^*_1 \psi^*_3.1 - z \psi^*_1 \psi^*_4 - z \psi^*_2 \psi^*_3.1 \\
        &+ (y_2 - y_1)(y_5 - y_1)(y_5 - y_3) \psi^*_2 \psi^*_5.1 
        + (y_5 - y_1)^2 \psi^*_3 \psi^*_5.1 + (y_5 - y_1) \psi^*_4 \psi^*_5.1\\
        =\;& q (y_2 - y_1)(y_5 - y_3) \ket{0,0} + q (y_5 - y_1)\ket{1,0} + q \ket{2,0} + q \ket{1,1} \\
        &+ (y_2 - y_1)(y_5 - y_1)(y_5 - y_3) \ket{3,1} 
        + (y_5 - y_1)^2 \ket{3,2} + (y_5 - y_1) \ket{3,3}.
    \end{align*}
    Here we have identified in the last step the quasi-periodicity parameter $z$ from \eqref{quasiperiod} with the quantum parameter $q$ according to $z=(-1)^{k-1}q=-q$.
    \end{example}
The proof of Theorem \ref{thm:fermion} requires the following Lemma.
\begin{lemma}
        (i) The following commutation relation holds true, when acting on $\bigwedge^{k-1}V[z^{\pm1}]$,
        \begin{equation}
        \label{fschurcom}
            s_{\lambda}(X^q_1,\ldots,X_k^q | y) \psi^*_i = \sum_{\mu}  (t^{\lambda/\mu}.\psi^*_{i}) s_{\mu}(X^q_1,\ldots, X^q_{k-1}|y),
        \end{equation}
        where the sum runs over all partitions $\mu=(\mu_1,\ldots,\mu_{k-1})$ such that $\lambda/\mu$ is a horizontal strip and $t^{\lambda/\mu}$ is the product of translations as defined in \eqref{t2h.s.}, 
        $
            t^{\lambda/\mu}=\prod\limits_{\square \in \lambda /\mu} t_{k + c(\square)}\;.
        $
 (ii) Similarly, for $\psi_i$ acting on $\bigwedge^{k}V[z^{\pm 1}]$ we have the following relation:
        \begin{equation} 
        \label{commpsi}
            s_{\lambda}(X_1^q, \ldots , X_{k-1}^q|y) \psi_i = \sum\limits_{\mu} (-1)^{|\lambda/\mu|} (t^{-\lambda/\mu}.\psi_i) s_{\mu}(X_1^q, \ldots, X_k^q|y),
    \end{equation}
    where the sum runs over all partitions $\mu$ such that $\lambda/\mu$ is a vertical strip (including the empty one for $\mu=\lambda$) and $t^{-\lambda / \mu} = \prod\limits_{\square \in \lambda / \mu}t^{-1}_{k + c(\square)}$ with the translations acting according to %
    $t^{-1}_j. \psi_i = \psi_{i - 1} + (y_i - y_j) \psi_i$; compare with Lemma \ref{lem:WaffClifford}.
    \end{lemma}
    \begin{proof}
        We start by considering the branching rule \eqref{BranchingRule} for factorial Schur polynomials (with respect to the last variable) which is a direct consequence of their tableaux representation:
        \[
        s_\lambda(x|y) = \sum_{\mu} \prod\limits_{\square \in \lambda / \mu} (x_k - y_{k + c(\square)}) s_{\mu}(x_1, \ldots, x_{k - 1}|y),
        \]
        where the sum runs over all $\mu=(\mu_1,\ldots,\mu_{k-1})$ such that $\lambda/\mu$ is a horizontal strip. 

        Given any $v\in\bigwedge^{k-1}V[z^{\pm1}]$, consider (up to a sign) the preimage $v \otimes v_i \in V^{\otimes k}[z^{\pm1}]$ of the vector $\psi^*_i v = v_i \wedge v \in \bigwedge^{k}V[z^{\pm1}]$ under the projection $V[z^{\pm1}]^{\otimes k}\twoheadrightarrow\bigwedge^kV[z^{\pm1}]$. Then
        \begin{align*}
        s_{\lambda}(X^q|y) v \otimes v_i &= \sum_{\mu} \prod\limits_{\square \in \lambda / \mu} (X^q_k - y_{k + c(\square)}\cdot 1) s_{\mu}(X^q_1, \ldots, X^q_{k - 1}|y) v \otimes v_i \\
        &= \sum_\mu s_{\mu}(X^q_1, \ldots, X^q_{k - 1}|y)v \otimes \prod\limits_{\square \in \lambda / \mu} t_{k + c(\square)}.v_i,
        \end{align*}
        which gives us the desired formula \eqref{fschurcom} after projecting back to the wedge space.
        
        To prove the second formula, recall the reverse branching rule \eqref{ReverseBranching} for factorial Schur polynomials
        \[
        s_{\lambda}(x_1, \ldots, x_{k - 1}|y) = \sum\limits_{\mu}(-1)^{|\lambda/\mu|}\prod\limits_{\square \in \lambda / \mu}(x_k - y_{k + c(\square)}) s_{\mu}(x|y),
        \]
        where the sum runs over all Young diagrams $\mu$ such that $\lambda / \mu$ is a vertical strip. Lift the action of the operators $X_i$ and $\psi_j$ to the tensor product $V^{\otimes k}[z^{\pm 1}]$, where $\psi_j$ acts by contraction in the last tensor component. (This may differ by a sign from our previous definition of the lifted action, but does not affect the resulting formula.) Noting that $
        \psi_i (X_k - y_j\cdot 1) = t^{-1}_j.\psi_i, 
        $ 
        we obtain
        \begin{align*}
        \sum\limits_{\mu}(-1)^{|\lambda / \mu|} (t^{- \lambda / \mu}.\psi_i)s_{\mu}(X^q|y) \;&= \sum\limits_{\mu} (-1)^{|\lambda / \mu|}\psi_i \prod\limits_{\square \in \lambda / \mu}(X_k^q - y_{k + c(\square)})s_\mu(X^q | y) \\
        \;&= \psi_i s_{\lambda}(X_1^q, \ldots, X_{k - 1}^q|y) = s_{\lambda}(X_1^q, \ldots, X_{k - 1}^q|y) \psi_i,  
        \end{align*}
        where the last equality holds because $\psi_i$ acts by contraction only in the last tensor component. Projecting back to the wedge space we obtain statement of the proposition.
    \end{proof}
    Let us state two particular cases of the formulae \eqref{fschurcom} and \eqref{commpsi} in which they simplify significantly.
    \begin{cor}
        The following commutation relations hold, when acting in $\bigwedge^{k - 1}V[z^{\pm1}]$ and $\bigwedge^{k}V[z^{\pm1}]$ respectively,
        \begin{equation}
        \label{elemcom}
            e_j(X_1^q, \ldots, X_k^q|y) \psi^*_i = \psi^*_i e_j(X_1^q, \ldots, X_{k - 1}^q|y) + (t_{k - j + 1}.\psi_i^*)e_{j - 1}(X_1^q, \ldots, X_{k - 1}^q|y),
        \end{equation}
        \begin{equation}
        \label{compcom}
            h_j(X_1^q, \ldots, X_{k - 1}^q|y) \psi_i = \psi_i h_j(X_1^q, \ldots, X_k^q|y) - (t^{-1}_{k + j - 1}.\psi_i)h_{j - 1}(X_1^q, \ldots, X_k^q|y).
        \end{equation}
    \end{cor}
    Note that in formula \eqref{fermionformula} we assume that the fermionic operators are arranged in ascending order. However, we can (and sometimes will) reorder the fermionic operators in the product formula \eqref{fermionformula} using the canonical anticommutation relations \eqref{CAR}. Alternatively, we can apply the branching rule in the proof of \eqref{fschurcom} with respect to a different variable as the factorial Schur polynomial is symmetric in the $x$-variables. This reordering procedure is important when showing Graham positivity of the equivariant quantum Pieri rules since the latter depends on the definition of the Schubert basis which we have identified with an ascending wedge product via \eqref{Satake}.
\begin{proof}[Proof of Theorem \ref{thm:fermion}]
    Using the product formula \eqref{star_product} from Proposition \ref{prop:quantumSatake} and the last lemma (but with respect to the first variable) one shows by induction on $k$ that the products \eqref{star_product} and \eqref{fermionformula} coincide.
\end{proof}
\begin{remark}\rm
Note that the commutation relation \eqref{commpsi} allows for an alternative computation of the quantum product \eqref{fermionformula} in terms of the inverse translations and the annihilation operators $\psi_i$ instead: for any Grassmannian permutation $w=[i_1\ldots i_n]\in S^{k,n}$ we have
\[
\ket{w}=(-1)^{\ell(w^\vee)}\psi_{i_n}\ldots\psi_{i_{k+1}}v_1\wedge\ldots\wedge v_n\;,
\]
where $w^\vee=[i_{k+1}\ldots i_ni_1\ldots i_k]\in S^{n-k,n}$. Applying repeatedly \eqref{commpsi} one arrives by a completely analogous argument at a dual fermionic product formula,
\begin{multline}
    \label{fermionformula2}
        \ket{u} \star \ket{w} =\\
        (-1)^{k(n-k)+\ell(w)+|\lambda(u)|} \sum\limits_{T'(u)} (t^{-\nu_1}.\psi_{w_n}) \cdots (t^{-\nu_{n-k}}.\psi_{w_{n+1-k}})v_1\wedge\ldots\wedge v_n\,, 
    \end{multline}
where the sum runs over all partitions $T'(u)$ of $\lambda(u)$ into vertical strips $\nu_i$. Recall that a vertical strip is a skew diagram which can have multiple boxes in the same column but at most one box in each row; we leave the details to the reader.    
\end{remark}
As an immediate consequence of \eqref{fermionformula} and Theorem \ref{thm:fermion} we have the following new expression of equivariant Gromov-Witten invariants.
\begin{cor}
Let $u=[u_1\ldots u_n],v=[v_1\ldots v_n],w=[w_1\ldots w_n]\in S^{k,n}$ be Grassmannian permutations and set $\lambda_i=u_{k+1-i}+i-k-1$ for $i=1,\ldots,k$. Then the equivariant Gromov-Witten invariants \eqref{EGWdef} are given by
\begin{equation}\label{EGW=VEV}
C_{uv}^{w}(y,q)=\sum_{T(\lambda)}
\langle\psi_{w_k}\cdots\psi_{w_1}(t^{\rho_1}.\psi^{*}_{v_1}) \cdots(t^{\rho_k}.\psi^*_{v_k})\rangle
\end{equation}
where the sum runs over all semi-standard tableaux $T$ of shape $\lambda$.
\end{cor}

\begin{remark}\rm
Note that the above formula simplifies in the non-equivariant setting when we specialise $y_1=\ldots=y_n=0$. As explained above the translations then simply act via $t_i.\psi^*_j=\psi^*_{j+1}$ for all $i,j=1,\ldots,n$ and, thus, we can rewrite \eqref{EGW=VEV} as
\[
C_{uv}^{w,d}(0)=\sum_{\alpha}
(-1)^{d(k-1)}K_{\lambda,\alpha}\langle\psi_{w_k}\cdots\psi_{w_1}\psi^{*}_{v_1+\alpha_1} \cdots\psi^*_{v_k+\alpha_k}\rangle
\]
where the sum now runs over all compositions $\alpha\in\Z^k_{\ge 0}$ of $\ell(u)=|\lambda(u)|$, $n\,d=\ell(u)+\ell(v)-\ell(w)$, and $K_{\lambda,\alpha}$ are the {\em Kostka numbers}, the number of semi-standard tableaux of weight $\alpha$; compare with \cite[Cor.1.4]{Korff09}\footnote{In loc. cit. the formula has been further simplified using that $K_{\lambda,\alpha}=K_{\lambda,\alpha'}$, where $\alpha'$ is any permutation of $\alpha$, and Wick's Theorem.}. In the case of setting also $d=0$ one obtains Littlewood-Richardson coefficients and computing the latter via this identity is known as {\em Racah-Speiser algorithm} or {\em Weyl's method of characters}; see e.g. \cite{knapp1996lie}.
\end{remark}

\subsection{Fermions and the shuffle product}\label{sec:shuffle} 
Let us mention yet another interpretation of the action of fermionic creation operators solely in terms of factorial Schur polynomials. To do that, we briefly recall the construction of cohomological Hall algebras due to Kontsevich and Soibelman \cite{kontsevich2011cohomological}, for details, we refer to the original paper. Their construction works for any finite quiver $Q$, but here we are only interested in the simplest case $A_1$, namely the quiver $Q=\bullet$ with one vertex and no arrows. To the quiver $Q$ one associates the space $\mathcal{H} = \bigoplus_{\underline{k}}\mathcal{H}_{\underline{k}}$, where $\underline{k}$ is a dimension vector of the quiver and $\mathcal{H}_{\underline{k}}$ is the $\GL_{\underline{k}}$ equivariant cohomology of the space of quiver representations of $Q$. In \cite{kontsevich2011cohomological} the authors define a graded multiplication $m_{\underline{k}_1, \underline{k}_2}: \mathcal{H}_{\underline{k}_1} \otimes \mathcal{H}_{\underline{k}_2} \to \mathcal{H}_{\underline{k}_1 + \underline{k}_2}$ on the space $\mathcal{H}$, which they then explicitly calculate in terms of the {\em shuffle product} from \cite{feigin1995vector}.  In our case the dimension vector is simply a non-negative integer $k \in\Z_{\ge 0}$ and $\mathcal{H}_{k} = H^*_T(\Gr(k;\infty)) \simeq \ring[x_1, \ldots, x_{k}]^{S_{k}}$ is the $T$-equivariant cohomology of the classifying space. The shuffle product is then given by
\begin{equation}\label{shuffle}
    (f_1 \ast f_2)(x_1, \ldots, x_{k_1 + k_2}) = \sum\limits_{\substack{|I|=k_1 \\ |J| = k_2 \\ I \cup J = [k_1 + k_2]}} \frac{f_1 (x_{i_1}, \ldots , x_{i_{k_1}}) f_2(x_{j_1}, \ldots, x_{j_{k_2}})}{\prod\limits_{l = 1}^{k_1}\prod\limits_{r = 1}^{k_2}(x_{j_r} - x_{i_l})},
\end{equation}
where the sum is taken over all disjoint ordered subsets $I=\{i_1<\ldots<i_{k_1}\}$ and $J=\{j_1<\ldots<j_{k_2}\}$ of $\{1,2,\ldots,k_1+k_2\}$. 

Recall that the action of the fermionic creation operator produces from the (equivariant) Schubert class $\sigma_w\in qH_T^*(\Gr(k;n))$ labelled by the Grassmannian permutation $w\in S^{k,n}$ another (equivariant) Schubert class $\sigma_{w''}=(-1)^{\#(w'',w)}\psi^*_i(\sigma_w)\in qH^*_T(\Gr(k+1;n))$, possibly with a minus sign; see Figure \ref{fig:CliffYoung} for an example. By abuse of notation let us denote the Grassmannian permutation $w''\in S^{k+1,n}$ by $\psi_i^*.w$. The following simple proposition explains the connection between the action of the fermionic creation operators and the shuffle product.
\begin{prop}
    Let $s_{\lambda(w)}(x_2, \ldots, x_{k + 1}|y)$ be the factorial Schur polynomial representing the corresponding equivariant Schubert class $\sigma_w$ in the presentation \eqref{QHdef}, then
    \begin{equation}\label{simpleshuffle}
        (x_1|y)^{l-1} \ast s_{\lambda(w)}(x_2, \ldots, x_{k + 1}|y) = 
        (-1)^{\# (\psi_l^*.w,w)}s_{\lambda(\psi^*_{l}.w)} (x_1, \ldots, x_{k + 1}|y)
    \end{equation}
    if $l \notin \{w_1, \ldots w_k \}$ and $0$ otherwise. In other words, computing the shuffle product of the factorial Schur polynomial with the factorial monomial $(x|y)^{l-1}$ is the same as acting on the corresponding Grassmannian permutation by the fermionic creation operator.
\end{prop}
\begin{proof}
    Let us compute the shuffle product 
    \begin{multline}\label{shufflexpand}
        (x_1|y)^{l-1} \ast s_{\lambda(w)}(x_2, \ldots, x_{k + 1}|y) =\\ \sum\limits_{i = 1}^{k + 1} \frac{(x_i | y)^{l - 1}}{\prod\limits_{j \neq i}(x_j - x_i)} s_{\lambda(w)} (x_1, \ldots, x_{i - 1}, x_{i + 1}, \ldots x_{k + 1}|y)=\\
         \sum\limits_{i = 1}^{k + 1}\frac{(-1)^{k + 1 - i}(x_i|y)^{l - 1}}{\Delta(x)} a_{\lambda(w)}(x_1, \ldots, x_{i - 1}, x_{i + 1}, \ldots, x_{k+1}|y),
    \end{multline}
    where $\Delta(x)= \prod\limits_{i<j}^{k + 1}(x_i - x_j)$ is the Vandermonde determinant and 
    \[
    a_{\lambda}(x_1, \dots, x_k|y) = \det((x_i|y)^{k + \lambda_j - j})
    \]
    is the alternant in the definition of the factorial Schur polynomial; see Appendix \ref{app:A}. On the other hand, the expression \eqref{shufflexpand} is up to sign exactly the expansion of the alternant in the definition of $s_{\lambda(\psi^*_l.w)}(x_1, \ldots, x_{k+1}|y)$ with respect to the row corresponding to the exponent $l-1$, which is zero if $l \in \{ w_1, \ldots w_k \}$ because of two coinciding rows corresponding to the same exponent.
\end{proof}
\noindent Notice that in the shuffle reformulation of the fermionic creation operators the quasi-periodicity condition \eqref{quasiperiod} follows immediately when we specialise the variable $x_i$ to satisfy the Bethe ansatz equations \eqref{GrassBAE} and assume $n$-periodicity of the equivariant parameters $y_{i + n} = y_i$.

To describe also the action of the annihilation operators $\psi_i$ we first recall Newton's interpolation formula for polynomials; see e.g. \cite{macdonald1991notes}.
\begin{lemma}
    Let $p(x)\in\C[x]$ be a polynomial of degree $\le n$. Then
    \[
    p(x)=\sum_{r=0}^n[p]_r\,(x|y)^r,\qquad [p]_{r}=\sum_{i=1}^{r+1}\frac{p(y_i)}{\prod_{1\le j\ne i\le r}(y_i-y_j)}\;.
    \]
    Equivalently, the `coefficient' $[p]_r$ can be expressed in terms of Demazure operators as
    \[
    [p]_r=\partial_r\ldots\partial_1p(y_1),\qquad \partial_i=\frac{1-s_i}{y_i-y_{i+1}}\;.
    \]
\end{lemma}
We now apply this interpolation formula to the last variable in a factorial Schur polynomial $s_{\lambda(w)}(x_1,\ldots,x_k|y)$ representing an equivariant Schubert class $\sigma_w$ in $qH^*_T(\Gr(k;n))$.
\begin{prop}
    Let $w\in S^{k,n}$ and $\lambda(w)$ be the corresponding partition. Then
    \begin{align}
        \left[(x_1 - x_k) \cdots (x_{k-1} - x_k)s_{\lambda(w)}(x_1, \ldots, x_k | y)\right]_{i-1} &=\notag\\
        &(-1)^{\#(w,\psi_i.w)} s_{\lambda(\psi_i.w)}(x_1, \ldots, x_{k - 1}|y)\;,\label{annihshuffle}
    \end{align}
    or alternatively in terms of the Demazure operators $\partial_i=\frac{1-s_i}{y_i-y_{i+1}}$,
    \begin{align}
        \partial_{i-1}\ldots\partial_1\left[(x_1 - y_1) \cdots (x_{k-1} - y_1)s_{\lambda(w)}(x_1, \ldots, x_{k-1},y_1 | y)\right] &=\notag\\
        (-1)^{\#(w,\psi_i.w)}& s_{\lambda(\psi_i.w)}(x_1, \ldots, x_{k - 1}|y)\;.\label{annihshuffle2}
    \end{align}
    Here $w'=\psi_i.w\in S^{k-1,n}$ denotes the Grassmannian permutation obtained by the procedure as explained in the example of Figure \ref{fig:CliffYoung} and $\#(w,\psi_i.w)=\#(w,w')$ is the height of the ribbon removed minus one.
\end{prop}

\begin{proof}
    Consider the alternant formula \eqref{factSchur} for the factorial Schur polynomial from Appendix \ref{app:A}. Expanding the determinant in the numerator with respect to the last column we find
    \begin{equation*}
        (x_1 - x_k) \cdots (x_{k-1} - x_k)s_{\lambda(w)}(x_1, \ldots, x_k | y)=\sum_{i=1}^k(-1)^{k-i}(x_k|y)^{\lambda_i+k-i}\frac{M_{i,k}}{\prod_{1\le a<b\le k-1}(x_a-x_b)},
    \end{equation*}
    where $M_{i,k}$ is the minor obtained by deleting the $i$th row and $k$th column. The latter is precisely the factorial Schur polynomial $s_{\lambda(\psi_i.w)}(x_1,\ldots,x_{k-1}|y)$.
\end{proof}

In Section \ref{sec:LRduality} we will show that the action \eqref{annihshuffle} of the annihilation operator can also be interpreted in terms of the shuffle product \eqref{shuffle}, albeit on the ring $qH^*_T(\Gr(n-k;n))$; see the commutative diagram \eqref{LRcd}.

\section{Applications}
In this section we state some combinatorial applications of our description of equivariant quantum cohomology using the Clifford algebra and also make contact with previous results obtained in the literature. Namely, we give new combinatorial proofs for
\begin{enumerate}
\item {\em Level-Rank duality}, a ring isomorphism $qH^*_T(\Gr(k;n))\simeq qH^*_T(\Gr(n-k;n))$ in \cite{gorbounovkorff2014,gorbounov2017quantum} and apply it to relate the fermion creation and annihilation operators;
\item {\em Graham positivity} of the equivariant quantum Pieri rules. The general result of Graham positivity for equivariant quantum Schubert calculus of Grassmannians is due to Mihalcea \cite{Mihalcea06} using geometric methods, but here we compare with the combinatorial formulae obtained by Bertiger et al. in \cite{BEMT22}.
\end{enumerate}
Motivated by the recent result \cite{gao2025graham} on generalised Graham positivity for (non-quantum) triple Schubert calculus  \cite{MS99,KnutTao2003}, we conclude by stating a conjecture for generalised Graham positivity for a (yet to be geometrically defined) quantum version of triple Schubert calculus and prove a manifestly positive quantum Pieri rule.

\subsection{Recurrence formulae for Gromov-Witten invariants}
We use the the commutation relations \eqref{fschurcom} and \eqref{commpsi} to state (one step) recurrence relations between Gromov-Witten invariants.
\begin{cor}
    We have the following recurrence relations between products in the rings $qH_T^*(\Gr(k;n))$ and $qH_T^*(\Gr(k\mp1;n))$,
    \begin{gather} \label{recurrence}
        \sigma_\lambda\star_k\sigma_\mu=\frac{1}{k}\sum_{j=1}^n\sum_{\nu}(t^{\lambda/\nu}.\psi^*_j)(\sigma_\nu\star_{k-1}\psi_{j}(\sigma_\mu))\\
        \sigma_\lambda\star_k\sigma_\mu=\frac{1}{n-k}\sum_{j=1}^n\sum_{\nu'}(-1)^{|\lambda/\nu'|}(t^{-\lambda/\nu'}.\psi_j)(\sigma_{\nu'}\star_{k+1}\psi^*_{j}(\sigma_\mu))\notag\;,
    \end{gather}
    where the sum in the first (second) identity runs over all partitions $\nu\subset (k-1)\times(n+1-k)$ ($\nu'\subset k\times (n-k)$) such that $\lambda/\nu$ is a horizontal (vertical) strip. We note that the quantum deformation parameters $q_k$ and $q_{k\pm1}$ in the respective rings are related by a minus sign, i.e. $q_k=-q_{k\pm1}$.
\end{cor}
\begin{proof}
    To obtain the first formula note that $\sum_{i=1}^n\psi^*_i\psi_i|_{\bigwedge^kV}=k\cdot 1$. Inserting the latter identity into the quantum product \eqref{star_product} we arrive at the claimed equation using \eqref{fschurcom}. Similarly, one obtaines the second one using that $\sum_{i=1}^n\psi_i\psi^*_i|_{\bigwedge^kV}=(n-k)\cdot 1$ and then applying \eqref{commpsi}.
\end{proof}
N.B. as the equivariant Gromov-Witten invariants $C_{\lambda\mu}^{\nu,d}(y)\in\Z_{\ge 0}[y_1,\ldots,y_n]$ the above identities \eqref{recurrence} imply that the sum of the structure constants for the products on the right hand side must be divisible respectively by $k$ and $n-k$. This is simply a reflection of the fact that there are respectively $k$ and $n-k$ values of $j$ in the sums for which we obtain non-trivial recurrence relations, as the following simple example shows.
    \begin{example}\rm
        Let us consider the case of $\Gr(2;4)$ and classes $\lambda=(2,1)$ and $\mu = (2,2)$, their quantum product can be computed from \eqref{fermionformula} to be
        \begin{multline}\label{ex24}
        \ket{(2,1)} \star_2 \ket{(2,2)} = q \ket{(2,1)} + q (y_3 - y_1) \ket{(1,1)} + q (y_4 - y_2) \ket{(2,0)} \\ + q(y_3 - y_1)(y_4 - y_2) \ket{(1,0)} + (y_4 - y_1)(y_3 - y_1) (y_4 - y_2) \ket{(2,2)}.
        \end{multline}
        We shall focus on the first identity in \eqref{recurrence} and  represent the class $\ket{\mu} = \ket{(2,2)}$ as $\psi^*_3 \psi^*_4.1$. Applying the commutation relations \eqref{fschurcom} gives
        \begin{multline*}
        \ket{(2,1)}\star_2\ket{(2,2)}=s_{(2,1)}(X^q | y) \psi^*_3 \psi^*_4.1 = \\
        (\psi^*_4 + (y_3 - y_1)\psi^*_3) s_{(2)}(X^q | y) \psi^*_4 .1 + (z \psi^*_1 +(y_4 - y_1) \psi^*_4)s_{(1)}(X^q | y) \psi^*_4 .1=\\
        (\psi^*_4 + (y_3 - y_1)\psi^*_3)\ket{(2)}\star_1\ket{(3)}
        +(-q \psi^*_1 +(y_4 - y_1) \psi^*_4)\ket{(1)}\star_1\ket{(3)}
        \end{multline*}
        Comparing coefficients on both sides of the equality we obtain the following recurrence relation for the Gromov-Witten invariant $C_{(2,1),(2,2)}^{(2,1)}(2, q) = q $ (the first term in \eqref{ex24})\footnote{In this example we write $C^{\nu}_{\lambda\mu}(k,q)$ to denote the Gromow--Witten invariants for $\Gr(k;n)$ omitting the dependence on the equivariant parameters to ease the notations.},
        \[
        C_{(2,1),(2,2)}^{(2,1)}(2,q) = -C_{2,3}^2(1, -q) + (y_4 - y_1) C_{1,3}^{2}(1, -q)\;.
        \]
        where on the right hand side we have Gromov-Witten invariants for $\mathbb{P}^3$. 
        
        \noindent Alternatively, we can use the commutation relation \eqref{fschurcom} for $\psi^*_4$ to obtain
        \begin{multline*}
            \ket{(2,1)} \star_2 \ket{(2,2)} = 
            -s_{(2,1)}(X^q|y) \psi^*_4 \psi^*_3.1= 
            (q \psi^*_1 - (y_4 - y_1)\psi^*_4) \ket{(2)} \star_{1} \ket{(2)}\\+ (q \psi^*_2 + q (y_4 - y_3) \psi^*_1 - (y_4 - y_3)(y_4 - y_1)\psi^*_4) \ket{(1)} \star_1 \ket{(2)}.
        \end{multline*}
        Again, we concentrate on on the first term in \eqref{ex24} to find the alternative relations
        \[
        C_{(2,1),(2,2)}^{(2,1)}(2, q) = (y_4 - y_1)C_{2,2}^1(1,-q) + q C_{1,2}^3(1,-q) + (y_4 - y_3)(y_4 - y_1)C_{1,2}^1(1,-q)\;.
        \]
        As both expressions for $C_{(2,1),(2,2)}^{(2,1)}(2, q)$ in terms of the invariants for $\mb{P}^3$ must be equally valid, we end up for $\mb{P}^3$ with the identity,
        \begin{multline*}
         -C_{2,3}^2(1, -q) + (y_4 - y_1) C_{1,3}^{2}(1, -q)=\\
            (y_4 - y_1)C_{2,2}^1(1,-q) + q C_{1,2}^3(1,-q) + (y_4 - y_3)(y_4 - y_1)C_{1,2}^1(1,-q)\;.
        \end{multline*}
        Thus, their sum in \eqref{recurrence} is obviously divisible by two. In other words, we obtain from \eqref{recurrence} in general $k$ different recurrence relations for a Gromov-Witten invariant in $qH^*_T(\Gr(k;n))$.
    \end{example}
    Let us now formulate the recurrence identities \eqref{recurrence} directly in terms of the equivariant Gromov-Witten invariants. To ease our notation somewhat we use the abbreviations $\psi^*_i . \lambda$ or $\psi_i . \lambda$ for the Young diagrams obtained by acting with the Clifford generators on the Schubert class $\sigma_\lambda$; see Figure \ref{fig:CliffYoung} for an example. We also introduce the following ``matrix elements'' in $\ring[q^{\pm1}]$ coming from the action of the translations on the fermionic operators
    \begin{equation}\label{matrixcoef}
        t^{\lambda/\rho}.\psi^*_i = \sum\limits_{r = 1}^n A_{ir}(\lambda/\rho) \psi^*_r, \qquad t^{-\lambda/\rho}.\psi_i = \sum\limits_{r = 1}^n B_{ir}(\lambda/\rho) \psi_r\;.
    \end{equation}
    \begin{cor}
        Denote by $C_{\lambda \mu}^{\nu}(k, q)$ the Gromov--Witten invariants of $\Gr(k;n)$, then the following recurrence relation is true
        \begin{equation}
            C_{\lambda \mu}^{\nu}(k, q) = \sum\limits_{r = 1}^n\sum\limits_{\rho} A_{i_j,r}(\lambda/\rho) C_{\rho, \psi_{i_j}.\mu}^{\psi_r. \nu }(k - 1, -q),
        \end{equation}
        where the sum runs over all $\rho=(\rho_1,\ldots,\rho_{k-1})$ such that $\lambda/\rho$ is a horizontal strip, $i_j=\mu_j+k+1-j$  and the $A_{i_j,r}(\lambda/\rho)$ are given by \eqref{matrixcoef}.
        Similarly, we have for all $i \neq \mu_j + k + 1 - j$ the recursion relation 
        \begin{equation} \label{secondrec}
            C_{\lambda \mu}^{\nu}(k,q) = \sum\limits_{\rho}(-1)^{|\lambda/\rho|} B_{ir}(\lambda/\rho) C_{\rho, \psi^*_i.\mu}^{\psi^*_r.\nu}(k+1,-q),
        \end{equation}
        where the sum now runs over all $\rho=(\rho_1, \ldots, \rho_k)$ such that $\lambda/\rho$ is a vertical strip and the $B_{ir}(\lambda/\rho)$ are given by \eqref{matrixcoef}.
    \end{cor}

\begin{proof}
    We employ once more the formulae \eqref{star_product} to find 
    \[
    \ket{\lambda} \star \ket{\mu} = \sum\limits_{\nu} C_{\lambda \mu}^{\nu}(k, q) \ket{\nu}=s_{\lambda}(X^q|y)\ket{\mu}=s_{\lambda}(X^q|y)\psi^*_{i_j}\psi_{i_j}\ket{\mu}
    \]
    and then \eqref{fschurcom}. In the final step we compare coefficients on both sides of the equality using that if $\psi_i \ket{\nu} \neq 0$, then
        \[
        \psi^*_i \ket{\alpha} = \pm \ket{\nu} \iff \ket{\alpha} = \pm \psi_i \ket{\nu} \;.
        \]
        In case that $\psi_i \ket{\nu} = 0$, then there are no solutions to the equation $\psi_i^* \ket{\alpha} = \ket{\nu}$. The formula \eqref{secondrec} is proved in the same way by inserting the "identity" $\psi_i \psi^*_i$ for $i \neq \mu_j + k + 1 - j$ and applying the formula \eqref{commpsi}.
\end{proof}

\begin{figure}
\centering
\includegraphics[width=.4\textwidth]{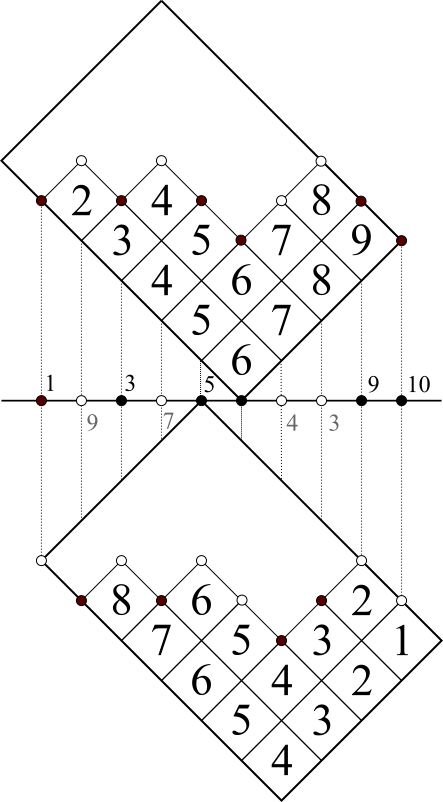} 
\caption{A graphical depiction of level-rank duality for an example with $k=6$ and $n=10$. On top is the Young diagram of $\lambda=(4,4,2,2,1,0)$ and on the bottom the diagram for the conjugate partition $\lambda'=(5,4,2,2)$ under level-rank duality. In the middle is the corresponding finite Maya diagram encoding the wedge product $\ket{\lambda}=v_1\wedge v_3\wedge v_5\wedge v_6\wedge v_9\wedge v_{10}$ where the indices correspond to the positions of the black go-stones in the Maya diagram. Under level-rank duality one swaps black and white go-stones and reverses their positions. 
}\label{fig:Young}
\end{figure}

\subsection{Equivariant Level-Rank Duality}\label{sec:LRduality}
It is well-known that Grassmannians of complementary dimensions are isomorphic $\Gr(k ; n) \simeq \Gr(n - k; n)$ as algebraic varieties, the isomorphism can be easily understood in the following way: consider the complex vector space $V = \mathbb{C}^n$ and a $k$-dimensional subspace $U \subset V$, consider the annihilator ${\rm Ann}(U)= \{ \alpha \in V^* | \alpha(U) = 0 \}$ of $U$. By basic linear algebra arguments it follows that ${\rm Ann}(U)$ is an $n-k$ dimensional subspace in $V^*$, thus we have $\Gr(k; V) \simeq \Gr(n - k, V^*)$ and this identification of algebraic varieties then implies a ring isomorphism $H^*_T(\Gr(k;n))\simeq H^*_T(\Gr(n-k;n))$. Here we shall recall its quantum analogue which was obtained in \cite[Cor. 6.8.]{gorbounovkorff2014}.

    Define an involution on $\ring[q^{\pm1}]$ via $f\mapsto f'$, where $f'(y_1,\ldots,y_n;q)=f(-y_n,\cdots,-y_1;q)$. We shall call an $\ring[q^{\pm 1}]$-module homomorphism $\varphi:M\to N$ {\em antilinear} if $\varphi(f\cdot m)=f'\cdot \varphi(n)$ for all $m\in M,n\in N$ and $f\in\ring[q^{\pm 1}]$.
\begin{thm}\label{thm:LRduality}
    The anti-linear map $\Theta_k:qH^*_T(\Gr(k;n))\to qH^*_T(\Gr(n-k;n))$ of $\ring[q^{\pm 1}]$-modules given by
    \begin{equation}\label{Theta}
    f\cdot\sigma_\lambda\mapsto f'\cdot\sigma_{\lambda'},\quad\forall f\in\ring[q^{\pm 1}], 
    \end{equation}
    where $\lambda'$ denotes the partition conjugate to $\lambda$, is a ring isomorphism. In particular, we have the following equality for the equivariant Gromov-Witten invariants \eqref{EGWdef},
    \begin{equation}\label{LRduality}
    C_{\lambda\mu}^{\nu,d}(y_1,\ldots,y_n)=C_{\lambda'\mu'}^{\nu',d}(-y_n,\ldots,-y_1)\;.
    \end{equation}
\end{thm}
    The proof of this result in \cite{gorbounovkorff2014} used the Bertram-Vafa-Intrilligator formula for the Gromov-Witten invariants and the Bethe ansatz. Here we wish to give an independent proof using the Clifford algebra and keeping this article self-contained.
    \begin{proof}[A bijective proof of level-rank duality]
        Let $u, w \in S^{k , n}$ and define $u', w' \in S^{n - k , n}$ to be the Grassmannian permutations corresponding to the conjugate partitions, i.e. they satisfy $\lambda'(u) = \lambda(u')$ and $\lambda'(w) = \lambda(w')$. In terms of the Clifford algebra we have 
    \begin{multline}
    \label{corresp}
       \Theta_k \ket{w} =\Theta_k( v_{w_1}\wedge\ldots\wedge v_{w_k}) =\\
       \Theta_k(\psi^*_{w_1} \cdots \psi^*_{w_k}.1)=
        \ket{w'} = v_{n+1-w_{n}}\wedge\ldots\wedge v_{n+1-w_{k+1}} = \\(-1)^{k (n - k) +\ell(w)}\psi_{n + 1 - w_1} \cdots \psi_{n + 1 - w_k} v_1\wedge\ldots\wedge v_n,
    \end{multline}
    where we recall that $\ell(w)=|\lambda(w)|$. From \eqref{corresp} and the identities
    \begin{gather}
        |\psi_i^*.\lambda|=|\lambda|+ i - k - 1 ,\qquad
        |\psi_i.\lambda|=|\lambda|+ k - i,
    \end{gather}
    where by abuse of notation $\psi_i^*.\lambda$ denotes the partition labelling the Schubert class $\psi^*_{i}(\sigma_\lambda)$ and, likewise, $\psi_i.\lambda$ the one labelling $\psi_i(\sigma_\lambda)$, we arrive at the commutation relations
    \begin{gather}\label{LRpsi}
        \Theta_{k+1} \circ\psi^*_i=(-1)^{n-k-i}\psi_{n+1-i}\circ\Theta_k,\qquad
        \Theta_{k-1}\circ\psi_i=(-1)^{n-k+1-i}\psi^*_{n+1-i}\circ\Theta_k\;.
    \end{gather}
    Compare the signs in formula \eqref{LRpsi} with the ones obtained from the graphical action of the Clifford algebra on Young diagrams depicted in Figures \ref{fig:CliffYoung} and \ref{fig:signcompute}. 
    \begin{figure}[h!]
        \centering
        \includegraphics[width=0.7\linewidth]{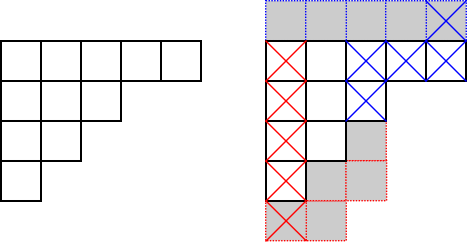}
        \caption{The figure depicts the action of fermionic creation operator $\psi^*_5$ on the Schubert class $\ket{w} = \psi^*_2 \psi^*_4 \psi^*_6 \psi^*_9.1 \in qH^*_T(\Gr(4;9))$ and the action of the fermionic annihilation operator $\psi_{5}$ on a level-rank dual Schubert class $\ket{w'} = \psi^*_2 \psi^*_3\psi^*_5\psi^*_7\psi^*_9.1 \in qH^*_T(\Gr(5;9))$. The action of the creation operator $\psi^*_5$ is obtained by adding a ribbon of length $5$, red dashed boxes in the figure, and removing the leftmost column of maximal height, the boxes  with red crosses. The action of the annihilation operator $\psi_5$ on the dual class consists of adding a row of maximal width on top (blue dashed boxes) and removal of a ribbon of length $5$ from the top (boxes with blue crosses). Together the added (red) and removed (blue) ribbons always meet at the boundary. Together both actions produce a factor $(-1)^{\rm ht^+-wd^-}$ with $\rm ht^+$ the height of the added ribbon and $\rm wd^-$ the width of the removed ribbon. This factor is exactly the one appearing in formula \eqref{LRpsi}.}
        \label{fig:signcompute}
    \end{figure}
    Using the product formulae \eqref{fermionformula} and \eqref{fermionformula2} we have
    \begin{gather}
    \label{prod}
        \ket{u} \star \ket{w} = \sum\limits_{T(u)} (t^{\rho_1}. \psi^{*}_{w_1}) (t^{\rho_{2}}.\psi^*_{w_{2}}) \cdots (t^{\rho_k}.\psi^*_{w_k}).1
    \end{gather}
    and
    \begin{multline}
    \label{proddual}
        \ket{u'} \star \ket{w'} =\\
        (-1)^{k(n-k)+\ell(w)+|\lambda(u)|} \sum\limits_{T'(u')} (t^{-\nu_1}.\psi_{n + 1 - w_1}) \cdots (t^{-\nu_k}.\psi_{n + 1 - w_k})v_1\wedge\ldots\wedge v_n 
    \end{multline}
    where the first sum runs over all possible partitions $T(u)$ of $\lambda(u)$ into horizontal strips $\rho_i$ and the second over all partitions $T'(u')$ of $\lambda(u')$ into vertical strips. We note that each partitioning $T'$ of $\lambda'(u)=\lambda(u')$ into vertical strips is in one-to-one correspondence with a partitioning $T$ of $\lambda(u)$ into horizontal strips: one simply needs to transpose each individual strip $\rho_i$ in $T$ together with the Young diagram $\lambda(u)$.
    Thus, we can simply identify $\nu_i = \rho'_i$ and match each term in \eqref{proddual} against one in \eqref{prod}. Moreover, for each single horizontal strip $\rho_i$ we can also directly match the resulting Clifford operators,
    \begin{equation}
    \label{bijproof}
        t^{\rho_i}.\psi^*_{w_i} = \prod\limits_{\square \in \rho_i} t_{k + 1 - i+c(\square)}.\psi^*_{w_i} \quad {\rm and} \quad t^{-\rho'_i}.\psi_{n+1-w_i} = \prod\limits_{\square' \in \rho'_i} t^{-1}_{n - k + i + c(\square')}.\psi_{n + 1 - w_i},
    \end{equation}
    where we have used that for each $\square'\in\rho'_i$ we have that $c(\square') = - c(\square)$ with $\square$ being the corresponding box in $\rho_i$. Namely, due to the anti-linearity of the map \eqref{Theta} and \eqref{corresp}, we arrive at
    \begin{multline*}
        \Theta_{k+1-i}\circ t_{k + 1 - i + c(\square)}.\psi^*_{w_i}  = \Theta_{k+1-i}\circ(\psi^*_{w_i + 1} + (y_{w_i} - y_{k + 1 - i + c(\square)})\psi^*_{w_i})=\\ 
         (-1)^{n-k+i-w_i+1} (\psi_{n - w_i} + (y_{n + 1 - w_i} - y_{n - k + i - c(\square)})\psi_{n+ 1 - w_i})\circ\Theta_{k-i}=\\
         (-1)^{n-k+i-w_i+1}t^{-1}_{n - k + i - c(\square)}.\psi_{n + 1 - w_i}\circ\Theta_{k-i} \;
    \end{multline*}
    which implies that
    \begin{equation*}
        \Theta_{k + 1 - i} \circ t^{\rho_i}.\psi^*_{w_i} = (-1)^{n - k + i - w_i + |\rho_i|} t^{-\rho'_i}.\psi_{n + 1 - w_i} \circ \Theta_{k - i}\;.
    \end{equation*}
    Applying the latter identity one now easily deduces from \eqref{prod} and \eqref{proddual} that $\Theta_k$ is an isomorphism of rings and, thus, \eqref{LRduality} follows.
    \end{proof}

Note that the relations \eqref{LRpsi} from the last proof give rise to the following:
\begin{cor}
     The diagram \eqref{LRcd} is commutative for any $0<k<n$.
\end{cor}
This, in particular, relates the action \eqref{annihshuffle} of the annihilation operator $\psi_i$ acting on $qH^*_T(\Gr(k;n))$ to the action of the shuffle product \eqref{shuffle} on $qH^*_T(\Gr(n-k;n))$; compare with \eqref{simpleshuffle}. 

\begin{example}\rm
        Let us compare the action of the fermionic creation operator $\psi^*_5$ on the Schubert class $\ket{w} = \psi^*_2\psi^*_4\psi^*_6\psi^*_9.1 \in qH^*_T(\Gr(4;9))$ against the action of the annihilation operator $\psi_5 = \psi_{9+1 - 5}$ on the Schubert class $\ket{w'} = \psi^*_2\psi^*_3\psi^*_5\psi^*_7\psi^*_9.1 \in qH^*_T(\Gr(5;9))$ under level-rank duality; see the depiction in terms of Young diagrams in Figure \ref{fig:signcompute}. We find
        \begin{gather*}
            \psi^*_5\ket{w} = \psi^*_2\psi^*_4 \psi^*_5\psi^*_6\psi^*_9.1, \quad \Theta_5 \circ \psi^*_5\ket{w} = \psi^*_2\psi^*_3\psi^*_7\psi^*_9.1, \\
            \psi_5 \circ \Theta_4(\ket{w}) = \psi_5 \ket{w'} = (-1)^{9-5-4} \psi^*_2 \psi^*_3 \psi^*_7 \psi^*_9.1 = \psi^*_2 \psi^*_3 \psi^*_7 \psi^*_9.1
        \end{gather*}
        in accordance with the commutative diagram \eqref{LRcd}.
    \end{example}
    
\subsection{Graham positivity and quantum Pieri rules}\label{sec:Pieri}
In this section we use the Clifford algebra to prove that the structure constants $C^{\nu,d}_{\lambda\mu}(y)$ of the equivariant quantum cohomology ring $qH^*_T(\Gr(k;n))$ satisfy {\em Graham positivity}, i.e. the $C^{\nu,d}_{\lambda\mu}(y)\in\Z_{\ge 0}[y_2-y_1,\ldots,y_n-y_{n-1}]$  for $\lambda$ being either a single column or a single row. These special cases are known as equivariant quantum Pieri rules. We note that the choice of torus weights $y_{i+1}-y_i$ and, thus, the formulation of Graham positivity, depends on the conventions when fixing the Borel subgroup in the representation of $\Gr(k;n)=\SL_n/P$ as homogeneous variety, where $P$ is a maximal parabolic subgroup. Here we shall follow the conventions from \cite{BEMT22} as we wish to compare directly with their combinatorial description of the quantum Pieri rules \cite[Thm 1.1]{BEMT22} and fix $B\subset\SL_n$ to be the subgroup of {\em lower triangular} matrices, a choice which is already visible in the definition of the matrices \eqref{X0} and \eqref{X}.

We shall proceed as follows: we first prove Graham-positivity for products $\sigma_{(1)^p}\star\sigma_\mu$, where $(1)^p$ is a single column of height $1\le p\le k$ and $\mu$ is arbitrary, by employing solely the fermionic product formula \eqref{fermionformula}. We then explain the connection with the formula on cylindric shapes from \cite[Thm 1.1]{BEMT22}. To prove the other quantum Pieri formula for products of the form $\sigma_p\star\sigma_\mu$ with $1\le p\le n-k$, we then employ the equivariant level-rank duality from \cite{gorbounovkorff2014}.

\begin{cor}\label{PieriPos}
    (i) Let $\lambda=(1)^p$ be single column of height $1\le p\le k$ and $w=[w_1\ldots w_n]$ a Grassmannian permutation with associated partition $\mu$. Then \eqref{fermionformula} simplifies to
    \[
    \ket{(1)^p} \star \ket{\mu} = \sum_{\{i_1<\ldots<i_p\}\subset[k]}
    \psi^*_{w_1}\circ\cdots\circ(t_{i_1}.\psi^*_{w_{i_1}})\circ\cdots\circ(t_{i_p-p+1}.\psi^*_{i_p})\circ\cdots\circ\psi^*_{w_k}.1\,,
    \]
    where the sum runs over all ordered $p$-subsets of $[k]=\{1,\ldots,k\}$. Since $i_r-r+1\le w_{i_r}$ and
    \[
    t_{i_r-r+1}.\psi^*_{w_{i_r}}=\psi^*_{w_{i_r}+1}+(y_{w_{i_r}}-y_{i_r-r+1})\psi^*_{w_{i_r}},\qquad r=1,\ldots,p,
    \]
    this formula is manifestly Graham positive.\medskip

    \noindent (ii) Setting $\lambda=(p)$, a single row of length $1\le p\le n-k$, we apply the anti-linear ring isomorphism \eqref{Theta} to find that the expansion of $\ket{(p)}\star\ket{\mu}=\Theta_{n-k}(\ket{(1^p)}\star\ket{\mu'})$ is also Graham positive.
\end{cor}

\begin{proof}
    (i) We start by noting that since $\lambda=(1)^p$ is a single column, the sum over the Young tableaux of shape $\lambda$ in \eqref{fermionformula} becomes a sum over ordered subsets $\{i_1<\ldots<i_p\}\subset[k]$, where the $i_r$ are the entries in the respective tableau. In particular, each horizontal strip $\rho_j=\lambda^{(j)}/\lambda^{(j-1)}$ in \eqref{fermionformula} contains at most a single box and there are in total $p$ of them.

    To show Graham positivity recall that $1\le w_1<w_2<\ldots<w_k\le n$, since $w\in S_n$ is Grassmannian. Hence, we must have that $i_r-r+1\le w_{i_r}$ as stated and, thus, the expansion of each term,
    \[ \psi^*_{w_1}\circ\cdots\circ(t_{i_1}.\psi^*_{w_{i_1}})\circ\cdots\circ(t_{i_p-p+1}.\psi^*_{i_p})\circ\cdots\circ\psi^*_{w_k}.1
    \]
    gives a $\Z_{\ge 0}$-polynomial in the weights $(y_{w_{i_r}}-y_{i_r-r+1})$. Furthermore, due to the quasi-periodic boundary conditions we have that if $i_p=k$ and $w_k=n$, then $t_{k-p+1}.\psi^*_{n}=(-1)^{k-1}q\psi^*_1+(y_n-y_{k-p+1})\psi^*_n$ and, although a reordering of terms is required, the additional factor $(-1)^{k-1}$ ensures positivity.

    To show (ii) it suffices to note that if the expansion coefficients of the product $\ket{(1^p)}\star\ket{\mu'}$ satisfy $f(y_1,\ldots,y_n)\in\Z_{\ge 0}[y_2-y_1,\ldots,y_n-y_{n-1}]$ then  their images $f'(y_1,\ldots,y_n)=f(-y_n,\ldots,-y_1)$ must also be Graham positive by the anti-linearity property of the map \eqref{Theta}. 
\end{proof}

To demonstrate the simplicity of our formula let us consider an easy example.
\begin{figure}\label{fig:LRduality}
\centering
\includegraphics[width=.8\textwidth]{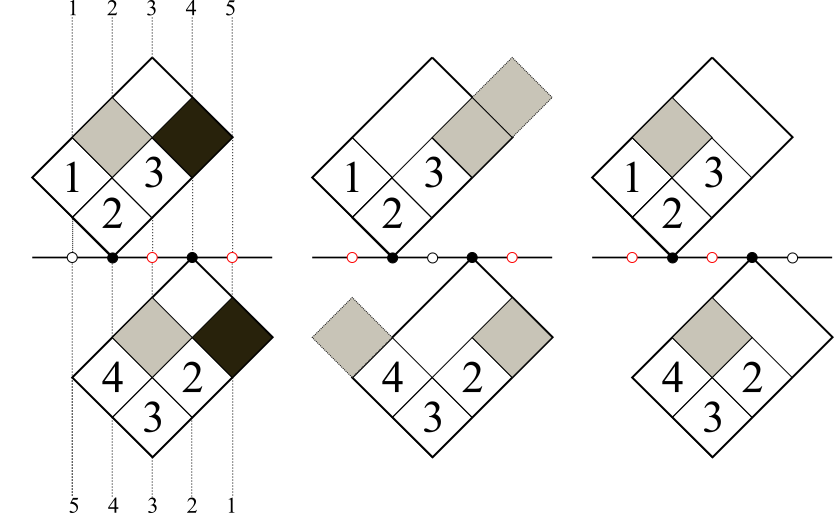} 
\caption{A graphical depiction of the quantum Pieri rule for horizontal strips using level-rank duality. Shown is the example of the product expansion $h_2\star s_{(2,1)}$ in $qH^*(\Gr(2;5))$ (top) and the corresponding expansion of $e_2\star s_{(2,1)}$ in $qH^*(\Gr(3;5))$ (bottom). The gray boxes are addable and the black box in the first diagram must be added; see Example \ref{ex:LRduality} for further explanations.
}
\end{figure}

\begin{example}\label{ex:LRduality}
\rm 
    Set $k=3$ and $n=5$. We wish to consider the product $\sigma_{(1)^2}\star\sigma_{(2,1)}$ in $qH^*_T(\Gr(3;5))$. The Grassmannian permutation $w=[13524]$ corresponds to our chosen partition $\mu$. There are three possible tableaux for $\lambda=(1,1)$,
    \[
    \begin{ytableau}[]
    1 \\ 2 \\   
    \end{ytableau}  \qquad
    \begin{ytableau}[]
    1 \\ 3 \\   
    \end{ytableau}\qquad
    \begin{ytableau}
        2\\
        3 \\
    \end{ytableau}\;,
    \]
    corresponding to the three ordered $2$-subsets $\{1,2\},\{1,3\},\{2,3\}$ of $[3]=\{1,2,3\}$. Each of these subsets gives rise to one of the following terms,
    \begin{gather*}
        (t_1.\psi^*_1)(t_1.\psi^*_3)\psi^*_5.1=
        \psi^*_2(\psi^*_4+(y_3-y_1)\psi^*_3)\psi^*_5.1\\
        (t_1.\psi^*_1)\psi^*_3(t_2.\psi^*_5).1=
        \psi^*_2\psi^*_3(q\psi^*_1+(y_5-y_2)\psi^*_5).1\\
        \psi^*_1(t_2.\psi^*_3)(t_2.\psi^*_5).1=
        \psi^*_1(\psi^*_4+(y_3-y_2)\psi^*_3)(q\psi^*_1+(y_5-y_2)\psi^*_5).1\;.
    \end{gather*}
    Note that each individual coefficient resulting from acting with the translations $t_i$ is Graham positive. 
    Expanding the latter and summing up terms we arrive at
    \begin{multline*}
         \ket{(1)^2}\star\ket{2,1}=(v_1\wedge v_3\wedge v_4)\star(v_1\wedge v_3\wedge v_5)=\\
         (y_5-y_2)(y_3-y_2)v_1\wedge v_3\wedge v_5+(y_5-y_2)v_1\wedge v_4\wedge v_5+\\
         (y_5+y_3-y_2-y_1)v_2\wedge v_3\wedge v_5+
         v_2\wedge v_4\wedge v_5+q\; v_1\wedge v_2\wedge v_3=\\
         (y_5-y_2)(y_3-y_2)\ket{2,1}+(y_5-y_2)\ket{2,2}+(y_5+y_3-y_2-y_1)\ket{2,1,1}+\ket{2,2,1}+q\ket{\emptyset}\;.
    \end{multline*}
    Let us now demonstrate the application of level-rank duality to obtain the other Pieri rule for horizontal strips; the reader might wish to consult Figure \ref{fig:LRduality} for a graphical depiction. The map \eqref{Theta} sends $qH^*_T(\Gr(3;5))$ to $qH^*_T(\Gr(2;5))$ and we now consider instead the quantum product of 
        \[
        \ket{(2)} \star \ket{2,1} =  h_2(X | y) v_2 \wedge v_4=\Theta_3(\ket{(1)^2}\star\ket{2,1})
        \]
    in $qH^*_T(\Gr(2;5))$. We simply need to apply the map $\Theta_3$ term by term to the result above to find
    \begin{multline*}
            h_2(X|y) v_2 \wedge v_4 =(y_4 - y_3)(y_4 - y_1) v_2 \wedge v_4 + (y_4 - y_1) v_3 \wedge v_4\\
            + (y_5 + y_4 - y_3 - y_1) v_2 \wedge v_5 + v_3 \wedge v_5+q v_1 \wedge v_2=\\
            (y_4 - y_3)(y_4 - y_1)\ket{2,1}+(y_4 - y_1)\ket{2,2}+
            (y_5 + y_4 - y_3 - y_1)\ket{3,1}+\ket{3,2}+q\ket{\emptyset}\;.
        \end{multline*}
        Looking at Figure \ref{fig:LRduality} one sees that the individual terms correspond to adding horizontal strips of length $0\le r\le p=2$ and, if the result is outside the bounding box, removing a maximal rim hook afterwards.
\end{example}

\subsection{The equivariant quantum Pieri rule in terms of cylindric shapes}\label{sec:cylloops}
\begin{figure}\label{fig:cylindricloop}
\centering
\includegraphics[width=.8\textwidth]{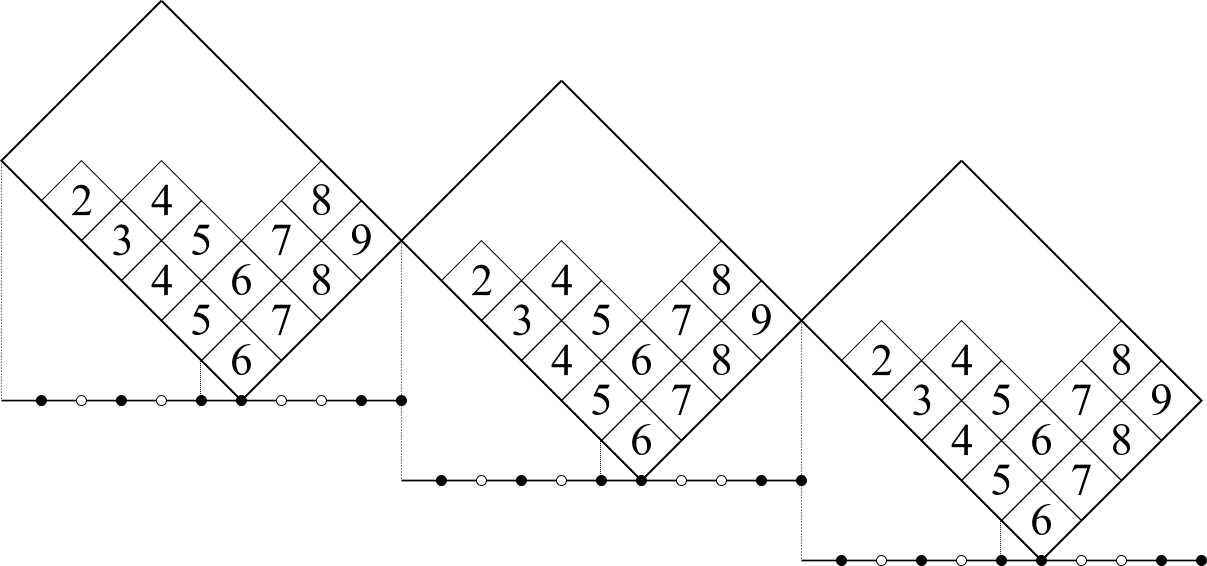} 
\caption{A graphical depiction of the connection between our fermionic picture and cylindric loops. Each partition can be encoded into a finite Maya diagram (a binary string depicted as a string of black and white go-stones) which is periodically repeated. Each Maya diagram defines the outline of a Young diagram by going downward by 45 deg for a black go-stone and going upward by 45 deg for a white one. Shown is the Young diagram for the partition $\lambda=(4,4,2,2,1,0)$ and the numbers in the boxes are the content of that box plus $k=6$. They correspond to the positions of the black go-stones in the Maya diagram.
}
\end{figure}

We recall the quantum equivariant Pieri rule for vertical strips from \cite{BEMT22}. The latter uses so-called {\em cylindric shapes} as introduced in \cite{Post05} when discussing non-equivariant quantum cohomology. 

In our setting the cylindric shapes are recovered as follows: each wedge product $\ket{\lambda}$ can be encoded in a binary string or a finite {\em Maya diagram}, a string of black and white go-stones of length $n$, where the positions $w_1<\ldots<w_k$ of the black go-stones are fixed by the Grassmannian permutation $w=[w_1\ldots w_n]\in S^{k,n}$ corresponding to $\lambda$. If we periodically repeat this Maya diagram as shown in Figure \ref{fig:cylindricloop} then the outline of the cylindric loop associated with $\lambda$ can be recovered by going 45 deg upwards for a white go-stone and 45 deg downwards for a black go-stone. The resulting lattice path in $\Z^2$ is called the cylindric loop $\lambda[0]$ in \cite{Post05}. Such a loop can be shifted upward or downward by either adding or removing $d$ maximal rim hooks in Figure \ref{fig:cylindricloop}. The resulting loop is denoted by $\lambda[d]$. A {\em cylindric shape} $\lambda/d/\mu$ is the number of boxes between the two cylindric loops $\lambda[d]$ and $\mu[0]$. 

The equivariant quantum Pieri rule can be written in terms cylindric shapes as follows: select $p$ rows of the Young diagram and add a single box in them denoting the resulting diagram by $\mu\cup\nu_p$ and calling $\nu_p$ a (generalised) vertical $p$-strip. Note that $\mu\cup\nu_p$ has {\em not} to be a Young diagram and in general it will not be one. The rule then reads \cite[Theorem 1]{BEMT22}:
\begin{thm}[Bertiger et al] One has the following product expansion in $qH^*_T(\Gr(k;n))$:
        \begin{equation}
        \label{QEPieri}
            \ket{(1)^p} \star \ket{\mu} = \sum\limits_{r = 0}^{p} \sum\limits_{\substack{\mu \subset \lambda \subset\, \mu \cup \nu_p \\\lambda /d/\mu = \nu_r}} q^d \prod\limits_{\alpha \in \nu_p  \setminus \nu_r} \rm{wt}(\alpha) \ket{\mu},
        \end{equation}
        where the sum runs over all possible additions of (generalised) vertical $p$-strips $\nu_p$ and partitions $\lambda$ such that its Young diagram is enclosed by the cylindric loops of $\mu$ and $\mu \cup \nu_p$ and $\lambda/d/\mu=\nu_r$ is a vertical $r$-strip whose boxes belong to $\nu_p$. The weight of each box $\Box \in \nu_p \setminus \nu_r$ is defined as
        \[
        {\rm{wt}}(\Box) = y_{\mu_{r(\Box)} + k + 1 - r(\Box)} - y_{k + 1 - r(\Box) - u({\Box})},
        \]
        where $r(\Box)$ is the row number of the box $\Box$ and $u(\Box)$ is the number of the boxes in $\nu_p$ below $\Box$. 
\end{thm}        
        Note that as pointed out in \cite{BEMT22} the quantum equivariant Pieri rule \eqref{QEPieri} is manifestly Graham positive, because
        \[
        \mu_{r(\Box)} + k + 1 - r(\Box) \geq k + 1 - r(\Box) - u(\Box).
        \]
In our setting the selection of the generalised vertical $p$-strip $\nu_p$ corresponds to selecting which black go-stones in the finite Maya diagram we are acting on with the translations $t_{i_r}$. The adding of boxes which do not lead to a well defined Young diagram are eliminated from our formula as they correspond to terms involving $(\psi^*_i)^2=0$.

\subsection{Triple quantum Schubert calculus and positivity}\label{sec:tripleSchubert} 

In this subsection, we speculate on the possible combinatorial definition of a {\em triple quantum Schubert calculus} for Grassmannians. 
Motivated by \eqref{MolevSagan} we make the following change of base ring, $$\ring=\Z[y_1,\ldots,y_n]\to\mb{A}=\Z[y_1,\ldots,y_n,z_1,\ldots,z_n]\;,$$ and consider now the free $\mb{A}$-module $\mb{V}=\mb{A}^n$ instead. If we let $S_n$ act on $\mb{A}$ diagonally, i.e. by simultaneously permuting the parameters $y_i$ and $z_i$, while acting via \eqref{g_i} on the standard basis in $\mb{V}$ as before, then $\mb{V}$ has the structure of an $\mb{A}\#S_n$-module. 
\begin{lemma}
    The matrices $\bt_i=X^q-z_i\cdot1$ for $i=1,\ldots,n$, with $X^q$ as in \eqref{Xmatrix}, together with the matrices \eqref{g_i} define another left action of the affine Weyl group $W=\Z^n\rtimes S_n$ on $\mb{V}^{\rm loc}$, where $\mb{V}^{\rm loc}$ is the corresponding module after the appropriate localisation $\mb{A}\to\mb{A}^{\rm loc}$ such that the inverses of\footnote{Here we have used the quantum parameter $q$ in \eqref{tripleq} instead of the quasi-periodicity parameter $z$ from \eqref{Xmatrix} in order to avoid confusion with the new set of equivariant parameters $z_i$.}
    \begin{equation}\label{tripleq}
    \det \bt_i=q+\prod_{j=1}^n(z_i-y_j)
    \end{equation}
    exist for all $i=1,\ldots,n$.
\end{lemma}
\begin{proof}
    As discussed previously the matrices $X^q$ commute with the $S_n$-action and as they are independent of the additional parameters $z_i$ this remains true. Since we defined the action of $S_n$ to be diagonal on $\mb{A}$ it then is immediate that $s_i\bt_i=\bt_{i+1}s_i$ and $s_i\bt_j=\bt_js_i$ for $|i-j|>1$. 
\end{proof}
Clearly, $\mb{V}^{\rm loc}$ plays the role analogous to that of the quantum cohomology ring of projective space in our previous construction and, motivated by \eqref{MolevSagan} let us define, analogously to \eqref{star_product}, the following bilinear operation $\bigwedge^k\mb{V}^{\rm loc}\otimes\bigwedge^k\mb{V}^{\rm loc}\to\bigwedge^k\mb{V}^{\rm loc}$ 
\begin{equation}\label{triple}
    \ket{\lambda}\star_z\ket{\mu}=s_{\lambda}(X^q|z)s_{\mu}(X^q|y)\ket{e} = s_{\lambda}(X^q|z) \ket{\mu} = \sum\limits_{\nu,d} q^d C^{\nu,d}_{\lambda\mu}(y,z) \ket{\nu}.
\end{equation}
\begin{lemma}
    The pair $(\bigwedge^k\mb{V}^{\rm loc},\star_z)$ is a well-defined associative unital ring.
\end{lemma}
\begin{proof}
    Associativity is clear as we have defined the product via matrix multiplication. The multiplicative unit remains the vector labelled by the empty partition $\lambda=\varnothing$, since $s_\varnothing(X^q|z)=s_\varnothing(X^q|y)$ are both the identity matrix.
\end{proof}
Notice that if both partitions $\lambda$ and $\mu$ are ``small'' enough, i.e. $\lambda_1+|\mu|< n-k$, then no $q$-terms occur in the product and the expansion in \eqref{triple} reduces to the Molev--Sagan formula \eqref{MolevSagan}. In general, however, the product \eqref{triple} will contain non-trivial $q$-terms and specialising $z_i = y_i$ we recover the product for equivariant quantum Schubert calculus from Prop \ref{prop:quantumSatake}. 

\begin{prop}
    Let $w_i=\mu_{k+1-i}+i$ with $w=[w_1,\ldots w_n]\in S^{k,n}$ Grassmannian. Then we have the following triple quantum Chevalley formula,
    \[
    \ket{1}\star_z\ket{\mu}=\sum_{\lambda-\mu=(1)}\ket{\lambda}+(y_{w_1}+\cdots+y_{w_k}-z_1-\cdots-z_k)\ket{\mu}+q\ket{\mu^-},
    \]
    where $\mu^-$ is the partition obtained by removing a maximal rim hook of length $n-1$. If this is impossible, then the last term is omitted. 
\end{prop}
\begin{proof}
    The proof of the quantum Chevalley formula is completely analogous to the case of equivariant quantum Schubert calculus; see Claim \ref{claim:quantumChevalley}.

    To prove the recurrence relation among the structure constants $C_{\lambda\mu}^{\nu,d}(y,z)$ one uses associativity of the product \eqref{triple} similar to the case of equivariant quantum Schubert calculus; compare with \cite{MihalceaAIM}.
\end{proof}

 One of the possible tests to determine whether the formula \eqref{triple} is a meaningful quantum version of \eqref{MolevSagan} and triple Schubert calculus is the {\em refined Graham positivity} property.
\begin{conjecture}
    The coefficients in \eqref{triple} are positive, i.e. $C^{\nu,d}_{\lambda \mu}(y,z) \in \Z_{\geq 0}[y_i - z_j]$.
\end{conjecture}
\noindent Let us illustrate the conjecture with a simple example.
\begin{example}
    Set $k = 2$ and $n = 5$. We compute the product \eqref{triple} for $\lambda = (2,1)$ and $\mu = (3,1)$. 
    \begin{align*}
        s_{(2,1)}(X^q|z) \ket{3,1} =\;& (X_1^q - z_1)(X_2^q - z_2)(X_1^q + X_2^q - z_2 - z_3) v_2 \wedge v_5\\
        =\;& q(y_2 - z_1)(y_1 + y_2 + y_5 - z_1 - z_2 - z_3) \ket{0}\\
        +\;& q(y_1 + y_2 + y_3 + y_5 - 2z_1 - z_2 - z_3)\ket{1,0} + q \ket{1,4} + q \ket{2,3} \\
        +\;& (y_2 - z_1)(y_5 - z_1)(y_2 + y_5 - z_2 - z_3)\ket{2,5} \\
        +\;& (y_5 - z_1)(y_2 + y_3 + y_5 - z_1 - z_2 - z_3)\ket{3,5} + (y_5 - z_1) \ket{4,5}.
    \end{align*}
\end{example}
\noindent Having identified $\mb{V}$ with the triple Schubert calculus analogue of projective space we can now apply the same construction as previously. The proofs generalise in a straightforward manner and we therefore omit details.
\begin{prop}
    The product \eqref{triple} can be computed in terms of fermions as
    \begin{equation}\label{triplefermi}
        \ket{\lambda}\star_z\ket{\mu}=s_{\lambda}(X^q|z)\ket{\mu} = \sum\limits_{T(\lambda)} (\bt^{\rho_1}. \psi^{*}_{i_1})\circ (\bt^{\rho_{2}}.\psi^*_{i_{2}})\circ \cdots\circ(\bt^{\rho_k}.\psi^*_{i_k}).1,
    \end{equation}
    where $\ket{\mu} = \psi^*_{i_1} \psi^*_{i_2} \cdots \psi^*_{i_k}.1$ and
    \[
    \bt^{\lambda/\mu}.\psi^*_i = \prod\limits_{\square \in \lambda/ \mu} \bt_{k + c(\square)}\psi^*_i, \quad \bt_j \psi^*_i = \psi^*_{i + 1} + (y_i - z_j) \psi^*_i
    \]
    with the same quasi-periodicity relation, $\psi^*_{n+i}=(-1)^{k-1}q\psi^*_i$, as before.
\end{prop}
\begin{proof}
    In complete analogy with \eqref{fschurcom} one derives the commutation relation
    \begin{equation}\label{triplecom}
        s_{\lambda}(X_1^q, \ldots , X_k^q|z)\psi^*_i = \sum\limits_{\nu} (\bt^{\lambda/\mu}.\psi^*_i)s_{\nu}(X_1^q, \ldots , X_{k - 1}^q|z)
    \end{equation}
    via the branching rule for factorial Schur functions. Applying this relation repeatedly we arrive at the claim.
\end{proof}
With the help of the fermionic formula \eqref{triplefermi}, we can now easily prove positivity for a {\em triple quantum Pieri rule}.
\begin{cor}
    The expansion coefficients $C^{\nu,d}_{(1^r)\mu}(y,z) \in \Z_{\geq 0}[y_i - z_j]$ are positive.
\end{cor}
\begin{proof}
      Specialising $\lambda = (1^r)$ we repeat the same argument as in Corollary \ref{PieriPos}.
\end{proof}
\appendix

\section{Factorial Schur Polynomials}\label{app:A}
Factorial Schur polynomials are a special case of double Schubert polynomials, which were originally introduced by Lascoux and Sch\"utzenberger. They are obtained by specialising the definition of double Schubert polynomials to Grassmannian permutations similar as Schur polynomials are a special case of single Schubert polynomials under the same specialisation. In the appendix we collect some basic known facts about factorial Schur polynomials and refer the reader to \cite{Mac92} for proofs and further details. 

Let $\{y_i\}_{i=1}^{\infty}$ be an infinite sequence of pairwise commuting variables (we assume $y_j = 0$ for $j > n$ in the main body of the paper), then the factorial Schur polynomial is defined via the alternant formula similar to the usual Schur polynomials
\begin{equation}
    \label{factSchur}
    s_{\lambda}(x|y) = \frac{\underset{n \times n}{\det}((x_j|y)^{\lambda_i + n - i})}{\Delta(x)}, 
\end{equation}
where $\Delta(x) = \prod\limits_{i < j}^n (x_i - x_j)$ is the Vandermonde determinant and $(x|y)^{i} = \prod\limits_{j = 1}^i (x - y_j)$ is the factorial power. Alternatively, one can define factorial Schur polynomials combinatorially using semi-standard Young tableaux, 
\begin{equation}
    \label{tablformula}
    s_{\lambda}(x|y) = \sum\limits_{T(\lambda)}   \prod\limits_{\square \in \lambda} (x_{T(\square)} - y_{T(\square) + c(\square)}),
\end{equation}
where $c({\square}) = j - i$ is the content of the box $\square = (i,j)$. Setting all the parameters $y_i = 0$ recovers the usual Schur polynomial. The complete and elementary factorial symmetric polynomials are defined in the same manner as in the non-factorial case,
\begin{equation}
    \label{complelem}
    h_k(x|y) = s_{(k)}(x|y), \quad e_{k}(x|y) = s_{(1^k)}(x|y).
\end{equation}
The factorial versions of the Jacobi-Trudi and N\"{a}gelsbach-Kostka formulae are
\begin{equation}
    s_{\lambda}(x|y) = \underset{n \times n}{\det}(h_{\lambda_i-i+j}(x | \tau^{1-j}y)), \qquad s_{\lambda}(x|y) = \underset{n \times n}{\det}(e_{\lambda'_i - i +j }(x | \tau^{j - 1}y)),
\end{equation}
where $\tau(y_i) = y_{i + 1}$ is a shift operator and $\lambda'$ is the conjugate partition.

\noindent Branching and reverse branching rules make it possible to express factorial Schur polynomials as a sum of factorial Schur polynomials dependent on fewer or more variables, respectively. First, we recall the well-known branching rule
\begin{equation}\label{BranchingRule}
    s_{\lambda}(x_1, \ldots, x_{k-1},x_k|y) = \sum\limits_{\mu} \prod\limits_{\square \in \lambda / \mu}(x_k - y_{k + c(\square)}) s_{\mu}(x_1, \ldots, x_{k-1}|y),
\end{equation}
where the sum runs over all partitions $\mu \subseteq \lambda$ such that $\lambda / \mu$ is a horizontal strip. We also present the reverse branching rule\footnote{We were not able to find an exact reference for the reverse branching rule, thus we provide a proof, which uses only the above mentioned properties of the factorial Schur polynomials.}
\begin{prop}
The factorial Schur polynomials satisfy the reverse branching rule
\begin{equation}\label{ReverseBranching}
    s_{\lambda}(x_1, \ldots, x_{k - 1}|y) = \sum\limits_{\mu}(-1)^{|\lambda/\mu|}\prod\limits_{\square \in \lambda / \mu}(x_k - y_{k + c(\square)}) s_{\mu}(x_1, \ldots, x_{k - 1}, x_k|y),
\end{equation}
where the sum runs over all partitions $\mu \subseteq \lambda$ such that $\lambda / \mu$ is a vertical strip. Note that since factorial Schur polynomials are symmetric with respect to $x$ variables, the branching rule \eqref{BranchingRule} with respect to another $x$ variable is the same.
\end{prop}
\begin{proof}
    We first prove the formula \eqref{ReverseBranching} for complete factorial symmetric polynomials. Indeed, to compute $h_i(x_1, \ldots, x_{k - 1}|y)$ we sum over the semistandard Young tableaux (SSYT) of shape $\lambda = (i)$ filled in with numbers at most $k-1$ via formula \eqref{tablformula}. On the other hand, the expression
    \[
    h_i(x_1, \ldots, x_k|y) - (x_k - y_{k + i - 1})h_{i - 1}(x_1, \ldots, x_k|y)
    \]
    counts all SSYT of the same shape $\lambda = (i)$ filled in with numbers at most $k$ minus all SSYT with at least one $k$ appearing, which is the same as counting all SSYT with at most $k-1$, thus we obtain
    \begin{equation}
    \label{reversecomplete}
    h_i(x_1, \ldots, x_{k - 1}|y) = h_i(x_1, \ldots, x_k|y) - (x_k - y_{k + i - 1})h_{i - 1}(x_1, \ldots, x_k|y).
    \end{equation}
    Now we use the Jacobi-Trudi formula to express the factorial Schur polynomials in terms of factorial complete symmetric polynomials
    \[
    s_{\lambda}(x_1, \ldots, x_{k - 1}|y) = \det(h_{\lambda_{i} - i + j}(x_1, \ldots, x_{k - 1}|\tau^{1 - j}y)).
    \]
    Apply the reverse branching rule \eqref{reversecomplete} for each of the complete factorial symmetric polynomials 
    \begin{multline*}
    h_{\lambda_i - i + j}(x_1, \ldots, x_{k - 1}|\tau^{1-j}y) = h_{\lambda_i - i + j}(x_1, \ldots, x_{k}|\tau^{1-j}y)\\
    - (x_k - y_{k +\lambda_i - i})h_{\lambda_i - i + j - 1}(x_1, \ldots, x_{k}|\tau^{1-j}y).
    \end{multline*}
    Notice that the factor 
    \[
    x_k - y_{k + \lambda_i - i} = x_k - y_{k + c(\square)}
    \]
    corresponds exactly to the deletion of the box on the border of $\lambda$ in the $i$-th row in the formula \eqref{ReverseBranching}. Thus, we can write, using the linearity of the determinant, the initial Jacobi-Trudi determinant as the sum of the determinants with boxes deleted from the border of $\lambda$ and only one box deleted in the same row
    \[
    \det(h_{\lambda_{i} - i + j}(x_1, \ldots, x_{k - 1}|\tau^{1 - j}y)) = \sum\limits_{\lambda/\mu=(1^r)}(-1)^r \prod\limits_{\square \in \lambda / \mu}(x_k - y_{k + c(\square)}) s_{\mu}(x_1, \ldots, x_k|y),
    \]
    which finishes the proof of the proposition.
\end{proof}
\noindent The factorial Schur polynomials also satisfy a dual version of the Cauchy identity \cite{Mac92}; see also \cite{bump2011factorial} for the proof using integrable lattice models. Let $x=(x_1, \ldots, x_k)$ and $z = (z_1, \ldots, z_{n - k})$, then the dual Cauchy identity holds
\begin{equation}\label{Cauchydual}
    \sum\limits_{\lambda \subset \mathcal{P}_{k,n}} s_{\lambda}(x|y) s_{\lambda^{*}}(z|-y) = \prod\limits_{i = 1}^k \prod\limits_{j = 1}^{n - k}(x_i + z_j) = E(z_1) \cdots E(z_{n - k}),
\end{equation}
where 
\[
E(z) = \prod\limits_{i = 1}^k(z + x_i) = \sum\limits_{r = 0}^k (z|-y)^r e_{k - r}(x|y)
\]
is a generating function of factorial elementary symmetric polynomials \eqref{complelem} and the sum runs over all partitions $\lambda$ fitting into the $k\times (n - k)$ rectangle with $\lambda^{*} = (\lambda^{\vee})^{'}$ is a partition obtained from $\lambda$ taking the complement in the $k \times (n-k)$ rectangle and transposing. 
\section{The local affine plactic algebra}\label{app:B}
We give an alternative combinatorial description of the Chern classes $c_r(\cS_k)$ under the quantum Satake correspondence using the Clifford algebra $\mc{C}_n(\ring)$. This description generalises the one of non-equivariant quantum cohomology from \cite{Post05} in terms of the (affine) nil Temperley-Lieb algebra $\mathrm{nilTL}_n$. Namely, in the equivariant setting one obtains instead a (non-faithful) representation of the local affine plactic algebra considered in \cite{korffstroppel2010} (albeit over the ring $\ring$ rather than the field $\C$).

 \begin{defn} \label{defplactic}
        The local affine plactic algebra $\mathcal{A}_n$ is a free unital associative algebra generated by the set $\{u_1, \ldots, u_{n} \}$ modulo the relations
        \begin{equation}
        \label{plactic_commute}
        u_i u_j = u_j u_i, \quad |i - j| \neq 1 \;\; mod \;\; n, 
        \end{equation}
        \begin{equation}
        \label{plactic}
        u_i^2 u_{i + 1} = u_i u_{i + 1} u_i, \quad u_i u_{i + 1}^2 = u_{i + 1} u_i u_{i + 1},
        \end{equation}
        where all the indices in the relation \eqref{plactic} are understood modulo $n$. 
    \end{defn}
    \subsection{Plactic polynomials}
    We recall the definition of a commutative subalgebra $\mc{B}_n\subset\mc{A}_n$ in terms of plactic polynomials from \cite{korffstroppel2010}: for any monomial in the generators $u_i$ we recall the notion of {\em clockwise} cyclic orientation (respectively {\em anticlockwise} cyclic orientation). Let $I$ be a multiset, i.e. a set with possible repetitions, with elements $1\le i\le n$ such that there exists at least one $1\le j\le n$ with $j\not\in I$. Then we define a {\em clockwise ordered monomial} $\prod\limits_{i \in I}^{\circlearrowright} u_i$ to be a product of the $u_i$ over the multiset $I$, such that $u_{i + 1}$ appears before $u_i$, where indices are understood modulo $n$. 
    
    Similarly, define the {\em anticlockwise ordered monomial} $\prod\limits_{i \in I}^{\circlearrowleft} u_i$, such that $u_i$ appears before $u_{i + 1}$ modulo $n$.
    
    Next we define the plactic elementary and complete symmetric polynomials; compare with \cite{korffstroppel2010}.
    \begin{defn}
        The affine plactic elementary and complete symmetric polynomials are defined by the cyclically ordered products
        \begin{equation}\label{noncomelem}
            \bs{e}_r = \sum\limits_{|I| = r}\prod\limits_{i \in I}^{\circlearrowleft} u_i \qquad\text{and}\qquad 
            \bs{h}_r = \sum\limits_{|I| = r} \prod\limits_{i \in I}^{\circlearrowright} u_i, \quad 0<r< n,
        \end{equation}
        where the first sum runs over all sets (without repetitions) of size $0< r<n$ and the second one runs over all multisets of size $0<r<n$. For $r=0$ we set $\bs{e}_0=\bs{h}_0=1$.
    \end{defn}
    N.B. the definition \eqref{noncomelem} mimics the definition of the elementary and complete symmetric polynomials (modulo the plactic relations) in the case of commuting variables; see \cite{macdonald1998symmetric}.  
    \begin{example}\rm
        Consider the case $n = 3$, then we have
        \[
        \bs{e}_1 = \bs{h}_1 = u_1 + u_2 + u_3, \quad \bs{e}_2 = u_1 u_2 + u_2 u_3 + u_3 u_1, \quad \bs{h}_2 = u_1^2 + u_2^2 + u_3^2 + u_2 u_1 + u_3 u_2 + u_1 u_3\;.
        \]
        We note that the quadratic terms $u_i^2$ do not vanish in contrast to the polynomials defined in terms of the affine nil-Temperley-Lieb algebra for non-equivariant quantum cohomology; see \cite{Post05}.
    \end{example}
    It was shown in \cite{korffstroppel2010} that the affine plactic elementary and complete symmetric polynomials commute.
    \begin{lemma}
    The elements \eqref{noncomelem} generate an abelian subalgebra $\mc{B}_n\subset\mc{A}_n$, i.e.
    \[
    \bs{e}_i\bs{e}_j =\bs{e}_j\bs{e}_i,\quad\bs{h}_i\bs{h}_j=\bs{h}_j\bs{h}_i,\quad
    \bs{e}_i\bs{h}_j=\bs{h}_j\bs{e}_i
    \]
        for all $0<i,j<n$.
    \end{lemma}
    \subsection{The Jordan-Schwinger map for fermions}
     Recall the Jordan-Schwinger map
    \begin{equation}\label{jsmap}
\jmath:U(\gl_n(\ring))\to\End_{\ring}\bigwedge\nolimits^kV,\qquad x\mapsto \sum_{i,j=1}^n\psi^*_ix_{ij}\psi_j,\quad\forall x\in\gl_n(\ring)\;.
    \end{equation}
    This map is a Lie algebra morphism where the Lie bracket is the commutator. Moreover, if we consider the canonical action of the universal enveloping algebra $U(\gl_n(\ring))$ on $V^{\otimes k}$ and its projection onto $\bigwedge^kV$ then we arrive at the commuting diagram
    \[
  \begin{tikzcd}
    U(\gl_n(\ring)) \arrow{r}{\Delta^{k-1}} \arrow[swap]{dr}{\jmath} & \End_{\ring} V^{\otimes k} \arrow{d}{\op{res}} \\
     & \End_{\ring}\bigwedge\nolimits^kV
  \end{tikzcd}
\]
  where $\Delta:U\to U\otimes_\ring U$ is the familiar coproduct of the Hopf algebra $U=U(\gl_n(\ring))$ and $\op{res}$ is the restriction to $\bigwedge^kV\subset V^{\otimes k}$.
    \begin{prop} \label{propfermplac}
        The map $\phi: \mathcal{A}_n \to \mathcal{C}_n \otimes_\ring R[z]$ defined by
        \begin{gather}
        \phi(u_i) = \psi^*_{i + 1}\psi_i + y_{i} \psi^*_{i}\psi_{i}, \quad 0<i <n, \qquad
        \phi(u_{n}) = z \psi^*_{1}\psi_n + y_n \psi^*_n\psi_n\;,
        \end{gather}
        is an algebra morphism. The representation of $\mc{A}_n$ is not faithful, but factors through the relations
        \[
        u_i^2=y_iu_i\qquad\text{and}\qquad y_{i+1}\,u_iu_{i+1}u_i=y_i\,u_{i+1}u_iu_{i+1}\,,
        \]
        where $i=1,\ldots,n$ and indices are understood modulo $n$.
    \end{prop}

    \begin{remark}\rm
        The above representation of the affine plactic algebra was discussed in \cite[Corollary 3.25 and Remark 3.26]{gorbounovkorff2014} albeit in the context of the affine nil Coxeter and the affine nil-Hecke algebra. Indeed, setting $y_1=\cdots=y_n=0$ the above representation is a representation of the nil-Coxeter algebra which factors through the projection onto the nil Temperley-Lieb algebra. Setting $y_1=\cdots=y_n=1$ instead, one obtains a representation of the affine nil-Hecke algebra.
    \end{remark}
    \begin{proof}
        The proof consists of checking the relations of the plactic algebra using relations in the Clifford algebra. The relation \eqref{plactic_commute} is obvious, because even number of fermions with distinct indices commute. Let us check one of the relations \eqref{plactic}
        \begin{gather*}
        u_i^2 u_{i + 1} = (\psi^*_{i + 2} + y_{i + 1} \psi^*_{i + 1}) \psi_{i + 1} (\psi^*_{i + 2} + y_{i + 1} \psi^*_{i + 1}) \psi_{i + 1} (\psi^*_{i + 3} + y_{i + 2} \psi^*_{i + 2}) \psi_{i + 2} \\
        = - y_{i + 1} (\psi^*_{i + 2} + y_{i + 1} \psi^*_{i + 1})(\psi^*_{i + 3} + y_{i + 2} \psi^*_{i + 2})\psi_{i + 1} \psi_{i + 2}, \quad 0 \leq i \leq n - 3 
        \end{gather*}
        and
        \begin{gather*}
            u_i u_{i + 1} u_i = (\psi^*_{i + 2} + y_{i + 1}\psi^*_{i + 1}) \psi_{i + 1}(\psi^*_{i + 3} + y_{i + 2}\psi^*_{i + 2}) \psi_{i + 2}(\psi^*_{i + 2} + y_{i + 1}\psi^*_{i + 1}) \psi_{i + 1} \\
            = y_{i + 1}  (\psi^*_{i + 2} + y_{i + 1} \psi^*_{i + 1})(\psi^*_{i + 3} + y_{i + 2} \psi^*_{i + 2})\psi_{i + 2} \psi_{i + 1}, \quad 0 \leq i \leq n - 3, 
        \end{gather*}
        which coincide due to the anticommutation relations. The relations including the parameter $q$ are checked in the same way.
    \end{proof}
    \subsection{Chern classes in terms of the plactic algebra}
    We now identify the plactic elementary and complete symmetric polynomials \eqref{noncomelem} with multiplication operators by the Chern classes of the tautological and quotient bundles over Grassmannians.
    \begin{prop}\label{placticandChern}
        Under the isomorphism from Proposition \ref{prop:quantumSatake} the images $\phi(\bs{e}_i),\phi(\bs{h}_j)\in\End_{\ring}\bigwedge^kV[z^{\pm 1}]/\langle z + (-1)^kq\rangle$ of the plactic elementary and complete symmetric polynomials act by quantum multiplication with the Chern classes of the tautological and quotient bundle respectively,
        \begin{equation}\label{plactic2Chern}
            \Sigma^q_k \circ \phi(\bs{e}_i) = c_i(\cS_k) \circ \Sigma^q_k
            \qquad\text{and}\qquad
            \Sigma^q_k \circ \phi(\bs{h}_j) = c_j(\cQ_k) \circ \Sigma^q_k\;,
        \end{equation}
        where $1\le i\le k$, $1\le j\le n-k$ and $0<k<n$. 
    \end{prop}
    If $k=0$ or $k=n$ then we define $\phi(\bs{e}_i),\phi(\bs{h}_j)$ through the identity \eqref{plactic2Chern}. 
    \begin{proof}
        It suffices to show that
        \[
        e_r(X_1^q,\ldots,X^q_k)=\phi(\bs{e}_r)\qquad\text{and}\qquad h_r(X^q_1,\ldots,X^q_k)=\phi(\bs{h}_r)\;.
        \]
        Let us reinterpret the action of the elementary symmetric polynomials $e_r(X^q)$ in combinatorial terms. To do that, we first lift the action to $V^{\otimes k}[z^{\pm 1}]/\langle z + (-1)^kq\rangle$ and write the elementary symmetric polynomials explicitly
        \[
        e_r(X^q) = \sum\limits_{1\le i_1 < \ldots<i_r\le k} X^q_{i_1} \cdots X^q_{i_r}.
        \]
        The action of each term on a basis vector $v_{i_1}\otimes\cdots\otimes v_{i_k}$ is described as follows: we choose an $r$-subset of $\{i_1,\ldots,i_k\}$ and for each element in that set we simply multiply with $y_{i_j}$ or swap $v_{i_j}$ with $v_{i_j+1}$ (unless it is already contained in the set $\{i_1,\ldots,i_k\}$ because then its projection onto $\bigwedge^kV[z^{\pm 1}]/\langle z + (-1)^kq\rangle$ is zero). Summing up all the resulting vectors and projecting down to $\bigwedge^k V[z^{\pm 1}]/\langle z + (-1)^kq\rangle$, we obtain the result of the action of the elementary symmetric polynomial $e_r(X^q)$. 

        Let us turn to the action of the plactic polynomial $\phi(\bs{e}_r)$. Observe that due to the cyclic order in the definition  \eqref{noncomelem} the combinatorial action of $\phi(\bs{e}_r)$ is exactly the same as $e_r(X^q)$. Indeed, since the plactic elementary polynomials are ordered anti-clockwise and square-free, each basis vector $v_{i_j}$ in $v_{i_1}\wedge\ldots\wedge v_{i_k}$ is acted on by the same rule as described above, which proves the statement for the elementary symmetric polynomials. The proof for complete symmetric polynomials follows immediately from \cite[Def. 5.16, Cor. 6.9]{korffstroppel2010}, where it was proved that plactic complete polynomials can be expressed in terms of plactic elementary polynomials by the Jacobi--Trudi formula, which finishes the proof. 
    \end{proof}
    
The following relates a general factorial Schur polynomial in the quantised Chern roots to the plactic polynomials:
\begin{cor}
    Let $E(z)=\sum_{r=0}^kz^r\phi(\bs{e}_r)$ for $0<k<n$. Then we have the identity
    \begin{equation}
        E(z_1)\cdots E(z_{n-k})=\sum_{\lambda}s_{\lambda}(X_1^q,\ldots,X^q_k|y)s_{\lambda^*}(z_1,\ldots,z_{n-k}|-y)
    \end{equation}
    where the sum runs over all partitions $\lambda$ with Young diagram inside the $k\times (n-k)$ rectangle and $\lambda^*=(\lambda^\vee)'$ is obtained from $\lambda$ by first taking the complement in the rectangle and then its transpose.
\end{cor}
\begin{proof}
    This is a direct consequence of the known Cauchy type identity for factorial Schur polynomials, see e.g. \cite[Eqn 6.17]{Mac92}.
\end{proof}
\section{Exactly solvable lattice models}\label{app:C}
In this section, we briefly summarise the results of \cite{GKS20} concerning the integrable $5$-vertex lattice models. We go through the definitions of the integrable lattice models and the respective Yang--Baxter algebras, for details, we refer to \cite[Sections 1 and 2]{GKS20}. Next, we explain their connection to the quantum equivariant Schubert calculus, which is explained in \cite[Sections 3\hyphen5]{GKS20}. In the end of this section we connect to the definition of the local affine plactic algebra and deduce that the images of the Yang--Baxter algebras in the corresponding module are generated by the fermionic operators \eqref{recferm}.
\subsection{Lattice models}
\noindent Firstly, we recall the definition of the quantum $R$-matrix
\begin{equation}
    \label{R-matrix}
    R_{12}(x) = \begin{pmatrix}
        1 & 0 & 0 & 0 \\
        0 & x & 1 & 0 \\
        0 & 1 & 0 & 0 \\
        0 & 0 & 0 & 1
    \end{pmatrix} \in \End(\C^2 \otimes \C^2)[x],
\end{equation}
which is a solution to the quantum Yang--Baxter equation
\begin{equation}
\label{Yang--Baxter}
R_{12}(x_1 - x_2) R_{13}(x_1 - x_3) R_{23}(x_2 - x_3) = R_{23}(x_2 - x_3) R_{13}(x_1 - x_3) R_{12}(x_1 - x_2).
\end{equation}
Using this solution of the quantum Yang--Baxter equation, we define the bialgebra in the following way.
\begin{defn}\label{YBalgebra}
    The Yang-Baxter algebra $\YB$ associated with the quantum $R$-matrix \eqref{R-matrix} is the associative unital $\C$-algebra generated by the expansion coefficients $T^{(r)}_{ij}$ of the matrix elements of the monodromy matrix,
    \[
    T(x) = ( T_{ij}(x))_{1\le i, j \le 2} = \begin{pmatrix}
        T_{11}(x) & T_{12}(x) \\
        T_{21}(x) & T_{22}(x) 
    \end{pmatrix}, \quad T_{ij}(x) = \sum\limits_{r = 0}^\infty T^{(r)}_{ij} x^r
    \]
    subject to the $RTT$-relation
    \[
    R_{12}(x_1 - x_2) T_1(x_1) T_{2}(x_2) = T_2(x_2) T_1(x_1) R_{12}(x_1 - x_2)
    \]
    together with the coproduct given by 
    \begin{equation}
    \label{coproductYB}
    \Delta^{\YB}(T_{ij}(x)) = \sum\limits_{k = 1}^2 T_{kj}(x) \otimes T_{ik}(x).
    \end{equation}
\end{defn}
Define {\em evaluation modules} $\C^2_{y}$ for the Yang--Baxter algebra $\YB$ by mapping $T(x)\mapsto L(x,y)$ with 
\begin{gather}
\label{evaluate1}
    L(x,y) = \begin{pmatrix}
        E_{11} + (x + y) E_{22} & E_{21} \\
        E_{12} & E_{22}
    \end{pmatrix} = \begin{pmatrix}
        1 & 0 & 0 & 0 \\
        0& x+y & 1 & 0 \\
        0 & 1 & 0 & 0 \\
        0 & 0 & 0 & 1
    \end{pmatrix} \in \End(\C^2[y])[x].
\end{gather}
Employing the coproduct \eqref{coproductYB} the representation of $\YB$ given by a tensor product $U := \C^2_{y_1} \otimes \ldots \otimes \C^2_{y_n}$ of such evaluation modules is then simply
\[
T(x) \mapsto L_{0n}(x , y_n) \ldots L_{01}(x , y_1).
\]
Recall the decomposition of the module $U$ into weight subspaces,
\[
U := \C^2_{y_1} \otimes \ldots \otimes \C^2_{y_n} = \bigoplus\limits_{k = 0}^n U_{k , n}, \quad U_{k,n} = \{v \in U | \Delta^{n - 1}(E_{22}) v = k v\}.
\]
Notice that similarly to \eqref{Satake} the standard basis in $U$ can also be labelled by Grassmannian permutations: namely, if $\{\veps_0,\veps_1\}$ denotes the standard basis in $\C^2$ we label the basis vector $\veps_{\alpha_1}\otimes\cdots\otimes\veps_{\alpha_n}\in U_{k,n}$ with $w=[i_1\ldots i_n]\in S^{k,n}$ if $i_1,\ldots,i_k$ give the positions of the basis vectors $\veps_0$ and $i_{k+1},\ldots,i_n$ the positions of the basis vectors $\veps_1$. Let $\Xi_k : U_{k,n} \to H^*_{T}(\Gr(k;n))$ be the linear isomorphism which sends the standard basis of $U_{k,n}$ to the equivariant Schubert classes and define $\Xi = \bigoplus\limits_{k = 0}^n \Xi_k$.

We summarise the main result of \cite{GKS20} in the following theorem.
\begin{thm}\label{geom-lattice}
    Under the isomorphisms $\Xi_k$ the geometric action of the Yang--Baxter algebra $\YB$ is given by the following convolution operators in terms of Chern classes and push-pull operators:
    \begin{gather}
    \Xi_k \circ T_{11}(x)|_{U_{k,n}} = x^k c_{x^{-1}}(\cS_k) \circ \Xi_k, \\
    \Xi_{k + 1} \circ T_{12}(x)|_{U_{k,n}} = x^k b_k \circ c_{x^{-1}}(\cS_k) \circ \Xi_k, \\
    \Xi_{k - 1} \circ T_{21}(x)|_{U_{k,n}} = x^{k - 1}c_{x^{-1}}(\cS_{k - 1}) \circ c_{k-1} \circ \Xi_k, \\
    \Xi_{k} \circ T_{22}(x)|_{U_{k,n}} = x^{k - 1} b_{k - 1} \circ c_{x^{-1}}(\cS_{k - 1}) \circ c_{k-1} \circ \Xi_k,
    \end{gather}
    where $c_x(\cS_k)$ is the Chern polynomial of the tautological bundle and $c_k=(p_1)_* p_2^*$ and $b_k=(p_2)_* p_1^*$ with the maps $p_1,p_2$ from \eqref{pushpull}. This result extends to equivariant quantum cohomology when replacing the operator $T_{11}(x)$ by its quantum deformation $T_{11}(x) + q T_{22}(x)$.
\end{thm}

\begin{remark}\rm
    We recall from \cite{gorbounovkorff2014,GKS20} that there is a second Yang-Baxter algebra $\YB'$ and a corresponding lattice model, called the {\em vicious walker model} in \cite{korff2014quantum}, which describes the multiplication by Chern classes of the quotient bundle $\cQ_{n-k}$ instead. Denote the corresponding monodromy matrix by $T'(x)$. Then the action of both algebras $\YB$ and $\YB'$ are related via level-rank duality \eqref{Theta}, $\Theta\circ T_{ij}(x)=T'_{ji}(x)\circ\Theta$. Therefore, we shall omit the details for this `dual model'. 
\end{remark}

\subsection{Plactic algebra and fermionic operators}
In this subsection, we comment on the relation between the action of the  Yang--Baxter algebras $\YB$ and $\YB'$ from \cite{GKS20} and the fermionic operators \eqref{CAR} in this article. 

The geometric nature of both actions has already been explained, the one for $\YB,\YB'$ in \cite[Theorem 7.3]{GKS20} and the one for the Clifford algebra in \eqref{fermgeom}, \eqref{recferm} of this article. In order to directly relate both actions let us denote by $\mathcal{Y}\subset \End_{\ring}(U)$ and $\mathcal{Y}' \subset \End_{\ring}(U)$ the subalgebras generated by the Yang--Baxter algebras $\YB$ and $\YB'$ localised at the images of the operators $T_{11}(x)$ and $T_{11}^{'}(x)$, respectively. 
\begin{prop}\label{fermtheor1}
    The operators $\gamma^*_j = \Xi^{-1} \circ \Sigma \circ \psi^*_j$ and $\gamma_j = \Xi^{-1} \circ \Sigma \circ \psi_j$ for $j=1,\ldots,n$ belong to $\mathcal{Y}\cap\mathcal{Y}'$.
\end{prop}
We can refine the last statement as follows: because the images $\gamma_j,\gamma_j^*$ of the fermionic operators lie in $\mathcal{Y}\cap\mathcal{Y}'$, the algebra $\mc{A}\subset\End_\ring(U)$ generated by them over the base ring $\ring= H^*_T(pt)[q^{\pm 1}]$ also lies in the intersection $\mathcal{Y}\cap\mathcal{Y}'$. Recall the definition \eqref{defplactic} of the plactic algebra and the plactic elementary and complete symmetric polynomials \eqref{noncomelem}. Due to Proposition \ref{propfermplac}, the image of the plactic algebra and, therefore, the images of the plactic polynomials in $\End_\ring(\bigwedge V)[z^{\pm 1}]/\langle z+ (-1)^k q \rangle$ also belong to the subalgebra $\mc{A}$. 

In terms of the Yang-Baxter algebras $\YB,\YB'$, the plactic elementary and complete symmetric polynomials correspond to expansion coefficients of the operators $T_{11}(x)+q T_{22}(x)$ and $T_{11}^{'}(x) + q T_{22}^{'}(x)$ respectively, which geometrically correspond to quantum multiplication by Chern classes of the tautological and quotient bundles. 

Finally, we observe that in the fermionic construction of the plactic algebra from Proposition \ref{propfermplac} the variable $z$ should be specified to $(-1)^{k - 1} q$, where the sign factor $(-1)^{k - 1}$ should be interpreted as the operator\footnote{N.B. the exponent of the particle number operators can, in principle, be written as a polynomial in the particle number operator} 
\[
(-1)^k = \exp\left(\sqrt{-1}\,\pi\sum\limits_{i = 1}^n \gamma^*_i \gamma_i\right)\;.
\]
Moreover, the projector onto any fixed weight subspace $U_{k,n}$ can also be written in terms this fermionic particle number operator $\sum_{i=1}^n\gamma_i^*\gamma_i$. 
Thus, we obtain the following refinement of Proposition \ref{fermtheor1}:
\begin{thm} \label{fermtheor2}
    The images $\mathcal{Y}$ and $\mathcal{Y}'$ of the Yang--Baxter algebras are both equal to the subalgebra $\mc{A}\subset\End_\ring(U)$ generated by the fermionic operators $\gamma_j$ and $\gamma^*_j$. Hence, they must be equal, $\mathcal{Y} = \mathcal{Y}'$.
\end{thm}
\noindent Geometrically, this statement can be reformulated as follows: quantum multiplication by any Schubert class, and thus any element of the equivariant quantum cohomology ring $qH^*_T(\Gr(k;n))$, can be performed solely by making use of the fermionic operators \eqref{recferm}. 
\begin{remark}\rm
    As explained in \cite[Proposition 7.2]{GKS20}, the operators $\gamma_j$ and $\gamma^*_j$ give two representations of the cohomological Hall algebra(COHA) of the quiver $A_1$. Theorem \ref{fermtheor2} states that the images of these two COHA representations generate the (localised) images $\mathcal{Y}$ and $\mathcal{Y}'$ of the corresponding quantum groups $\YB$ and $\YB'$.
\end{remark}
We conclude by expressing the generators of the Yang--Baxter algebra $\YB$ explicitly in terms of creation and annihilation operators using the elementary plactic polynomials. The formulae for the dual Yang--Baxter algebra $\YB'$ generators are obtained in a similar same way using level-rank duality.
\begin{prop}
    The generators of the Yang--Baxter algebra $\YB$ can be written in terms of fermionic operators as
    \begin{equation*}
    \begin{array}{l}
    T_{11}(x) \circ \Xi_k^{-1} \circ \Sigma_k^q = \Xi_k^{-1} \circ \Sigma^q_k \circ \sum\limits_{i = 0}^k x^{k - i} \phi(\bs{e}_i), \\
    T_{21}(x) \circ \Xi_k^{-1} \circ \Sigma_k^q =(-1)^{k - 1}  \Xi_{k-1}^{-1} \circ \Sigma^q_{k-1} \circ \sum\limits_{i = 0}^k x^{k - 1 - i} \phi(\bs{e}_i) \circ \psi_n, \\
    T_{12}(x) \circ \Xi_k^{-1} \circ \Sigma_k^q  = \Xi_{k+1}^{-1} \circ \Sigma^q_{k+1} \circ \psi^*_1\circ \sum\limits_{i = 0}^k x^{k - i} \phi(\bs{e}_i) , \\
    T_{22}(x) \circ \Xi_k^{-1} \circ \Sigma_k^q = (-1)^{k - 1} \circ \Xi_k^{-1} \circ \Sigma_k^q \circ \sum\limits_{i = 0}^{k-1} x^{k - 1 - i} \psi^*_1 \circ \phi(\bs{e}_i) \circ \psi_n . 
    \end{array}
    \end{equation*}
\end{prop}
\begin{proof}
    The formula for $T_{11}(x)$ follows from Proposition \ref{placticandChern}. The formulae for $T_{12}(x)$ and $T_{21}(x)$ generators are special cases of Proposition \ref{fermionform}. Finally, the formula for $T_{22}(x)$ follows from Theorem \ref{geom-lattice}.
\end{proof}
\section{Another proof of Level-Rank Duality}

We sketch an alternative proof of Theorem \ref{thm:LRduality} using the quantum Chevalley formula. 

    Since our construction rests on the Satake correspondence let us first establish the duality for the case $k=1$, i.e. $qH^*_T(\Gr(1;n))\simeq qH^*_T(\Gr(n-1;n))$ or $(V,\star)\simeq(\bigwedge^{n-1}V,\star)$. Under the Satake map \eqref{Satake} the Schubert basis in $qH^*_T(\Gr(1;n))$ corresponds to the standard basis in $V[q^{\pm1}]=\ring^n[q^{\pm 1}]$, $v_k=\psi^*_k(1)\mapsto\sigma_{(k-1)}$, and the Schubert basis in $qH^*_T(\Gr(n-1;n))$ to $v^k=(-1)^{k-1}\psi_k(v_1\wedge\ldots\wedge v_n)\mapsto \sigma_{(1)^{n-k}}$ with $k=1,\ldots,n$. We need to show that the map $\Theta_1:v_k\mapsto v^{n+1-k}$ gives a ring isomorphism up to the stated transformation of the coefficients in $\ring$. Using the same argument from \cite[Cor. 7.1]{MihalceaAIM} as previously, it suffices to show that the quantum Chevalley rule holds in both rings, because under the stated isomorphism $\Theta_1:\sigma_{(1)}\mapsto\sigma_{(1)}$. We have already done this for $(V,\star)$, so it suffices to compute
    \begin{multline*}
        \Theta(v_2)\star\Theta(v_{n+1-k})=v^{n-1}\star v^{k}=
    (X^q_1+\cdots+X^q_{n-1}-y_1\cdot1-\cdots-y_{n-1}\cdot1)v^{k}\\
    =(X^q_1-y_1\cdot 1+\cdots+X^q_{n-1}-y_{n-1}\cdot1)v_1\wedge\ldots\wedge v_{k-1}\wedge v_{k+1}\wedge\ldots\wedge v_n
    \end{multline*}
    for all $k=1,\ldots,n$. One easily verifies that
    \[
    (X^q_i-y_i\cdot1)v^k=\left\{
    \begin{array}{ll}
         0,&i<k-1  \\
         v^{k-1},& i=k-1\\
         (y_{i+1}-y_i)v^k,&i>k-1
    \end{array}
    \right.\;,\qquad k=2,\ldots,n
    \]
    and
    \[
    (X^q_{i}-y_{i}\cdot1)v^1=\left\{
    \begin{array}{ll}
         (y_{i+1}-y_i)v^1,&i<n-1  \\
         qv^{n}+(y_n-y_{n-1})v^1,& i=n-1
    \end{array}
    \right.\;.
    \]
    Hence, we arrive at
    \begin{eqnarray}\label{LR1}
    \Theta(v_2)\star\Theta(v_{n+1-k})&=&v^{n-1}\star v^{k}=\left\{
    \begin{array}{ll}
         v^{k-1}+(y_{n}-y_k)v^k,&k>1  \\
         qv^{n}+(y_n-y_{1})v^1,& k=1
    \end{array}
    \right.\\
    &=&\Theta(v_2\star v_{n+1-k})\notag
    \end{eqnarray}
    provided we impose the stated `anti-linearity' property of the map which sends $f=(y_n-y_k)\in\ring$ to $f'=(y_{n+1-k}-y_1)\in\ring$. As the map $\Theta_1$ is obviously degree preserving this establishes the `anti-linear' ring isomorphism $qH^*_T(\Gr(1;n))\simeq qH^*_T(\Gr(n-1;n))$ or $(V,\star)\simeq(\bigwedge^{n-1}V,\star)$.

In light of the `dual' quantum Chevalley rule \eqref{LR1} we can now repeat our previous construction step-by-step but starting from $qH^*_T(\Gr(n-1;n))\simeq \bigwedge^{n-1}V\simeq V^*$ instead and use 
Poincar\'e duality to describe the ring structure on $(\bigwedge^{n-k}V,\star)$. Namely, we have
\[
\ket{\lambda}\star\ket{w}=\Theta (s_{\lambda'}(\tilde X^q_1,\ldots\tilde X^q_{k};y)v^{n+1-w_n}\wedge\ldots\wedge v^{n+1-w_{n-k+1}}),\qquad \tilde X^q=(X^q)^T\;.
\]
The last formula shows that under level-rank duality we obtain an equivalent description in terms of the transpose of the matrix \eqref{Xmatrix}, which corresponds to swapping the Borel subalgebra in our conventions when defining the quantum cohomology of Grassmannians.    

\printbibliography
\end{document}